\documentclass{article}
\usepackage{amsmath,amsfonts,amsthm,graphics,color}
\usepackage{accents} % Used only for undertilde

\allowdisplaybreaks  %see end of Section 7.

\newtheorem{theorem}{Theorem}[section]
\newtheorem{lemma}[theorem]{Lemma}

\theoremstyle{definition}
\newtheorem{example}[theorem]{Example}
\newtheorem{num_definition}[theorem]{Definition}

\newtheorem{remark}[theorem]{Remark}
\newtheorem{question}{Question}

\newtheorem*{ack}{Acknowledgement}

\newenvironment{romenumerate}{\begin{enumerate}% gives (i), (ii) etc.
 }{\end{enumerate}}

\newcounter{thmenumerate}
\newenvironment{thmenumerate}
{\setcounter{thmenumerate}{0}%
 \def\item{\par% \ifnum\thethmenumerate=0\else\par\fi %\noindent\fi
 \refstepcounter{thmenumerate}\textup{(\roman{thmenumerate})\enspace}}
}

\newcommand{\aex}{a.e.\spacefactor=1000}
\newcommand\noproof{\qed}   % for theorem without proof

\newcommand\Phik{\Phi_{\kf}}
\newcommand\Sk{S_{\kf}}
\newcommand\Pk{\Phi_{\kf}}
\newcommand\Tk{T_{\kae}}

\newcommand\bp{{\mathfrak X}}

\newcommand\RR{{\mathbb R}}
\newcommand\eps{\varepsilon}
\newcommand\sss{{\mathcal S}}
\newcommand\dd{\,d}
\newcommand\norm[1]{\ensuremath{\|#1\|}}
\newcommand\tn[1]{\ensuremath{\|#1\|_2}}
\newcommand\bb[1]{\bigl(#1\bigr)}
\newcommand\bm[1]{\bigl|#1\bigr|}

\newcommand\downto{\searrow}
\newcommand\upto{\nearrow}

\newcommand{\refT}[1]{Theorem~\ref{#1}}

\newcommand{\refL}[1]{Lemma~\ref{#1}}
\newcommand{\refS}[1]{Section~\ref{#1}}
\newcommand{\refSS}[1]{Subsection~\ref{#1}}
\newcommand{\refD}[1]{Definition~\ref{#1}}
\newcommand{\refE}[1]{Example~\ref{#1}}
\newcommand{\refR}[1]{Remark~\ref{#1}}
\newcommand{\refQ}[1]{Question~\ref{#1}}

\newcommand\E{{\mathop{\mathbb E{}}\nolimits}}
\renewcommand\Pr{{\mathop{\mathbb P{}}\nolimits}}
\newcommand\ka{\kappa}
\newcommand\la{\lambda}
\newcommand\gam{\gamma}
\newcommand\Ga{\Gamma}
\newcommand\F{{\mathcal F}}
\newcommand\G{{\mathcal G}}
\newcommand\vx{{\bf x}}
\newcommand\bpk{{\mathfrak X}_{\kf}}

\newcommand\kae{{\ka_{\mathrm e}}}
\newcommand\kaep{{\ka_{\mathrm e}'}}
\newcommand\kf{{\undertilde{\ka}}}
\newcommand\kfn{{\undertilde{\ka}_n}}
\newcommand\bkf{\bar{\kf}}
\newcommand\kfp{\kf'}
\newcommand\kfpl{\kf^+}
\newcommand\kfM{{\kf^M}}
\newcommand\kaM{\ka^M}
\newcommand\kaeM{{\ka_{\mathrm e}^M}}
\newcommand\rhox[1]{\ensuremath{\rho_{#1}}}

\newcommand\rhogek{\rhox{\ge k}}
\newcommand\kapl{\ka^+}
\newcommand\pto{\overset{\mathrm{p}}{\to}}
\newcommand\dto{\overset{\mathrm{d}}{\to}}
\newcommand\Po{\operatorname{Po}}
\newcommand\Aut{\operatorname{Aut}}
\newcommand\aut{\operatorname{aut}}
\newcommand\emb{\operatorname{emb}}
\newcommand\cc{{\mathrm{c}}}
\newcommand\dcut{{\delta_{\square}}}
\newcommand\dloc{d_{\mathrm{loc}}}
\newcommand\set[1]{\ensuremath{\{#1\}}}

\newcommand\bigpar[1]{\bigl(#1\bigr)}
\newcommand\Bigpar[1]{\Bigl(#1\Bigr)}

\newcommand\bigabs[1]{\bigl|#1\bigr|}
\newcommand\Bigabs[1]{\Bigl|#1\Bigr|}

\newcommand\oi{[0,1]}
\newcommand\kx{r} %??
\newcommand\kF{|F|} %??
\newcommand\Symmr{\mathfrak S_r}
\newcommand\gnk{\ensuremath{G(n,\kf)}}
\newcommand\Mn{A_n}
\newcommand\Mij{a_{ij}}
\newcommand\kaxrij{a_{r,i,j}}% like this ok?
\newcommand\kaxMrij{a^{(M)}_{r,i,j}}
\newcommand\taux{A}
\newcommand\tauxij{a_{ij}}
\newcommand\N{B}
\newcommand\Nij{b_{ij}}
\newcommand\BN{{\overline\N}}%we'd like something between bar and overline here!
\newcommand\BNij{{\overline b}_{ij}}
\newcommand\BNii{{\overline b}_{ii}}
\newcommand\Btauxii{{\overline a}_{ii}}
\newcommand\Btauxij{{\overline a}_{ij}}
\newcommand\tauxpij{a'_{ij}}
\newcommand\tauxp{A'}
\newcommand\Btaux{{\overline \taux}}%we'd like something between bar and overline here!
\newcommand\Btau{{\overline\tau}}
\newcommand\tauxM{\taux^{(M)}}% notation clash!
\newcommand\tauxMij{a^{(M)}_{ij}}
\newcommand\tauM{\tau^M}
\newcommand\fall[1]{_{(#1)}}
\newcommand\cn[1]{\|#1\|_\square}
\newcommand\kaphi{\ka^{(\varphi)}}
\newcommand\betamax{\beta_{\mathrm{max}}}

\newcommand\xie{\xi}
\newcommand\xiek{\xie(\kf)}
\newcommand\xb{\tilde{\xi}}
\newcommand\ntoo{\ensuremath{{n\to\infty}}}
\newcommand\Mtoo{\ensuremath{{M\to\infty}}}
\newcommand\cA{\mathcal A}
\newcommand\cG{\mathcal G}

\newcommand\cQ{\mathcal Q}
\newcommand\tD{\widetilde D}
\newcommand\bbN{\mathbb N}
\newcommand\bbR{\mathbb R}
\newcommand\ga{\alpha}
\newcommand\gd{\delta}
\newcommand\gD{\Delta}

\newcommand\La{\Lambda}
\newcommand\Var{\operatorname{Var}}
\newcommand\Cov{\operatorname{Cov}}
\newcommand\Be{\operatorname{Be}}
\newcommand\Bi{\operatorname{Bi}}
\newcommand\CPo{\operatorname{CPo}}
\newcommand\MCPo{\operatorname{MCPo}}
\newcommand\ett[1]{\boldsymbol1[#1]} 
\renewcommand\ll{\lambda}
\newcommand\LL{\Lambda}
\newcommand\cpol{\CPo(\ll)}
\newcommand\cpolx{\CPo(\ll_x)}
\newcommand\mcpol{\MCPo(\LL)}
\newcommand\sumj{\sum_{j=1}^\infty}
\newcommand\sumjf{\sum_{j\in V(F)}}
\newcommand\sumfj{\sum_{F,j}}
\newcommand\suma{\sum_{\ga}}
\newcommand\prodj{\prod_{j=1}^\infty}
\newcommand\pgf{probability generating function}
\newcommand\gfcpol{\varphi_{\cpol}}
\newcommand\gfcpolx{\varphi_{\cpolx}}
\newcommand\gfmcpol{\varphi_{\mcpol}}
\newcommand\ints{\int_\sss}
\newcommand\lfj{\la_{F,j}}
\newcommand\lxfj{\lfj(x)}
\newcommand\lxd{\ll_x\set d}
\newcommand\dfj{{d_F(j)}}
\newcommand\dn{D_n}
\newcommand\nl{n_\ell}
\newcommand\nfj{N_{F,j}}
\newcommand\dq{D'}

\newcommand\qqw{^{-1/2}}
\newcommand\qw{^{-1}}
\newcommand\qww{^{-2}}
\newcommand\dtv{d_{\mathrm{TV}}}
\newcommand\pfjax{p_{F,j,\ga}(\vx)}
\newcommand\fj{_{F,j}}
\newcommand\hlfj{\widehat\la\fj}
\newcommand\hlfjx{\hlfj(\vx)}
\newcommand\hX{\widehat X}
\newcommand\hLa{\widehat \ll}
\newcommand\CS{Cauchy--Schwarz}
\newcommand\CSineq{\CS{} inequality}
\newcommand\M{^M}
\newcommand\nfm{N_F^M}
\newcommand\gnkm{\ensuremath{G(n,\kfM)}}
\newcommand\lxdm{\ll_x\M\set d}
\newcommand\mcpolm{\MCPo(\LL\M)}
\newcommand\intsf{\int_{\sss^{\kF}}}
\newcommand\sumF{\sum_{F\in\F}}
\newcommand\Gt{{\widetilde G}}
\newcommand\Op{O_{\mathrm{p}}}
\newcommand\op{o_{\mathrm{p}}}
\newcommand\opn{\op(n)}
\newcommand\tH{{\widetilde H}}
\newcommand\nr{n_{\mathrm r}}
\newcommand\nd{n_{\mathrm d}}
\newcommand\nx{n_{\mathrm x}}
\newcommand\gbx{\beta}
\newcommand\gb{\gbx_1}
\newcommand\gbb{\gbx_2}
\newcommand\gbbb{\gbx_3}
\newcommand\ha{{\hat a}}
\newcommand\Gp{G^{\langle p \rangle}}%bond percolation
\newcommand\Gpp{G^{[p]}}%site percolation
\newcommand\isom{\cong}
\newcommand\Fp{F^{\langle p \rangle}}
\newcommand\kap{\ka^{\langle p \rangle}}
\newcommand\kfq{\kf^{\langle p \rangle}}
\newcommand\Gr{\G^{\mathrm r}}
\newcommand\Grt{\G^{\mathrm r}_t}
\newcommand\tone{\ts_1}
\newcommand\ts{{\tilde t}}% $\traw/\aut(F)$
\newcommand\traw{t}% symbol for subgraph density in kernel
\newcommand\ii[1]{\int #1}% integral of a kernel family
\newcommand\iid{i.i.d.}

\begin{document}
\title{Sparse random graphs with clustering}

\author{B\'ela Bollob\'as\thanks{Department of Mathematical Sciences,
University of Memphis, Memphis TN 38152, USA}
\thanks{Trinity College, Cambridge CB2 1TQ, UK}
\thanks{Research supported in part by NSF grants CNS-0721983, CCF-0728928
and DMS-0906634, and ARO grant W911NF-06-1-0076}
\and Svante Janson%
\thanks{Department of Mathematics, Uppsala University,
 PO Box 480, SE-751 06 Uppsala, Sweden}
\and Oliver Riordan%
\thanks{Mathematical Institute, University of Oxford, 24--29 St Giles', Oxford OX1 3LB, UK}}
\date{October 21, 2009}
\maketitle

\begin{abstract}
In 2007 we introduced a general model of sparse random graphs with
(conditional) independence between the edges. The aim of this paper is to present an
extension of this model in which the edges are far from independent,
and to prove several results about this extension.
The basic idea is to construct the random graph by adding
not only edges but also other small graphs. In other words,
we first construct an inhomogeneous random hypergraph with (conditionally)
independent hyperedges,
and then replace each hyperedge by a (perhaps complete) graph.
Although flexible enough to produce graphs with significant dependence between 
edges, this model is nonetheless mathematically tractable. Indeed,
we find the critical point where a giant component emerges in full generality,
in terms of the norm of a certain integral operator, and relate the size
of the giant component to the survival probability of a certain
(non-Poisson) multi-type branching process.
While our main focus is the phase transition, we also study the degree
distribution and the numbers of small subgraphs. We illustrate the model
with a simple special case that produces graphs with power-law degree sequences
with a wide range of degree exponents and clustering coefficients.
\end{abstract}

\section{Introduction and results}\label{sec_ir}

In~\cite{BJR}, a very general model for sparse random graphs
was introduced, corresponding to an inhomogeneous version of $G(n,c/n)$,
and many properties of this model were determined, in particular,
the critical point of the phase transition where the giant component
emerges.
Part of the motivation was to unify many of the new random graph
models introduced as approximations to real-world networks.
Indeed, the model of~\cite{BJR} includes many of these models as exact special
cases, as well as the `mean-field' simplified versions of many of the more
complicated models. (The original forms are frequently too complex 
for rigorous mathematical analysis, so such mean-field versions
are often studied instead.)
Unfortunately, there are many models with key features that are
not captured by their mean-field versions, and hence not by the
model of~\cite{BJR}. The main problem
is that many real-world networks exhibit {\em clustering}: for example,
while there are $n$ vertices and only $5n$ edges, there may be $10n$ 
triangles, say. In contrast, the model of~\cite{BJR}, like
$G(n,c/n)$, produces graphs that contain essentially no triangles or short cycles.

Most models introduced to approximate particular real-world networks
turn out to be mathematically intractable, due to the dependence
between edges. Nevertheless, many such models have been studied;
as this is not our main focus, let us just list a few examples of early work
in this field. One of the starting points in this area was the
(homogeneous) `small-world' model of Watts and Strogatz~\cite{WS}. Another
was the observation of power-law degree sequences in various networks
by Faloutsos, Faloutsos and Faloutsos~\cite{FFF}, among others. 
Of the new inhomogeneous models,  perhaps the most studied
is the `growth with preferential attachment' model introduced in an imprecise
form by Barab\'asi and Albert~\cite{BAsc}, later made precise
as the `LCD model' by Bollob\'as and Riordan~\cite{diam}. Another is the `copying'
model of Kumar, Raghavan, Rajagopalan, Sivakumar, Tomkins and Upfal~\cite{Upfal},
generalized by Cooper and Frieze~\cite{CF}, among others.
For (early) surveys of work in this field see, for example, Barab\'asi and Albert~\cite{BAsurv},
Dorogovtsev and Mendes~\cite{DMsurv}, or Bollob\'as and Riordan~\cite{BRsurv}.

Roughly speaking, any sparse model with clustering
must include significant dependence between edges, so one might expect
it to be impossible to construct a general model of this type that
is still mathematically tractable. However, it turns out that one can do this.
The model that we shall define is essentially a generalization of that
in~\cite{BJR}, although we shall handle certain technicalities
in a different way here.

Throughout this paper we use standard graph theoretic notation as in~\cite{Bollobas:MGT}.
For example, if $G$ is a graph then $V(G)$ denotes its vertex set,
$E(G)$ its edge set, $|G|$ the number of vertices, and $e(G)$ the number of edges.
We also use standard notation for probabilistic asymptotics as in \cite{JLR}:
a sequence ${\mathcal E}_n$ of events holds {\em with high probability}, or {\em whp},
if $\Pr({\mathcal E}_n)\to 1$ as $n\to\infty$. If $(X_n)$ is a sequence of random variables
and $f$ is a deterministic function, then $X_n=\op(f(n))$ means $X_n/f(n)\pto 0$,
where $\pto$ denotes convergence in probability.

\subsection{The model}
Let us set the scene for our model. By a {\em type space} we simply mean a
probability space $(\sss,\mu)$.
Often, we shall take $\sss=[0,1]$ or $(0,1]$ with $\mu$ Lebesgue measure.
Sometimes we consider $\sss$ finite. As will become clear,
any model with $\sss$
finite can 
be realized as a model with type space $[0,1]$, but sometimes the notation
will be simpler with $\sss$ finite.
More generally, as shown in~\cite{SJ210},
every instance of the random
graph model we are going to describe can be realized 
as an equivalent model with type space $[0,1]$. 
Hence, when it comes to proofs, 
we lose no generality by taking $\sss=[0,1]$, but we usually
prefer allowing an arbitrary type space, which is more flexible for
applications. For example, as with the model in~\cite{BJR},
type spaces such as $\sss=[0,1]^2$
are likely to be useful for geometric applications, as in~\cite{kerapp}.

Let $\F$ consist of one representative of each isomorphism class of
finite connected graphs, chosen so that if $F\in \F$ has $r$
vertices then $V(F)=[\kx]=\{1,2,\ldots,\kx\}$.  
Given $F\in \F$ with $\kx$ vertices,
let $\ka_F$ be a measurable function from $\sss^\kx$ to $[0,\infty)$;
we call $\ka_F$ the {\em kernel} corresponding to $F$. 
A sequence $\kf=(\ka_F)_{F\in\F}$ is a {\em kernel family}. 
In our results we shall impose an additional integrability
condition on $\kf$, but this is not needed to define the model.

Let $\kf$ be a kernel family and $n$ an integer;
we shall define a random graph $G(n,\kf)$ with vertex set
$[n]=\set{1,2,\ldots,n}$.
First let $x_1,x_2,\ldots,x_n\in \sss$ be \iid\ (independent
and identically distributed) with the distribution $\mu$.
Given $\vx=(x_1,\ldots,x_n)$, construct $G(n,\kf)$ as follows, starting
with the empty graph. For each $\kx$ and each $F\in \F$ with $|F|=\kx$,
and for every $\kx$-tuple of distinct vertices
$(v_1,\ldots,v_\kx)\in [n]^\kx$, add a copy
of $F$ on the vertices $v_1,\ldots,v_\kx$ (with vertex $i$ of $F$
mapped to $v_i$) with probability 
\begin{equation}\label{pdef}
 p=
p(v_1,\dots,v_\kx;F)=
\frac{\ka_F(x_{v_1},\ldots,x_{v_\kx})}{ n^{\kx-1}},
\end{equation}
all these choices being independent. If $p>1$, then we simply add a copy
with probability $1$. 
We shall often call the added copies of the various $F$ that
together form \gnk{} {\em atoms} since, in our construction of \gnk,
they may be viewed as indivisible building blocks. Sometimes we
refer to them as
\emph{small graphs}, although there is in general no bound on their sizes.
Usually we think of $G(n,\kf)$ as a simple graph, in which case
we simply replace any multiple edges by single edges. Typically there will be
very few multiple edges, so this makes little difference.

Note that we assume that the atoms of \gnk{} are connected.
The extension to the case where some atoms
may be disconnected is discussed in
\refS{Sdisconnected}. 

The reason for dividing by $n^{\kx-1}$ in \eqref{pdef}
is that we wish to consider sparse graphs; indeed,
our main interest is the case when $G(n,\kf)$ has $O(n)$ edges. As it turns out,
we can be slightly more general; however, when $\kappa_F$
is integrable (which we shall always assume), the expected number
of added copies of each graph $F$ is $O(n$).
Note that all incompletely specified integrals are with respect
to the appropriate $r$-fold product measure $\mu^r$ on $\sss^r$.

\begin{remark}
There are several plausible choices for the normalization
in \eqref{pdef}. 
The one we have chosen means
that if $\ka_F=c$ is constant, then
(asymptotically) 
there are
on average $cn$ copies of $F$ in total, and
each vertex is on average in 
$\kx c$
copies of $F$. An alternative is to divide the expression in \eqref{pdef} by $\kx$;
then (asymptotically) 
each vertex would on average 
be in $c$ copies of $F$. 
Another alternative, natural when adding cliques only but less so in the
general case, would be to divide by $\kx!$; this is equivalent to
considering unordered sets of $\kx$ vertices instead of ordered $\kx$-tuples.
When there is only one kernel, corresponding to adding edges,
this would correspond to the normalization used in
\cite{BJR}, and in particular to that of the classical model $G(n,c/n)$;
the normalization we use here differs from this by a factor of 2.
Yet another normalization would be to divide by $\aut(F)$, the
number of automorphisms of $F$; this is equivalent to considering the
distinct copies of $F$ in $K_n$, which is natural but 
leads to extra factors $\aut(F)$ in many formulae, and
we do not find that the advantages outweigh the disadvantages.
\end{remark}

As in~\cite{BJR}, there are several minor variants of $G(n,\kf)$;
perhaps the most important is the {\em Poisson multi-graph} version of $G(n,\kf)$.
In this variant,
for each $F$ and each $\kx$-tuple, we add a Poisson $\Po(p)$
number of copies of $F$ with this vertex set, 
where $p$ is given by \eqref{pdef}, and we keep multiple edges.

Alternatively, we
could add a Poisson number of copies and delete multiple edges, which
is the same as adding one copy with probability $1-e^{-p}$ and no copy
otherwise.
More generally, we could add one copy of $F$ with probability $p+o(p)$,
and two or more copies with probability $o(p)$. As long as the error
terms are uniform over graphs $F$ and $\kx$-tuples $(v_1,\ldots,v_\kx)$,
all our results will apply in this greater generality. Since this
will follow by simple sandwiching arguments (after reducing
to the `bounded' case; see \refD{Dbounded}), we shall consider
whichever form of the model is most convenient; usually
this turns out to be the Poisson multi-graph form.

\begin{remark}
Under certain mild conditions, the results of \cite{SJ212}
imply a strong form of asymptotic
equivalence between the various versions of the model.
For example,
if we add copies of $F$ with probability
$p+O(p^2)$, where the implied constant is uniform over $F$ and $(v_1,\dots,v_\kx)$, and
\begin{equation}\label{ea}
 \E\sum_F \sum_{v_1,\dots,v_{|F|}} p(v_1,\dots,v_{|F|};F)^3 =o(1),
\end{equation}
then the resulting model is equivalent to that with probability $p$,
in that the two random graphs can be coupled to agree whp;
this is a straightforward modification of  \cite[Corollary 2.13(i)]{SJ212}.
%\marginal{With the same proof from \cite[Theorem 2.9]{SJ212}. SJ}
Extending the argument in \cite[Example 3.2]{SJ212}, it can be shown
that \eqref{ea} holds if
\begin{equation*}
 \sum_{F\in\F} \int_{\sss^{|F|}} \ka_F^{\kF/(\kF-1)} <\infty.
\end{equation*}
This certainly holds for the bounded kernel families (see~\refD{Dbounded})
that we consider in most of our proofs, although \eqref{ea} is easy
to verify directly for such kernel families.
\end{remark}

In the special case where all $\ka_F$ are zero apart from $\ka_{K_2}$, the kernel
corresponding to an edge, we recover (essentially) a special case of the model
of~\cite{BJR}; we call this the {\em edge-only} case, since we add only edges,
not larger graphs.
We write $\ka_2$ for $\ka_{K_2}$.
Note that in the edge-only case, given $\vx$, two vertices $i$ and $j$
are joined with probability
\begin{equation}\label{2sym}
 \frac{\ka_2(x_i,x_j)+\ka_2(x_j,x_i)}{n}
  +O\left(\frac{(\ka_2(x_i,x_j)+\ka_2(x_j,x_i))^2}{n^2}\right).
\end{equation}
The correction
term will never matter, so we may as well replace $\ka_2$ by its
symmetrized 
version. In fact, we shall always assume that $\ka_F$ is invariant under
the action of the automorphism group $\Aut(F)$ of the graph $F$.
In other words, if $\phi:[r]\to[r]$ is a permutation such that $\phi(i)\phi(j)\in E(F)$
if and only if $ij\in E(F)$, then we assume that
$\ka_F(\phi(x_1),\ldots,\phi(x_r))=\ka_F(x_1,\ldots,x_r)$
for all $x_1,\ldots,x_r\in\sss$.
In the Poisson version, or if we add copies of graphs $F$ with
probability $1-e^{-p}$, 
the correction terms in \eqref{2sym} and its generalizations
disappear:
in the edge-only case, given $\vx$, vertices $i$ and $j$
are joined with probability
$1-\exp\bigpar{-(\ka_2(x_i,x_j)+\ka_2(x_j,x_i))/n}$, and in general
we obtain exactly the same random graph if we
symmetrize each $\ka_F$ with respect to $\Aut(F)$.

For any kernel family $\kf$, let $\kae$ be the corresponding
{\em edge kernel}, defined by
\begin{equation}\label{kaedef}
 \kae(x,y) = \sum_F \sum_{ij\in E(F)}
\int_{\sss^{V(F)\setminus \{i,j\}}}
  \ka_F(x_1,\ldots,x_{i-1},x,x_{i+1},\ldots,x_{j-1},y,x_{j+1},\ldots,x_{|F|}),
\end{equation}
where the second sum runs over all $2e(F)$ ordered pairs $(i,j)$ with
$ij\in E(F)$,
and we integrate over all variables apart from $x$ and $y$.
Note that the sum need not always converge; since every term is positive
this causes no problems: we simply allow $\kae(x,y)=\infty$
for some $x,y$. 
Given $x_i$ and $x_j$,
the probability that $i$ and $j$ are joined
in $G(n,\kf)$ is at most $\kae(x_i,x_j)/n$, and this upper bound is
typically quite sharp.
%\marginal{NB - here not exact for Poisson!}
For example, if $\kf$ is bounded in the sense of
\refD{Dbounded} below, then the probability is 
$\kae(x_i,x_j)/n+O(1/n^2)$. 
In other words,
$\kae$ captures the edge probabilities in $G(n,\kf)$, but not the correlations.

Before proceeding to deeper properties, let us note that the expected
number of added copies of $F$ is $(1+O(n\qw))n\int_{\sss^{|F|}} \ka_F$.
Unsurprisingly, the actual number turns out to be
concentrated about this mean. 
Let 
\begin{equation*}%\label{xie}
  \xiek
= \sum_{F\in\F} e(F)\int_{\sss^{|F|}} \ka_F 
=\frac12\int_{\sss^2} \kae
\le\infty
\end{equation*}
be the {\em asymptotic edge density} of $\kf$. 
Since every copy of $F$ contributes $e(F)$ edges,
the following theorem is almost obvious, provided we can ignore
overlapping edges. A formal proof will be given in
\refS{sec_subgraphs}.
(A similar result for the total number of atoms is given
in \refL{badv}.)
\begin{theorem}\label{Tedges}
As \ntoo, $e(\gnk)/n$
converges in probability to the asymptotic edge density $\xiek$. 
In other words, if $\xiek<\infty$ then $e(\gnk)=\xiek n+\opn$, and if
$\xiek=\infty$ then, for every constant $C$, we have $e(\gnk)>Cn$ whp.
Moreover, $\E e(\gnk)/n\to\xiek\le\infty$ 
\end{theorem}

As in~\cite{BJR}, our main focus will be the emergence of the giant component.
By the {\em component structure} of a graph $G$, we mean the set of vertex
sets of its components, i.e., the structure encoding only which vertices
are in the same component, not the internal structure of the components themselves.
When studying the component structure of $G(n,\kf)$, the model can be
simplified somewhat. Recalling that the atoms $F\in \F$ are connected
by definition,  when we add an atom $F$ to a graph $G$,
the effect on the component structure is simply to unite all components of $G$
that meet the vertex set of $F$, so only the vertex set of $F$ matters,
not its graph structure.
We say that $\kf$ is a {\em clique kernel family}
if the only
non-zero kernels are those corresponding to complete graphs;
the corresponding random graph model $G(n,\kf)$ is
a {\em clique model}. For questions
concerning component structure, it suffices to study clique models.
For clique kernels we write $\ka_\kx$ for $\ka_{K_\kx}$; as above,
we always assume that $\ka_\kx$ is symmetric, here meaning invariant under
all permutations of the coordinates of $\sss^\kx$.
Given a general kernel family $\kf$, the corresponding (symmetrized)
clique kernel family
is given by $\bkf=(\ka_\kx)_{\kx\ge2}$ with 
\begin{equation}\label{hk}
 \ka_\kx(x_1,\ldots,x_\kx) =
  \sum_{F\in \F: |F|=\kx} 
  \frac{1}{\kx!}
 \sum_{\pi\in \Symmr}
   \ka_F(x_{\pi(1)},\ldots,x_{\pi(\kx)}),
\end{equation}
where $\Symmr$ denotes the symmetric group of
permutations of $[r]$.
(This is consistent with our notation $\ka_2=\ka_{K_2}$.) 
In the Poisson version, with or without merging of parallel edges, 
the probability of adding some connected graph $F$ on a given
set of $r$ vertices is exactly the same in $G(n,\kf)$ and $G(n,\bkf)$,
so there is a natural coupling of these random graphs in which
they have exactly the same component structure.
In the non-Poisson version, the probabilities
are not quite the same, but close enough for our results
to transfer from one to the other.
Thus, when considering the size (meaning number of vertices) of the giant component in $G(n,\kf)$,
we may always replace $\kf$ by the corresponding clique kernel family.

It is often convenient to think of a clique model as a random hypergraph,
with the cliques as the hyperedges; for this reason we call a clique
kernel family 
a {\em hyperkernel}. Note that each unordered set of $\kx$ vertices corresponds
to $\kx!$ $\kx$-tuples, so the probability that we add a $K_\kx$ on a given set
of $r$ vertices is $\kx!\ka_r(x_{v_1},\ldots,x_{v_\kx})/n^{\kx-1}$. (More precisely,
this is the expected number of $K_r$s added with this vertex set.)

\subsection{A branching process}
Associated to each hyperkernel $\kf=(\ka_\kx)_{\kx\ge 2}$, there is a branching
process $\bpk$ with type space $\sss$, defined as follows.
We start with generation $0$ consisting of
a single particle whose type is chosen randomly from $\sss$
according to the distribution $\mu$. A particle $P$
of type $x$ gives rise to children in the next
generation according to a two-step process: first,
for each $\kx\ge 2$, construct a Poisson process $Z_\kx$ on $\sss^{\kx-1}$
with intensity 
\begin{equation}\label{bpint}
 \kx\ka_\kx(x,x_2,\ldots,x_\kx)\dd\mu(x_2)\cdots\dd\mu(x_\kx).
\end{equation}
We call the points of $Z=\bigcup_{\kx\ge 2} Z_\kx$ the {\em child cliques} 
of $P$. 
There are $\kx-1$ children of $P$ for each child clique
$(x_2,\ldots,x_\kx)\in \sss^{\kx-1}$, 
one each of types $x_2,\ldots,x_\kx$.
Thus the types of the children of $P$ form a multiset on $\sss$, with a certain
compound Poisson 
distribution we have just described. As usual, the children of different
particles are independent of each other, and of the history.

Considering the relationship to the graph $G(n,\kf)$, 
the initial factor $r$ in \eqref{bpint} arises because a particular vertex $v$
may be any one of the $r$ vertices in an $r$-tuple $(v_1,\ldots,v_r)$
on which we add a $K_r$.

We also consider the branching processes $\bpk(x)$, $x\in\sss$, defined exactly as $\bpk$,
except that we start with a single particle of the given type $x$.

\subsection{Two integral operators}
We shall consider two integral operators naturally
associated to $\bpk$.
Given any (measurable) $f:\sss\to [0,1]$,
define $\Sk(f)$ by
\begin{multline}\label{Sk}
 \Sk(f)(x) 
\\
= \sum_{\kx=2}^\infty \int_{\sss^{\kx-1}} \kx\ka_\kx(x,x_2,x_3,\ldots,x_\kx) 
 \left(1-\prod_{i=2}^{\kx} (1-f(x_i))\right) \dd\mu(x_2)\cdots\dd\mu(x_\kx),  
\end{multline}
and let
\[
 \Pk(f)(x) = 1-e^{-\Sk(f)(x)}.
\]
(The factors $\kx$ in \eqref{Sk} and in the definition of $\bpk$ are
unfortunate consequences of our choice of normalization.)

Let $P$ be a particle of $\bpk$ in generation $t$ with type $x$,
and suppose that each particle in generation $t+1$ of type $y$
has some property $\cQ$ with probability $f(y)$, independently of the
other particles. 
Given a child clique $(x_2,\ldots,x_\kx)$ of $P$,
the bracket in the definition of $\Sk$ expresses the probability that one
or more of the $\kx-1$ corresponding child particles has property $\cQ$.
Hence $\Sk(f)(x)$ is the expected number of child cliques containing
a particle with property $\cQ$, and, from the Poisson distribution of the
child cliques, $\Pk(f)(x)$ is the probability that there is at 
least one such clique, i.e., the probability that at least one child of $P$
has property $\cQ$.

Let $\rho(\kf)$ denote the survival probability of the branching
process $\bpk$, 
and $\rho_\kf(x)$ the survival probability of $\bpk(x)$. 
Assuming for the moment that the function $\rho_\kf:\sss\to[0,1]$ is measurable,
from the comments above and the independence
built into the definition of $\bpk$,
we see that the function $\rho_\kf$ satisfies
\[
 \rho_\kf = \Pk(\rho_\kf).
\]
Using simple standard arguments as in~\cite{BJR}, for example,
it is easy to check that $\rho_\kf$ is given
by the maximum solution to this equation, i.e., the pointwise
supremum of all solutions $f:\sss\to [0,1]$ to
\begin{equation}\label{fS}
 f = 1-e^{-\Sk(f)};
\end{equation}
see Lemma~\ref{l_max} below. From the definitions of $\bpk$ and $\bpk(x)$,
it is immediate that
\[
 \rho(\kf) = \int_\sss \rho_\kf(x) \dd\mu(x).
\]

In our analysis we shall also consider the linear operator $\Tk$
defined by
\begin{equation}\label{Tkdef}
 \Tk(f)(x) = \int_\sss \kae(x,y)f(y)\dd\mu(y),
\end{equation}
where $\kae$ is defined by \eqref{kaedef}.
For a hyperkernel $\kf$ (which is the only type of kernel family for which
we define the branching process), we have
\begin{equation}\label{ce}
 \kae(x,y) = \sum_{\kx\ge 2}\kx(\kx-1)
  \int_{\sss^{\kx-2}}\ka_\kx(x,y,x_3,x_4,\ldots,x_\kx)
   \dd\mu(x_3)\cdots\dd\mu(x_\kx),
\end{equation}
from which it is  easy to check that $\Tk$ is the linearized form of
$\Sk$:
more precisely, $\Tk$ is obtained by replacing
${1-\prod_{i=2}^{\kx} (1-f(x_i))}$ by $\sum_{i=2}^{\kx}f(x_i)$ in
the definition \eqref{Sk} of $\Sk$.

Let us note two simple consequences of this fact.
For any sequence $(y_i)_i$ in $\oi$ we have
${1-\prod_{i} (1-y_i)}\le\sum_{i}y_i$, so
\begin{equation}\label{st}
0\le\Sk(f)\le\Tk(f)
\end{equation}
for any $f:\sss\to [0,1]$.
Also, ${1-\prod_{i} (1-y_i)}> 0$ if and only if $\sum_{i}y_i>0$.
Since the integral of a non-negative function is positive
if and only if the function is positive on a set of positive measure,
it follows that for any $f:\sss\to [0,1]$ we have
\begin{equation}\label{SzTz}
 \Sk(f)(x)>0 \iff \Tk(f)(x)>0.
\end{equation}

In the edge-only case, when only $\ka_2$ is non-zero, 
$\kae=2\ka_2$ and
$\Tk=\Sk$.
When translating results from~\cite{BJR}, it is sometimes $\Tk$ and sometimes
$\Sk$ that plays the role of the linear operator $T_{\ka}$ appearing there.

\subsection{Main results}
In most of our results we shall need to impose some sort of integrability
condition on our kernel family; the exact condition depends
on the context.

\begin{num_definition}\label{Dint}
  \begin{thmenumerate}
	\item\label{Dwint}
A kernel family $\kf=(\ka_F)_{F\in\F}$ is 
\emph{integrable} if
\begin{equation}\label{wint}
 \ii{\kf} = \sum_{F\in\F} |F|\int_{\sss^{|F|}} \ka_F <\infty.
\end{equation}
This means that the expected number of atoms containing a given
vertex is bounded.

	\item\label{Deint}
A kernel family 
$\kf=(\ka_F)_{F\in\F}$ is 
\emph{edge integrable} if
\begin{equation*}%\label{eint}
 \sum_{F\in\F} e(F)\int_{\sss^{|F|}} \ka_F <\infty;
\end{equation*}
equivalently, $\xiek<\infty$ or
$\int_{\sss^2}\kae<\infty$.
This means that the expected number of edges in $\gnk$ is $O(n)$, 
see \refT{Tedges},
and
thus the expected degree of a given vertex is bounded. 
  \end{thmenumerate}
\end{num_definition}

Note that a hyperkernel $(\ka_r)$ is integrable if and only if
$ \sum_{r\ge2} r\int_{\sss^{r}} \ka_r <\infty$, and 
edge integrable if and only if
$ \sum_{r\ge2} r^2\int_{\sss^{r}} \ka_r <\infty$.

Since we only consider connected atoms $F$, it is clear that
\begin{equation*}
\text{edge integrable $\implies$  integrable}.
\end{equation*}

Our main result is that if $\kf$ is an integrable
kernel family satisfying a certain extra assumption, then the 
normalized size of the giant component in $G(n,\kf)$ is simply
$\rho(\kf)+\op(1)$. The extra assumption is essentially that the graph
does not split into two pieces. As in \cite{BJR}, we say
that a symmetric kernel $\kae:\sss^2\to [0,\infty)$
is \emph{reducible} if
\begin{equation*}
  \text{$\exists A\subset \sss$ with
$0<\mu(A)<1$ such that $\kae=0$ a.e.\ on $A\times(\sss\setminus
A)$};
\end{equation*}
otherwise $\kae$ is \emph{irreducible}. Thus $\kae$ is irreducible
if
\begin{equation*}%\label{t1ay}
  \text{$A\subseteq \sss$ and $\kae=0$ a.e.\ on $A\times(\sss\setminus A)$
implies $\mu(A)=0$ or $\mu(\sss\setminus A)=0$}.
\end{equation*}
A kernel family $(\ka_F)_{F\in\F}$ or hyperkernel $(\ka_\kx)_{\kx\ge 2}$ is
{\em irreducible} if the corresponding edge kernel $\kae$ is
irreducible.
It is easy to check that a kernel family $(\ka_F)_{F\in\F}$ is
irreducible if and only if for every
$A\subset \sss$ with $0<\mu(A)<1$ there exists an $F\in\F$ such that,
with $r=\kF$, if $x_1,\dots,x_r$ are chosen independently at random in $\sss$
with distribution $\mu$, then there is a positive probability that 
$\set{x_i}\cap A\neq\emptyset$,
$\set{x_i}\cap(\sss\setminus A)\neq\emptyset$
and
$\ka_F(x_1,\dots,x_r)>0$.
Informally, $(\ka_F)_{F\in\F}$ is irreducible if, whenever we partition
the type space into two non-trivial parts, edges between vertices
with types in the two parts are possible.

Note that a kernel family $\kfp$ and the corresponding
hyperkernel $\kf$ do {\em not} have the same edge kernel:
replacing each atom by a clique in general adds edges,
so $\kaep\le \kae$
with strict inequality possible. If $\kaep$ is irreducible,
then so is $\kae$; using the characterization
of irreducibility above,
it is easy to check that the reverse implication also holds.

We are now ready to state our main result; we write $C_i$ for the number
of vertices in the $i$th largest component of a graph $G$.

\begin{theorem}\label{th1}
Let $\kfp=(\ka'_F)_{F\in \F}$ be an irreducible,
integrable kernel family,
and let $\kf=(\ka_\kx)_{\kx\ge2}$ be the corresponding hyperkernel, given by \eqref{hk}.
Then
\[
 C_1\bb{G(n,\kfp)} = \rho(\kf)n +\op(n),
\]
and $C_2\bb{G(n,\kfp)} =\op(n)$.
\end{theorem}

The reducible case reduces to the irreducible one; see \refR{R_red}. 

\begin{remark}
Unsurprisingly, part of the proof of \refT{th1} involves
showing that (in the hyperkernel case)
the branching process captures the `local structure'
of $G(n,\kf)$; see Section~\ref{sec_loc} and in particular
Lemma~\ref{nkint}. So \refT{th1} can be seen as saying
that {\em within this broad class of models} the
local structure determines the size of the giant component.
Of course, the restriction is important, as shown
by the fact that the global assumption of irreducibility
is necessary.
\end{remark}

Of course, for \refT{th1} to be useful we would like to know
something about the survival probability $\rho(\kf)$; as noted
earlier, $\rho(\kf)$ can be calculated from $\rho_\kf$,
which is in turn the largest solution to a certain functional equation~\eqref{fS}.
Of course, the main thing we would like to know is
when $\rho(\kf)$ is positive; as in~\cite{BJR}, it turns out
that the answer depends on the $L^2$-norm $\norm{\Tk}\le\infty$
of the operator $\Tk$ defined by \eqref{Tkdef}.
(Since this operator is symmetric, its $L^2$-norm
is the same as its spectral radius. In other contexts,
it may be better to work with the latter.)

\begin{theorem}\label{th2}
Let $\kf$ be an integrable hyperkernel.
Then $\rho(\kf)>0$ if and only if $\norm{\Tk}>1$.
Furthermore, if $\kf$ is irreducible and $\norm{\Tk}>1$,
then $\rho_\kf(x)$ is the unique non-zero solution
to the functional equation \eqref{fS}, and $\rho_\kf(x)>0$ 
holds for a.e.\ $x$.
\end{theorem}

In general, $\norm{\Tk}$ may be rather hard to calculate; a non-trivial 
example where we can calculate the norm easily is given in \refSS{ss_egg}.
Let us give a trivial example here: suppose that each $\ka_r$ is constant,
say $\ka_r=c_r$. Then $\kae(x,y) = \sum_r r(r-1)c_r = 2\xie(\ka)$ for
all $x$ and $y$, so
\begin{equation}\label{unif}
 \norm{\Tk}=2\xie(\ka).
\end{equation}
This is perhaps surprising:
it tells us that for such uniform hyperkernels, the critical point
where a giant component emerges
is determined only by the total number of edges added; it does not
matter what size cliques they lie in, even though,
for example, the third edge in every triangle is `wasted'. This is not true for arbitrary
kernel families: we must first replace each atom by a clique.

Note that for any hyperkernel,
\[
 \norm{\Tk}\ge \langle 1,\Tk 1\rangle = \int \kae = 2\xie(\ka),
\]
with equality if and only if $1$ is an eigenfunction, i.e., if the
asymptotic expected degrees $\la(x)=\int_\sss \kae(x,y)\dd\mu(y)$
are the same (ignoring sets of measure 0); 
c.f.\ \cite[Proposition 3.4]{BJR}.

\subsection{Relationship to the results in~\cite{BJR}}

In the edge-only case, the present
results are almost (see below) special
cases of those~\cite{BJR}. The set-up here is much simpler, as we choose to
insist that the vertex types $x_1,\ldots,x_n$ are i.i.d. This avoids many of
the complications arising in~\cite{BJR}.
In one way, the present set-up is, even in the edge-only
case, more general than that considered in~\cite{BJR}:
with the types \iid, there is no need to restrict the kernels
other than to assume integrability
(in~\cite{BJR} we needed them continuous a.e.),
and one does not need to impose the `graphicality' assumption
of~\cite{BJR}. Thus the edge-only case here actually complements
the results in~\cite{BJR}. We could form a common generalization,
but we shall not do this in detail; we believe that it is just a question
of combining the various technicalities here and in~\cite{BJR}, and that
no interesting new difficulties arise. Of course, these technicalities
are rather beside the point of the present paper; our interest is the
extension from kernels to hyperkernels. This turns out not to be
as straightforward as one might perhaps expect.
The problem is that the correlation between edges forces
us to deal with a non-linear operator, namely $\Sk$.
\medskip

The rest of the paper is organized as follows. In the next section we prove the results
about the non-Poisson branching process $\bpk$ that we shall need later,
the most important of which is \refT{th2}.
In Section~\ref{sec_loc} we consider the local coupling between the graph
and the branching process, showing in particular that the `right' number
of vertices are in components of any fixed size. In Section~\ref{sec_giant}
we complete the proof of Theorem~\ref{th1},
showing that whp there is at most one `large' component, which is then
a `giant' component of the right size. We briefly
discuss percolation on the graphs $G(n,\kf)$ in \refS{Sdisconnected}. 
In Sections~\ref{sec_degrees}
and~\ref{sec_subgraphs} we consider simpler properties of $G(n,\kf)$, namely
the asymptotic degree distribution and the number of subgraphs
isomorphic to a given graph.
Our results in~\refS{sec_subgraphs} include \refT{Tedges} as a simple special
case.
In Section~\ref{sec_pl} we illustrate
the flexibility of the model by carrying out explicit calculations for a special
case, giving graphs with power-law degree sequences with a range of exponents
and a range of clustering and mixing coefficients; see \refS{sec_pl}
for the definitions of these coefficients.  
Finally, in \refS{sec_lim} we discuss connections between our model
and various notions of graph limit, and state two open questions.

\section{Analysis of the branching process}

In this section, which is the heart of the paper, 
we forget about graphs, and study the (compound Poisson)
branching process $\bpk$.
One might expect the arguments of~\cite{BJR} to carry over {\em mutatis mutandis}
to the present context, but in the branching process analysis
this is very far from the truth; this applies especially to
the proof of \refT{unique} below.

Throughout the section we work with an
integrable hyperkernel
$\kf=(\ka_\kx)_{\kx\ge 2}$, i.e., we assume
that $\ii{\kf} = \sum_r r\int\ka_r<\infty$.
Our main aim
in this section is to prove \refT{th2}.

For $x\in \sss$ let
\[
 \la(x) = (\Sk(1))(x) 
= \sum_{\kx=2}^\infty \int_{\sss^{\kx-1}} \kx\ka_\kx(x,x_2,x_3,\ldots,x_\kx) 
 \dd\mu(x_2)\cdots\dd\mu(x_\kx),
\]
so $\la(x)$ is the expected number of child cliques of a particle of type
$x$.
We have
\[
 \int_\sss \la(x) \dd\mu(x) 
= \sum_{\kx\ge 2} \int_{\sss^\kx} \kx \ka_\kx = \ii{\kf},
\]
which is finite by our integrability assumption \eqref{wint}.
It follows that $\la(x)<\infty$ holds almost everywhere.
Changing each kernel $\ka_\kx$ on a set of measure zero, we may assume that
$\la(x)$ is finite for all $x$. (Such a change is irrelevant for the branching
process and for the graph.) From now on, we thus assume that
$\la(x)<\infty$ holds for all $x$, for any hyperkernel $\kf$ we consider.

Since a Poisson random variable with finite mean is always finite, any 
particle in $\bpk$ has a finite number of child cliques, and hence
a finite number of children, even though the expected number of children
may perhaps be infinite. Hence, the event that the branching process
dies out (i.e., that some generation is empty) coincides with the event that 
it is finite.

Using this fact, we have the following, standard result.
Recall that $\rho_\kf(x)$ denotes the survival probability of
the branching process $\bpk(x)$ that starts with a single particle
of type $x$, and $\rho_\kf$ denotes the function $x\mapsto \rho_\kf(x)$ .
\begin{lemma}\label{l_max}
The function $\rho_\kf$ satisfies
the functional equation \eqref{fS}. Furthermore, if $f:\sss\to[0,1]$ is any other
solution to \eqref{fS}, then $0\le f(x)\le \rho_\kf(x)<1$ holds for every $x$.
\end{lemma}
\begin{proof}
Let $\rho_t(x)$ be the probability that $\bpk(x)$ survives for at least
$t$ generations, so $\rho_0$ is identically $1$.
Conditioned on the set of child cliques, and hence children, of the root,
each child of type $y$ survives for $t$ further generations
with probability $\rho_t(y)$. These events are independent for different
children by the definition of the branching process, so
$\rho_{t+1}=\Pk(\rho_t)$. The result follows from the monotonicity
of $\Phik$ and the fact that $\rho_t(x)\downto \rho_\kf(x)$,
noting that $\Phik(1)(x)=1-e^{-\lambda(x)}<1$ for the strict inequality.
\end{proof}

Let us remark for the last time on the measurability of the functions
we consider: in the proof above, $\rho_0$ is measurable by definition.
From the definition of $\Phik$ and the measurability of each $\ka_k$,
it follows by induction that each $\rho_t$ is measurable,
and hence that $\rho_\kf$ is. Similar arguments apply in many places later,
but we shall omit them.

We next turn to the uniqueness of the non-zero solution (if any)
to \eqref{fS}. The key ingredient in establishing this
is the following simple inequality concerning the non-linear operator $\Sk$.

\begin{lemma}\label{l_fSg}
Let $\kf$ be an integrable hyperkernel,
and let $f$ and $g$ be measurable functions on $\sss$ with $0\le f\le g\le 1$.
Then
\[
 \int_\sss f \Sk g \le \int_\sss g \Sk f.
\]
\end{lemma}
\begin{proof}
We may write $\Sk$ as $\sum_{r\ge 2} S_r$, where $S_r$
is the non-linear operator corresponding to the single
kernel $\ka_\kx$, so $S_r(f)$ is defined by the summand in \eqref{Sk}.
It suffices to prove that
\begin{equation}\label{fSkg}
 \int_\sss f S_\kx g \le \int_\sss g S_\kx f.
\end{equation}
We shall in fact show that for any (distinct) $x_1,\ldots,x_\kx\in \sss$
we have
\begin{equation}\label{pointwise}
 \sum_{\pi\in \Symmr} 
   f(x_{\pi(1)})\left(1- \prod_{i=2}^\kx\bb{1-g(x_{\pi(i)})}\right)
\le
 \sum_{\pi\in \Symmr} 
   g(x_{\pi(1)})\left(1- \prod_{i=2}^\kx\bb{1-f(x_{\pi(i)})}\right)
\end{equation}
Since $\ka_\kx$ is symmetric, \eqref{fSkg} follows. (In fact, 
\eqref{fSkg} can be true in general only if \eqref{pointwise} always holds,
considering the symmetrization of a delta function.)
Now \eqref{pointwise} can be viewed as an inequality
in $2\kx$ variables $f(x_1),\ldots,f(x_\kx),g(x_1),\ldots,g(x_\kx)$.
This inequality is linear in each variable. Furthermore,
it is linear in each pair $(f(x_i)$, $g(x_i))$. 
In proving \eqref{pointwise} for any $0\le f\le g\le 1$,
we may thus assume that for each $i$ one of three possibilities holds:
$0=f(x_i)=g(x_i)$, $f(x_i)=g(x_i)=1$, or $f(x_i)=0$ and $g(x_i)=1$.
In other words, we may assume that $f$ and $g$ are $\{0,1\}$-valued.

Suppose then for a contradiction
that \eqref{pointwise} fails for some $\{0,1\}$-valued $f$ and $g$ with $f\le g$.
Then there must be some permutation $\pi$ such that
\begin{equation}\label{pw2}
 f(x_{\pi(1)})\left(1- \prod_{i=2}^\kx\bb{1-g(x_{\pi(i)})}\right)
>
 g(x_{\pi(1)})\left(1- \prod_{i=2}^\kx\bb{1-f(x_{\pi(i)})}\right),
\end{equation}
which we may take without loss of generality to be the identity permutation.
Since both sides of \eqref{pw2} are $\{0,1\}$-valued, the left must be
$1$ and the right $0$. Since the left is $1$, we have $f(x_1)=1$,
so, using $f\le g$, $g(x_1)=1$. But now for the right hand side
of \eqref{pw2} to be 0 the final product in \eqref{pw2} must be $1$,
so $f(x_i)=0$ for $i=2,\ldots,\kx$, i.e., $f$ takes the value $1$ only once.
Of course, $g$ must take the value $1$ at least twice, otherwise
we have equality.
But now the left hand side of \eqref{pointwise} is exactly $(\kx-1)!$,
coming from terms with $\pi(1)=1$ and hence $f(x_{\pi(1)})=1$.
The right hand side is at least $(\kx-1)!$, from any $\pi$
mapping $1$ to some $j\ne 1$ with $g(x_j)=1$.
Hence \eqref{pointwise} holds after all, giving
a contradiction and completing the proof.
\end{proof}

If $\kf$ is reducible, then \eqref{fS} may in general
have several non-zero solutions. To prove uniqueness
in the irreducible case we need to know what
irreducibility tells us about $\Sk$. 

\begin{lemma}\label{l_red2}
If there exists a measurable $f:\sss\to [0,1]$ with 
$0<\mu\set{f>0}<1$ and $\set{\Sk f>0}\subseteq \set{f>0}$,
then $\kf$ is reducible.
\end{lemma}
\begin{proof}
Let $A=\set{f>0}$, so by assumption $\Sk f=0$ on $A^\cc=\sss\setminus A$.
From \eqref{SzTz} we have $\set{\Tk f=0}=\set{\Sk f=0}$,
so $\Tk f=0$ on $A^\cc$. From the definition of $\Tk$ it follows
that $\kae=0$ a.e.\ on $A^\cc\times A$, so $\kae$ is reducible.
But this is what it means for $\kf$ to be reducible.
\end{proof}
In fact, taking $f$ to be a suitable indicator function,
one can check that the converse of Lemma~\ref{l_red2} also holds.

Using Lemmas~\ref{l_fSg} and~\ref{l_red2}
it is easy to deduce uniqueness of any non-zero
solution to~\eqref{fS}. 

\begin{theorem}\label{unique}
Let $\kf$ be an irreducible, integrable hyperkernel,
and let $f$ and $g$ be solutions to \eqref{fS}
with $0\le f(x)\le g(x)\le1$ for every $x$.
Then either $f=0$ or $f=g$.
In particular, the only solutions to \eqref{fS} are $\rho_\kf$
and the zero function, which may or may not coincide.
\end{theorem}
\begin{proof}
We may suppose that $f$ is not $0$ a.e.; otherwise, $f=\Pk(f)$
would be identically zero.
Since $f$ solves \eqref{fS}, we have $\set{f=0}=\set{\Sk f=0}$,
so by \refL{l_red2}
we cannot have $0<\mu\{f>0\}<1$. The only possibility
left is that $\mu\{f>0\}=1$, i.e., $f>0$ a.e.
Turning to $g$, since $\kf$ is integrable, we have
$\Sk(g)(x)\le\Sk(1)(x)=\la(x)<\infty$ for a.e. $x$,
and thus $g=\Pk(g)<1$ a.e.

Since $f$ and $g$ solve \eqref{fS}, we have $\Sk(f)(x)=-\log(1-f(x))$
and $\Sk(g)(x)=-\log(1-g(x))$.
Hence,
\begin{align*}
 f\Sk(g) = -f\log(1-g) &= f(g+g^2/2+g^3/3+\cdots) \\
 &\ge g(f+f^2/2+f^2/3+\cdots) = g\Sk(f)
\end{align*}
whenever $0\le f\le g\le1$, with strict inequality whenever $0<f<g$. 
%(The sums may be infinite if $g=1$, but that is no problem.) 
Since $\kf$ is integrable, it is immediate 
from the definition \eqref{Sk} 
that $\Sk f$ and $\Sk g$ are integrable,
and it follows that
\[
 \int_\sss f \Sk g \ge \int_\sss g \Sk f,
\]
with strict inequality unless $f=g$ a.e.
Since \refL{l_fSg} gives the reverse inequality, we
have $f=g$ a.e., and thus $f=\Pk f=\Pk g=g$.
The second statement then follows from \refL{l_max}.
\end{proof}

\refT{unique} generalizes the corresponding result in~\cite{BJR}, namely
Lemma~5.9. Indeed, in the edge-only case (when only $\ka_2$ is non-zero),
the operators $\Sk$ and $\Tk$ coincide, and \refL{l_fSg} holds trivially,
using the symmetry of $\Tk$. This shows that, with hindsight, the proof
of Lemma~5.9 in~\cite{BJR} may be simplified considerably,
by considering $\int_\sss fTg$ instead of $\int_\sss fTh$, $h=(g-f)/2$.
This is significant, since the proof in~\cite{BJR} does not adapt
readily to the present context.

Although simple, the proof of \refT{unique} above is a little mysterious
from a branching process point of view.
It is tempting to think that the result is `obvious', and indeed
that a corresponding result should hold for any Galton--Watson process.
However, some conditions are certainly necessary, and it is not
clear what the right conditions are for a general process.
(Irreducibility is always needed, of course.)
In~\cite{Rsmall}, a corresponding result is proved for a general branching
process satisfying a certain continuity assumption; the proof
uses the convexity property $\Phi(\la f)\ge \la\Phi(f)$ for any function $0\le f\le 1$
and any $0\le \la\le 1$, which holds for all Galton--Watson branching processes.
In~\refT{unique}, continuity is not needed, but some kind of symmetry is;
there does not seem to be an obvious common generalization of these results.

Indeed, the next example shows that the situation is not that simple: in 
the compound Poisson case (as opposed to the simple
Poisson case), symmetry of the relevant linear operator is not enough.

\begin{example}
Let $\sss=\{1,2,3,\ldots\}$ with $\mu\{i\}=2^{-i}$ for each $i$,
and consider the branching process $\bp=\bp(x)$ with type space $(\sss,\mu)$ defined
as follows. Start with a single particle of some given type $x$. Each particle of type $i$
has a Poisson number of children of type $i+1$ with mean $2 = 2^{i+2}\mu\{i+1\}$;
we call these `forward children'.
Also, for $i\ge 2$, a particle of type $i$ has `backward children' of type $i-1$:
the number of these is $4^{i+1}$ times a Poisson with mean $4^{-i}$.
Note that the expected number of backward children is $4=2^{i+1}\mu\{i-1\}$.
Defining the `edge-kernel' $\kae$ so that the expected number of children
of type $j$ that each particle of type $i$ has is given by
$\kae(i,j)\mu\{j\}$, 
we have
$\kae(i,j)=2^{1+\max\{i,j\}}$ if $|i-j|=1$ and $\kae(i,j)=0$ otherwise,
so $\kae$ is symmetric and irreducible.

Define the non-linear operator $\Phi$ associated to $\bp$ in the natural way,
so $\Phi(f)(x)$ is the probability that at least one child of the root of type
$x$ has a certain property, if each child of type $y$ has this property
independently with probability $f(y)$. As before, the survival probability $\rho(x)$
satisfies $\rho=\Phi(\rho)$.

Let $\tau(x)$ denote the probability that
the process {\em survives transiently}, i.e., survives forever, but, for each $i$,
contains in total only finitely many particles of type $i$.
Consider the `forward process' given by ignoring backward children. This is simply
a Poisson Galton--Watson process with on average 2 offspring, and so survives
with some positive probability. Also, given that it survives, there is a positive
probability that for every $t$, generation $t$ contains at most $3^t$ particles, say.
But since the particles in generation $t$ have type $x+t$, the expected number of {\em sets}
of backwards children of all particles in the forward process is at most
$\sum_{t\ge 0} 3^t4^{-t-1}< \infty$, 
and with positive probability the particles in the forward process
have no backwards children. But in this case, the forward process is the whole
process, and the process survives transiently. Hence $\tau(x)>0$ for every $x$.

Let $\sigma(x)=\rho(x)-\tau(x)$ be the probability that the process
survives recurrently. Considering the children of the initial particle,
we see that $\sigma=\Phi(\sigma)$.
The process restricted to any two consecutive types is already
supercritical, and so has positive probability of surviving
by alternating between these types. Thus $\sigma(x)>0$ for all $x$.
We showed above that $\tau(x)=\rho(x)-\sigma(x)>0$ for all $x$,
so $0<\sigma(x)<\rho(x)$, and $f=\Phi(f)$ has (at least) two
non-zero solutions, namely $\sigma$ and $\rho$.
\end{example}

Let us turn to the analysis of the solution $\rho_\kf$ to \eqref{fS},
and in particular 
the question of when $\rho>0$, i.e., when the branching process
$\bp_\kf$ is supercritical. 
Throughout we consider an integrable hyperkernel $\kf$,
with corresponding edge kernel $\kae$.

Recall that we may assume that $\la(x)=\Sk(1)(x)$ is finite everywhere. Hence,
for any $f$ satisfying \eqref{fS}, we have $f(x)<1$ for all $x$.
On the other hand, we cannot assume that $\kae$ is integrable, or
indeed finite. 
For one natural example, consider the integrable hyperkernel with each
$\ka_\kx$ 
constant, and $\ka_\kx=1/\kx^3$. In this case $\kae(x,y)=\infty$
for all $x$ and $y$. 
If $\kae$ is infinite on a set of positive measure, then we take $\norm{\Tk}$
to be infinite. 

\begin{lemma}\label{l_sc}
If $\norm{\Tk}\le 1$, then $\rho(\kf)=0$.
\end{lemma}
\begin{proof}
Suppose that $f$ is a solution to \eqref{fS} that is not $0$ a.e.
Since $-\log(1-t)>t$ for $0<t<1$, we have
$\Sk(f)(x)\ge f(x)$, with strict inequality on a set of positive
measure.
But $\Tk(f)(x)\ge \Sk(f)(x)$ by \eqref{st}, 
so $\Tk(f)(x)\ge f(x)$, with strict inequality on a set of positive
measure. Hence $\tn{\Tk f}>\tn{f}$, so $\norm{\Tk}> 1$.
\end{proof}

Lemmas 5.12 and 5.13 of~\cite{BJR} carry over to the present context,
with only minor modifications.
Given functions $f_1,f_2,\ldots$ and $f$, we write $f_n\upto f$
if the sequence $(f_n)$ is monotone increasing and converges to $f$ pointwise.
\begin{lemma}
If $0\le f\le 1$ and $\Phik(f)\ge f$, then $\Phik^m(f)\upto g$ as $m\to\infty$,
for some $1\ge g\ge f$ with $\Phik(g)=g$.
\end{lemma}
\begin{proof}
Since $f\le\Phik(f)$, monotonicity of $\Phik$ gives $\Phik(f)\le \Phik^2(f)$ and,
by induction, $\Phik^m(f)\le \Phik^{m+1}(f)$ for all $m\ge 0$.
Since $0\le \Phik^m(f)\le 1$, it follows that
$g(x)=\lim_{m\to\infty} \Phik^m(f)(x)$ exists for every $x$, and $0\le g\le 1$.
From monotone convergence we have $\Sk(g)=\lim_{m\to\infty} \Sk(\Phi_k^m(f))$,
from which it follows that $\Phik(g)=g$.
\end{proof}

\begin{lemma}\label{l_fup}
If there is a function $f:\sss\to [0,1]$,
not a.e.\ $0$, such that
$\Sk(f)\ge (1+\delta)f$ 
for some $\delta>0$, then $\rho(\kf)>0$.
\end{lemma}
\begin{proof}
The proof is the same as that of Lemma 5.13 in~\cite{BJR}, using $\Sk$ in place
of $T_\ka$.
\end{proof}

The next step is to show that if $\norm{\Tk}>1$, then there is a function
$f$ with the property described in \refL{l_fup}. In~\cite{BJR} we did this
by considering a bounded kernel. Here we have to be a little more
careful, as we are working with the non-linear operator $\Sk$ rather than
with $\Tk$; this is no problem if we truncate our kernels suitably.

\begin{num_definition}\label{Dbounded}
We call a hyperkernel $\kf=(\ka_r)_{r\ge2}$ {\em bounded} 
if two conditions hold: only
finitely many of the $\ka_\kx$ are non-zero, and each $\ka_\kx$ is
bounded.

Similarly (for later use), a general kernel family
$(\ka_F)_{F\in\F}$ is \emph{bounded} if only
finitely many of the $\ka_F$ are non-zero, and each $\ka_F$ is
bounded.
\end{num_definition}
In other words, $\kf$ is bounded if there are constants $R$ and $M$
such that $\ka_\kx=0$ for $\kx>R$, and $\ka_\kx$ is pointwise bounded by $M$
for $\kx\le R$. Note that if $\kf$ is bounded, then the corresponding
edge kernel $\kae$ is bounded in the usual sense.

Given a hyperkernel $\kf=(\ka_r)$, 
for each $M>0$ we let  $\kfM$ be the bounded hyperkernel obtained
from $\ka$ by truncating each $\ka_\kx$, $\kx\le M$, at $M$, and replacing
$\ka_\kx$ by a zero kernel for $\kx>M$.
Thus
\begin{equation}\label{trunc}
  \kaM_r=
  \begin{cases}
	\ka_r\wedge M, & r\le M,
\\
0, & r>M.
  \end{cases}
\end{equation}
The truncation $\kfM=(\kaM_F)_{F\in\F}$ of a general kernel family
$(\ka_F)_{F\in\F}$ is defined similarly, replacing
the condition $r\le M$ by $|F|\le M$.

\begin{lemma}\label{l_ef}
If $\norm{\Tk}>1$ then there is a $\delta>0$ and an $f:\sss\to [0,1]$,
not a.e.\ $0$,
such that $\Sk(f)\ge (1+\delta)f$.
\end{lemma}
\begin{proof}
We slightly modify the proof of Lemma 5.16 of~\cite{BJR}.

Consider the truncated hyperkernels $\kfM$ defined in \eqref{trunc}.
From \eqref{ce} and monotone convergence, the corresponding edge kernels
$\kaeM$ tend up to $\kae$ (which may be infinite in some places) pointwise.
Arguing as in the proof of Lemma 5.16 of~\cite{BJR}, since $\norm{\Tk}>1$
there is some positive $f$ with $\tn{f}=1$ and $1<\tn{\Tk f}\le\infty$.
By monotone convergence, $T_{\kaeM}f\upto \Tk f$,
so $\tn{T_{\kaeM}f}\upto \tn{\Tk f}$, and there is some $M$ with
$\norm{T_{\kaeM}}\ge \tn{T_{\kaeM}f}>1$.

Since $\kaeM$ is bounded, setting $\delta=(\norm{T_{\kaeM}}-1)/2>0$,
by Lemma 5.15 of~\cite{BJR} it follows
that there is a bounded $f\ge 0$ with $f$ not 0 a.e.\ such that
\[
 T_{\kaeM} f = \norm{T_{\kaeM}}f = (1+2\delta) f.
\]
We may assume that $0\le f\le1$.
If $0\le y_i\le\gam<1$, $i=1,\dots,r$, then (by induction)
$1-\prod_{i=1}^r(1-y_i)\ge(1-\gam)^{r-1}\sum_{i=1}^ry_i$, and it follows
that if $\gamma>0$ is chosen small enough, then
\[
 S_{\kfM} (\gamma f) 
\ge(1-\gam)^{M-1} T_{\kaeM}(\gam f)
\ge  (1+\delta) (\gamma f).
\]
Since
$\Sk(\gamma f)\ge S_{\kfM}(\gamma f)$, the result follows.
\end{proof}

Theorem~\ref{th2} follows by combining the results above.
\begin{proof}[Proof of Theorem~\ref{th2}]
Together Lemmas~\ref{l_sc}, \ref{l_fup} and \ref{l_ef} show that
$\rho(\kf)>0$ if and only if $\norm{\Tk}>1$.
Uniqueness is given by Theorem~\ref{unique}.
The final statement is immediate from \refL{l_red2}.
%considering the set 
%$A$ where $\rho_x(\kf)=0$: the edge kernel
%$\kae$ is zero a.e. on $A\times(\sss\setminus A)$,
%so if $\kf$ is irreducible then $\mu(A)=0$ or $\mu(A)=1$.
\end{proof}

\bigskip
Having proved \refT{th2}, our next aim is to prove \refT{th1}.
The basic strategy will involve comparing the neighbourhoods of a
vertex in the random graph $G(n,\kf)$ with the branching process $\bpk$.
As in~\cite{BJR}, it will be convenient to carry out the comparison
only for certain restricted hyperkernels. In order to deduce results
about $G(n,\kf)$ in general, one needs approximation results
both for the graph and for the branching process. We now turn
to such results for branching processes.

Lemma 6.3 and Theorems 6.4 and 6.5 of~\cite{BJR} carry over to the present
context, {\em mutatis mutandis}, using the results above about $\rho(\kf)$
instead of the equivalents in~\cite{BJR}, and replacing $T_\ka$
by $\Sk$ or $\Tk$ as appropriate: $\Sk$ when considering $\Phi_\ka$,
and $\Tk$ when arguing using $L^2$-norms.
In these results $\rho_\kf$ denotes the function $x\mapsto \rho_\kf(x)$,
and $\rho_{\ge k}(\kf,x)$ and $\rho_{\ge k}(\kf)$ denote
respectively the probabilities that $\bpk(x)$ and $\bpk$
have total size at least $k$, 
where the {\em size} of a branching process is the total number
of particles in all generations.

\begin{lemma}\label{L12}
  If\/ $\kf\le\kfp$, then $\rho(\kf)\le \rho(\kfp)$.
 \noproof
\end{lemma}

\begin{theorem}\label{TappB}
  \begin{thmenumerate}
\item
 Let $\kfn$, $n=1,2,\dots$, be a sequence of hyperkernels on $(\sss,\mu)$
 increasing
 \aex{} to an integrable hyperkernel $\kf$.
Then
$\rho_{\kfn}\upto \rho_\kf$ \aex{} and
$\rho(\kfn)\upto\rho(\kf)$.
\item
 Let $\kfn$, $n=1,2,\dots$, be a sequence of 
integrable
hyperkernels on $(\sss,\mu)$
 decreasing 
 \aex{} to $\kf$.
Then
$\rho_{\kfn}\downto \rho_\kf$ \aex{} and
$\rho(\kfn)\downto\rho(\kf)$.
\noproof
 \end{thmenumerate}
\end{theorem}

\begin{theorem}\label{TappC}
  \begin{thmenumerate}
\item
 Let $\kfn$, $n=1,2,\dots$, be a sequence of hyperkernels on $(\sss,\mu)$
 increasing 
 \aex{} to a hyperkernel $\kf$.
%\marginal{removed integrable here - it's trivially not needed looking at the proof in \cite{BJR}}
Then, for every $k\ge 1$,
$\rhogek(\kfn;x)\upto\rhogek(\kf;x)$ for \aex{} $x$ and
$\rhogek(\kfn)\upto\rhogek(\kf)$.
\item
 Let $\kfn$, $n=1,2,\dots$, be a sequence of 
integrable
hyperkernels on $(\sss,\mu)$
 decreasing
 \aex{}
to  $\kf$.
Then, for every $k\ge1$,
$\rhogek(\kfn;x)\downto\rhogek(\kf;x)$ for \aex{} $x$ and
$\rhogek(\kfn)\downto\rhogek(\kf)$.
\noproof
 \end{thmenumerate}
\end{theorem}

\begin{remark}
The assumption that $\kfn$ be integrable in Theorems
\ref{TappB}(ii) and \ref{TappC}(ii) can be weakened to
$\la_{\kfn}(x)<\infty$ for a.e.\ $x$,
where $\la_{\kfn}(x)$ is the expected number of
child cliques in
$\bp_{\kfn}$ of a particle of type $x$;
see \cite{BJR}.
\end{remark}

\section{Local coupling}\label{sec_loc}

We now turn to the local coupling between our random graph
and the corresponding branching
process, relating the distribution of small components in $G(n,\kf)$
to the branching process $\bpk$.
In~\cite{BJR}, we were essentially forced to condition on the vertex types,
since these were allowed to be deterministic to start with.
Here, with \iid\ vertex types, there
is no need to do so. This allows us to couple directly for
all bounded hyperkernels, rather than simply for finite type ones.

We shall consider a variant of the usual component exploration process,
designed to get around the following problem. When we test edges from a
given vertex $v$ to all other vertices, the probability of finding a given edge
$vw$ depends on the type of $w$ as well as that of $v$. Hence,
{\em not} finding such an edge changes the conditional distribution
of the type of $w$. If the kernel is well behaved, it is easy to
see that this is a small effect. Rather than quantify this, it is easier
to embed $G(n,\kf)$ inside a larger random graph with uniform kernels.
Testing edges in the larger graph does not affect the conditional
distribution of the vertex types; we make this precise below.
In doing so, it will be useful to take the hypergraph viewpoint: given
a hyperkernel $\kf$, let $H(n,\kf)$ be the hypergraph on $[n]$
constructed according to the same rules as $G(n,\kf)$, except that
instead of adding a $K_\kx$ we add a hyperedge with $\kx$ vertices.
In fact, we consider the Poisson version of the model, allowing multiple
copies of the same hyperedge.

Let $\kf$ be a bounded hyperkernel, and let $\kfpl$ be a corresponding
upper bound, so $\kapl_\kx$ is the constant kernel $M$ for $\kx\le R$,
and zero for $\kx>R$, while $\ka_\kx\le\kapl_r$ holds pointwise for all $\kx$.

Taking, as usual, our vertex types $x_1,\ldots,x_n\in \sss$
to be independent, each having the distribution $\mu$,
we construct coupled random (multi-)hypergraphs $H_n$ and $H^+_n$ on $[n]$
as follows: first construct $H^+_n=H(n,\kf^+)$
by taking, for every $2\le r\le R$, a Poisson $\Po(r!M/n^{r-1})$ number of copies
of each possible $r$-element hyperedge, with all
these numbers independent.
Although in our formal definition of $H^+_n$ we first
decide the vertex types, $H^+_n$ is clearly independent
of these types. Hence, given $H^+_n$, the types are (still)
\iid\ with distribution $\mu$.

Given $H^+_n$ and the \iid\ types $x_1,\ldots,x_n$ of the vertices,
we may form $H_n$ by selecting each hyperedge $\{v_1,\ldots,v_\kx\}$ of $H_n^+$ to be
a hyperedge of $H_n$ with probability $\ka_\kx(x_{v_1},\ldots,x_{v_\kx})/M$,
independently of all other hyperedges.
It is easy to see that this gives the right distribution for $H_n=H(n,\kf)$.
(If we disallowed multiple copies of an edge, there would be an irrelevant
small correction here.)

Turning to the branching processes, there is an analogous coupling of
$\bpk$ and $\bp_{\kf^+}$: first construct $\bp_{\kf^+}$, which
may be viewed as a single-type process, according
to our two-step construction via child cliques.
Then assign each particle a type
according to the distribution $\mu$,
independently of the other particles and of the branching process.
Then form the child cliques in $\bpk$ by keeping each child
clique in $\bp_{\kf^+}$ with an appropriate probability depending
on the types, deleting not only the children corresponding
to deleted child cliques, but also all their descendants.

Let $v\in [n]$ be chosen uniformly at random, independently
of $H_n$ and $H_n^+$.
Let $\Gamma_d$ denote the $d$-neighbourhood
of $v$ in $H_n$, and $\Gamma_d^+$ that in $H_n^+$.
Counting the expected number of cycles shows that for any fixed $d$, 
the hypergraph $\Ga_d^+$ is whp treelike.
Furthermore, standard arguments as for $G(n,c/n)$ show that one may couple $\Ga_d^+$
and the first $d$ generations of $\bp_{\kf^+}$ so as
to agree in the natural sense whp. 
When $\Ga_d^+$ is treelike, then $\Ga_d\subset \Ga_d^+$ may be constructed
using exactly the same random deletion process that gives (the first
$d$ generations of) $\bpk$ as a subset of $\bp_{\kf^+}$.
It follows that $\Ga_d$ and the first $d$ generations of $\bpk$ may be
coupled to agree whp.

Recalling that $G(n,\kf)$ and $H_n$ have the same components,
for any fixed $k\ge 1$ one can determine whether the component containing
$v$ has exactly $k$ vertices by examining $\Ga_{k+1}$.
Writing $N_k(G)$ for the number of vertices of a graph $G$ that
are in components of size $k$, it follows
that
\[ 
 \E N_k(G(n,\kf)) = n\Pr(|\bpk|=k) +o(n).
\]
As in~\cite{BJR}, starting from two random vertices easily gives a corresponding
second moment bound, giving convergence in probability.
\begin{lemma}\label{nkbdd}
Let $\kf$ be a bounded hyperkernel. Then
\[
 \frac{1}{n} N_k(G(n,\kf)) \pto \Pr(|\bpk|=k)
\]
for any fixed $k$.\noproof
\end{lemma}
Of course it makes no difference whether we work with $N_k$ or
$N_{\ge k}=n-\sum_{j=1}^{k-1} N_{j}$: \refL{nkbdd} also tells us that
\begin{equation}\label{ngek}
 \frac{1}{n} N_{\ge k}(G(n,\kf)) \pto \Pr(|\bpk|\ge k).
\end{equation}

The extension to arbitrary hyperkernels
is easy from Theorem~\ref{TappC}.

\begin{lemma}\label{nkint}
Let $\kf$ be an integrable hyperkernel. Then for each fixed $k$ we have
\[
 \frac{1}{n} N_{\ge k}(G(n,\kf)) \pto \rhogek(\kf).
\]
\end{lemma}
\begin{proof}
As in~\cite{BJR}, we simply approximate $\kf$ by bounded hyperkernels.
For $M>0$ let $\kfM$ be the truncated hyperkernel defined by \eqref{trunc}.

Let $k\ge 1$ be fixed, and let $\eps>0$ be arbitrary.
From  monotone convergence and integrability,
\[
 \lim_{M\to\infty} \sum_{\kx\ge 2} \int_{\sss^\kx}r\kaM_\kx = 
  \sum_{\kx\ge 2} \int_{\sss^\kx}r\ka_\kx < \infty,
\]
so for $M$ large enough we have
\[
 \Delta = \sum_{\kx\ge 2}  \int_{\sss^\kx}r(\ka_\kx-\kaM_\kx) \le \eps^2/(6k),
\]
say.
By Theorem~\ref{TappC}(i), increasing $M$ if necessary, we may also assume
that
\begin{equation}\label{rcl}
 \rhogek(\kfM)\ge \rhogek(\kf)-\eps/3.
\end{equation}

Since $\kfM\le \kf$ holds pointwise, we may couple the hypergraphs
$H_n'$ and $H_n$
associated to $G(n,\kfM)$ and $G(n,\kf)$ so that $H_n'\subseteq H_n$.
Recall that $G(n,\kf)$ is produced from $H_n$ by replacing each
hyperedge $E$ with $r$ vertices by an $r$-clique. However, as noted earlier,
if we form $G_n$ from $H_n$ by replacing each $E$ by any connected simple
graph on the same set of vertices, then $G_n$ and $G(n,\kf)$ will have exactly the same
component structure, and in particular $N_{\ge k}(G_n)=N_{\ge k}(G(n,\kf))$.
Let us form $G_n$ and $G_n'$ in this way from $H_n$ and $H_n'$, replacing
any hyperedge with $r$ vertices by some tree on the same set
of vertices. Recalling that $H_n'\subseteq H_n$,
we may of course assume that $G_n'\subseteq G_n$.

Writing $e_r(H)$
for the number of $r$-vertex hyperedges in a
hypergraph $H$, 
\begin{equation*}
  \begin{split}
 \E\bb{ |E(G_n)\setminus E(G_n')| } &\le \sum_{r\ge 2} (r-1)\E\bb{e_r(H_n)-e_r(H_n')} 
\\
 &\le \sum_{r\ge 2} {(r-1)n}\int_{\sss^r} (\ka_r-\kaM_r)
\le n\Delta.
  \end{split}
\end{equation*}

Hence,
\[
 \Pr\bb{|E(G_n)\setminus E(G_n')| \ge \eps n/6k } \le n\Delta/(\eps n/6k) \le \eps.
\]
Recalling that $G_n'\subseteq G_n$ and noting that adding one edge to a graph cannot
change $N_{\ge k}$ by more than $2k$, we see that with probability at least $1-\eps$
we have
\[
 \bm{ N_{\ge k}(G(n,\kfM)) - N_{\ge k}(G(n,\kf)) }
 = \bm{ N_{\ge k}(G_n')-N_{\ge k}(G_n) }
 \le 2k(\eps n/6k) = \eps n/3.
\]
Applying Lemma~\ref{nkbdd} (or rather~\eqref{ngek})
to the bounded hyperkernel $\kfM$, we have
$\frac{1}{n} N_{\ge k}(G(n,\kfM))\pto \rhogek(\kfM)$. Using \eqref{rcl}
it follows that when $n$ is large enough,
with probability at least $1-2\eps$, say, we have
$|\frac{1}{n} N_{\ge k}(G(n,\kf))- \rhogek(\kf)|\le \eps$.
Since $\eps>0$ was arbitrary, we thus have
$\frac{1}{n} N_{\ge k}(G(n,\kf)) \pto  \rhogek(\kf)$ as required.
\end{proof}

\section{The giant component}\label{sec_giant}

The local coupling results of the previous section easily give us the
`right' number of vertices in large components. As usual, we
will pass from this to a giant component by using the `sprinkling'
method of Erd\H os and R\'enyi~\cite{ER_evol}, first uncovering the bulk of the edges,
and then using the remaining `sprinkled' edges to join
up the large components.
The following lemma gathers together the relevant consequences
of the results in the previous section.

\begin{lemma}\label{l_easypart}
Let $\kf=(\ka_r)$ be an integrable hyperkernel, and let $G_n=G(n,\kf)$. 
Then $C_1(G_n)\le \rho(\kf)n +\op(n)$.
Furthermore, given any $\eps>0$, there is a $\delta>0$
and a function $\omega=\omega(n)\to\infty$
such that
\begin{equation}\label{nbp}
 N_{\ge \omega}(G_n') \ge (\rho(\kf)-\eps)n
\end{equation}
holds whp, where $G_n'=G(n,(1-\delta)\kf)$.
\end{lemma}

\begin{proof}
From Lemma~\ref{nkint} we have $\frac{1}{n} N_{\ge k}(G_n)\pto \rhogek(\kf)$
for each fixed $k$.
Since $\rhogek(\kf)\to \rho(\kf)$ as $k\to\infty$,
it follows that for some $\omega=\omega(n)\to \infty$ we have
\begin{equation}\label{nbig}
 \frac{1}{n} N_{\ge \omega}(G_n)\pto \rho(\kf);
\end{equation}
we may and shall assume that $\omega=o(n)$.
Since $C_1(G_n)\le \max\{\omega,N_{\ge \omega}(G_n)\}$, the first statement
of the lemma follows.

For the second, we may of course assume that $\rho(\kf)>\eps$; otherwise,
there is nothing to prove.
As $\delta\to 0$, from Theorem~\ref{TappB}(i) we have
$\rho((1-\delta)\kf)\to \rho(\kf)$. Fix $\delta>0$
with $\rho((1-\delta)\kf)>\rho(\kf)-\eps/2$, and 
let $G_n'=G(n,(1-\delta)\kf)$.
Applying \eqref{nbig} to $G_n'$, there is some
$\omega=\omega(n)\to\infty$ such that
\[
 N=N_{\ge\omega}(G_n')= \rho((1-\delta)\kf)n+\opn \ge (\rho(\kf)-\eps/2)n+\opn,
\]
which implies \eqref{nbp}.
\end{proof}

In the light of Lemma~\ref{l_easypart},
and writing $G_n$ for $G(n,\kf)$, to prove Theorem~\ref{th1}
it suffices to show that
if $\kf$ is irreducible, then 
for any $\eps>0$ we have
\begin{equation}\label{low}
 C_1(G_n) \ge (\rho(\kf)-2\eps)n
\end{equation}
whp; then $C_1(G_n)/n\pto \rho(\kf)$ as required.
Also, from \eqref{nbig} and the fact that
$C_1(G_n)+C_2(G_n)\le \max\{2\omega,N_{\ge\omega}(G_n)+\omega\}$,
we obtain $C_2(G_n)=\op(n)$ as claimed.

Since $(1-\delta)\kf\le \kf$, there is a natural coupling of 
the graphs $G_n'$ and $G_n$ appearing in Lemma~\ref{l_easypart}
in which $G_n'\subseteq G_n$ always holds.
Our aim is to show that, whp, in passing from $G_n'$ to $G_n$, the extra `sprinkled'
edges join up almost all of the $N$ vertices of $G_n'$ in `large' components
(those of size at least $\omega$) into a single component.

Unfortunately, we have to uncover
the vertex types before sprinkling, so we do not have the usual
independence between the bulk and sprinkled edges. A similar problem
arose in Bollob\'as, Borgs, Chayes and Riordan~\cite{QRperc}
in the graph context, as opposed to the present hypergraph
context. It turns out that we can easily reduce to the graph
case, and thus apply a lemma from~\cite{QRperc}. This needs
a little setting up, however. Here it will be convenient
to take $\sss=[0,1]$ with $\mu$ Lebesgue measure;
as noted in Section~\ref{sec_ir}, this loses no generality.

Let $f$ be a bounded symmetric measurable function $f:[0,1]^2\to \RR$.
Following Frieze and Kannan~\cite{FKquick},
the {\em cut norm} $\cn{f}$ of $f$ is defined by
\[
 \cn{f} = \sup_{S,T\subseteq [0,1]} \left| \int_{S\times T} f(x,y)\dd x\dd y\right|,
\]
where the supremum is taken over all pairs of measurable sets.
Note that $\cn{f}\le\norm{f}_1$, since the integral
above is bounded by $\int_{S\times T}|f|\le \int_{[0,1]^2}|f|$.

Given a kernel $\ka$ on $[0,1]$ and a measurable
function $\varphi:[0,1]\to [0,1]$, let $\kaphi$ be the kernel defined by
\[
 \kaphi(x,y) = \ka(\varphi(x),\varphi(y)).
\]
If $\varphi$ is a measure-preserving bijection, then $\kaphi$ is a {\em rearrangement}
of $\ka$. (One can also consider measure-preserving bijections between
subsets of $[0,1]$ with full measure; it makes no difference.)
We write $\ka\sim\ka'$ if $\ka'$ is a rearrangement of $\ka$.

Given two kernels $\ka$, $\ka'$ on $[0,1]$,
the {\em cut metric} of
Borgs, Chayes, Lov\'asz, S\'os and Vesztergombi~\cite{BCLSV:1}
is defined by
\[
 \dcut(\ka,\ka') = \inf_{\ka''\sim \ka'} \cn{\ka-\ka''}.
\]
Note that this is a pseudo-metric rather than a metric, as we can
have $\dcut(\ka,\ka')=0$ for different kernels.
(Probabilistically, it is probably more natural to consider couplings
between kernels as in~\cite{BCLSV:1}, rather than rearrangements, but this
is harder to describe briefly and turns out to make no difference.)

Let $\Mn$ be a symmetric $n$-by-$n$ matrix with non-negative entries $\Mij$, which
we may think of as a (dense) weighted graph. There is a piecewise-constant kernel
$\ka_{\Mn}$ associated to $\Mn$; this simply takes the value $\Mij$
on the square $((i-1)/n,i/n]\times ((j-1)/n,j/n]$, $1\le i,j\le n$.
There is also a sparse random graph $G(\Mn)$ associated to $\Mn$;
this is the graph on $[n]$ in which edges are present independently,
and the probability that $ij$ is an edge is $\Mij/n$. (If $\Mn$ has
non-zero diagonal entries then $G(\Mn)$ may contain loops. These
are irrelevant here.)

The main result of Bollob\'as, Borgs, Chayes and Riordan~\cite{QRperc}
is that if $\ka$ is 
an irreducible bounded kernel
and $(\Mn)$ is a sequence of matrices with uniformly bounded
entries such that $\dcut(\ka_{\Mn},\ka)\to 0$, then the normalized
size of the giant component in $G(\Mn)$ converges in probability to $\rho(\ka)$.
The sprinkling argument there relies on the following lemma concerning
the graph $G(A_n)$, in which edges are present independently.

\begin{lemma}\label{l_BBCR}
Let $\ka$ be an irreducible bounded kernel on $[0,1]$,
and $\delta$ and $\betamax$ positive constants.
There is a constant $c=c(\ka,\betamax,\delta)>0$
such that whenever $\Mn$ is a sequence of symmetric matrices with entries in $[0,\betamax]$
with $\dcut(\ka_{\Mn},\ka)\to 0$, then for sufficiently large
$n$ we have
\[
 \Pr(V_n\sim_{G(\Mn)} V_n') \ge 1-\exp(-cn)
\]
for all disjoint $V_n$, $V_n'\subset [n]$ with $|V_n|,|V_n'|\ge \delta n$,
where $V_n\sim_{G(\Mn)} V_n'$ denotes the event that $G(\Mn)$ contains a path starting
in $V_n$ and ending in $V_n'$.\noproof
\end{lemma}

In fact, this lemma is not stated explicitly in~\cite{QRperc},
but this is exactly the content of the end of Section 3 there;
%, from the top of page 19.
for an explicit statement and proof of (a stronger version of)
this lemma see~\cite[Lemma 2.14]{cutbr}.

We shall apply Lemma~\ref{l_BBCR} to graphs $G(\Mn)$ corresponding
to (subgraphs of) $G(n,\delta\kf)$, where $\delta$ is as
in Lemma~\ref{l_easypart}. To achieve independence
between edges, we shall simply take only one edge from each hyperedge.
Unfortunately, the problem of conditioning on the $x_i$ still remains;
we shall return to this shortly.

\begin{num_definition}\label{Gnkdef}
Let $\kf$ be an integrable hyperkernel and let $H_n$ be the Poisson (multi-)hypergraph
corresponding to $G(n,\kf)$.
Given the sequence $\vx=(x_1,\ldots,x_n)$,
let $\Gt(n,\kf,\vx)$ be the random (multi-)graph formed
from $H_n$ by replacing each $r$-vertex hyperedge $E$ by a single edge,
chosen uniformly at random from the $\binom{r}{2}$ edges
corresponding to $E$.
\end{num_definition}
With $\vx$ fixed, the numbers of copies of each edge $E$ in $H_n$
are independent Poisson random variables. From basic properties of
Poisson processes, it follows that, with $\vx$ fixed, the number of copies 
of each edge $ij$ in $\Gt(n,\kf,\vx)$ are also independent Poisson
random variables. Our next aim is to calculate the edge probabilities
in $\Gt(n,\kf,\vx)$.

As usual, we write $a\fall{b}$ for the {\em falling factorial} $a(a-1)\cdots(a-b+1)$.
Given $x_1,\dots,x_n$ and distinct $i,j\in [n]$, for $r\ge 2$ let
\begin{equation*}
  \kaxrij=n^{-(r-2)}\sum\ka_r(x_i,x_j,x_{k_3},\dots,x_{k_r}),
\end{equation*}
where the sum runs over all $(n-2)\fall{r-2}$ sequences $k_3,\dots,k_r$ of
distinct indices in $[n]\setminus\set{i,j}$,
and let $\taux$ be the $n$-by-$n$ matrix with entries
\begin{equation}\label{thdef}
  \tauxij=2\sum_{r\ge2} \kaxrij
\end{equation}
for $i\ne j$ and $\tauxij=0$ if $i=j$.

With $\vx$ given, the expected number of $r$-vertex hyperedges in $H_n$ containing $ij$
is $r(r-1)\kaxrij/n$. Hence the expected number
of $ij$ edges in $\Gt(n,\kf,\vx)$ is exactly $\tauxij/n$.
Now $\tauxij$ clearly depends on $x_i$ and $x_j$. Unfortunately,
it also depends on all the other $x_k$. The next lemma
will show that the latter dependence can be neglected.

Set
\[
 \ka_\kx(x,y,*) = \int_{\sss^{\kx-2}}\ka_\kx(x,y,x_3,x_4,\ldots,x_\kx) \dd\mu(x_3)\cdots\dd\mu(x_\kx),
\]
and let $\tau$ be the `re-scaled' edge kernel defined by
\begin{equation}\label{tdef}
  \tau(x,y) = 2\sum_{\kx\ge 2} \ka_\kx(x,y,*).
\end{equation}
Comparing with the formula \eqref{ce} for $\kae(x,y)$,
note that we have divided each term in the sum in \eqref{ce}
by $\binom{\kx}{2}$, the number of edges in $K_\kx$. 
Note that
\begin{equation}\label{tkae}
  \tau(x,y)=0 \iff \kae(x,y)=0.
\end{equation}

Recall that $\kaxrij$ and $\tauxij$ depend on the random
sequence $\vx$. In the next lemma, the expectation
is over the random choice of $\vx$; no graphs appear
at this stage.
\begin{lemma} \label{LA} 
Let $\kf=(\ka_r)_{r\ge2}$ be an integrable hyperkernel.
Then 
\begin{equation}\label{lak}
\E\frac1{n^2}\sum_{i\neq j} |\kaxrij-\ka_r(x_i,x_j,*)|
=o(1)
\end{equation}
for every $r$, and
\begin{equation}\label{lat}
\E\frac1{n^2}\sum_{i\neq j} |\tauxij-\tau(x_i,x_j)|
=o(1).
\end{equation}
\end{lemma}

\begin{proof}
  We have
\begin{equation*}
\E\bigpar{\kaxrij\mid x_i,x_j}
=(n-2)\fall{r-2}\,n^{-(r-2)}\ka_r(x_i,x_j,*).
\end{equation*}

Suppose first that $\ka_r$ is bounded.
Let
\begin{equation*}
  \begin{split}
  Y_{ij}
&=
\kaxrij-(n-2)\fall{r-2}\, n^{-(r-2)}\ka_r(x_i,x_j,*)
\\&
=n^{-(r-2)}
\sum(\ka_r(x_i,x_j,x_{k_3},\dots,x_{k_r})-\ka_r(x_i,x_j,*)),	
  \end{split}
\end{equation*}
where the sum again runs over all
$(n-2)\fall{r-2}$ sequences $k_3,\dots,k_r$ of
distinct indices in $[n]\setminus\set{i,j}$.
Given $x_i$ and $x_j$, each term in the sum has mean 0, and any two
terms with disjoint index sets \set{k_3,\dots,k_r} are independent.
Since there are $O(n^{2r-5})$ pairs of terms with overlapping 
index sets, and $\ka_r$ is bounded, we have
\begin{equation*}
  \E(Y_{ij}^2\mid x_i,x_j) =O(n^{2r-5-2(r-2)})=O(n^{-1}).
\end{equation*}
Thus $\E Y_{ij}^2 =O(n^{-1})$. Hence, by the 
Cauchy--Schwarz inequality, $\E|Y_{ij}|\le(\E Y_{ij}^2)^{1/2}=O(n^{-1/2})$,
and so
\begin{equation*}
\E\frac1{n^2}\sum_{i\neq j} |\kaxrij-\ka_r(x_i,x_j,*)|
=
\frac1{n^2}\sum_{i\neq j}\E |Y_{ij}|+O(1/n)
=O(n^ {-1/2}).
\end{equation*}
This proves \eqref{lak} and thus \eqref{lat} for bounded hyperkernels.

For general hyperkernels, we use truncation and define $\kaM_r$ by
\eqref{trunc}.
For the corresponding $\kaxMrij$, $\tauxM$ and $\tauM$,
\begin{gather*}
\E\frac1{n^2}\sum_{i\neq j} |\kaxrij-\kaxMrij|
\le\int(\ka_r-\kaM_r),
\\
\E\frac1{n^2}\sum_{i\neq j} |\ka_r(x_i,x_j,*)-\kaM_r(x_i,x_j,*)|  
\le\int(\ka_r-\kaM_r),
\intertext{and thus}
\E\frac1{n^2}\sum_{i\neq j} |\tauxij-\tauxMij|
\le2\sum_r\int(\ka_r-\kaM_r),
\\
\E\frac1{n^2}\sum_{i\neq j} |\tau(x_i,x_j)-\tauM(x_i,x_j)|  
\le2\sum_r\int(\ka_r-\kaM_r).
\end{gather*}
Since $(\ka_r)$ is integrable, given any $\eps>0$ we can make these
expected differences
less than $\eps$ by choosing $M$ large enough, and the result follows
from the bounded case. 
\end{proof}

With the preparation above we are now ready to prove Theorem~\ref{th1}.

\begin{proof}[Proof of Theorem~\ref{th1}]
We assume without loss of generality that $\sss=[0,1]$, with $\mu$
Lebesgue measure.

Let $\kfp=(\ka'_F)_{F\in \F}$ be an irreducible,
integrable kernel family,
let $\kf=(\ka_r)_{r\ge2}$ be the corresponding hyperkernel, given by \eqref{hk},
and let $\eps>0$.
As noted after \refL{l_easypart}, in the light of this lemma,
it suffices to prove the lower bound \eqref{low} on $C_1(G(n,\kf))$.
We may and shall assume that 
$\rho(\kf)>0$ and $\eps<\rho(\kf)/10$, say.

Let $\delta>0$ and $\omega=\omega(n)$
be as in \refL{l_easypart}, and let $H_n$, $H_n'$
and $\tH_n$ be the Poisson multi-hypergraphs associated
to the hyperkernels $\kf$, $(1-\delta)\kf$ and $\delta\kf$,
respectively.
Using the same vertex types $\vx=(x_1,\ldots,x_n)$ for all three hypergraphs,
there is a natural coupling in which $H_n = H_n'\cup \tH_n$,
with $H_n'$ and $\tH_n$ {\em conditionally} independent given $\vx$.

Define $\taux$ and $\tau$ by \eqref{thdef} and \eqref{tdef}, respectively,
starting from the integrable hyperkernel $\delta\kf$.
Note that $\tau$ is a kernel on $[0,1]$, while $\taux$ is
an $n$-by-$n$ matrix {\em that depends on $\vx$}.
Recall from \eqref{tkae} that $\tau(x,y)=0$ if and only if $\kae(x,y)=0$,
so $\tau$ is irreducible.
In order to be able to apply \refL{l_BBCR}, we would like
to work with a bounded kernel and matrices that are bounded uniformly
in $n$. We achieve this simply by considering $\Btaux=(\Btauxij)$ and $\Btau$
defined by
%\marginal{Would like something between $\backslash$bar and $\backslash$overline on $A$ and $B$.}
\[
 \Btauxij= \min\{\tauxij,1\} \quad\hbox{ and }\quad\Btau(x,y)=\min\{\tau(x,y),1\}.
\]

Let $\N$ be the (random)
`sampled' matrix corresponding to $\tau$, defined by
\[
 \Nij = \tau(x_i,x_j),
\]
and let $\BN$ be the corresponding matrix associated to $\Btau$.
The second statement of Lemma~\ref{LA} tells us exactly that 
\[
 \E\frac{1}{n^2} \sum_{i\ne j} |\tauxij-\Nij| =o(1),
\]
where the expectation is over the random choice of $\vx$.
Since $|\Btauxij-\BNij|\le |\tauxij-\Nij|$ for $i\ne j$,
while $|\Btauxii-\BNii|\le 1$, it follows that
\[
 \E\frac{1}{n^2} \sum_{i,j} |\Btauxij-\BNij| =o(1),
\]
or, equivalently, that $\E \norm{\ka_{\Btaux}-\ka_\BN}_1=o(1)$,
where we write $\ka_M$ for the piecewise constant kernel
$\ka_M:[0,1]^2\to\RR$ associated
to a matrix $M$.

Since $\dcut(\ka_1,\ka_2)\le \cn{\ka_1-\ka_2}\le \norm{\ka_1-\ka_2}_1$,
it follows that
$\E\dcut(\ka_{\Btaux},\ka_\BN)\to 0$, and hence that 
$\dcut(\ka_{\Btaux},\ka_\BN)\pto 0$.
Coupling the random sequences $\vx$ for different $n$ appropriately, we may
and shall assume that
\begin{equation}\label{c1}
  \dcut(\ka_{\Btaux},\ka_\BN)\to 0
\end{equation}
almost surely.

Since $\Btau$ is a bounded kernel on $[0,1]$, i.e.,
a `graphon' in the terminology of~\cite{BCLSV:1},
Theorem 4.7 of Borgs, Chayes, Lov\'asz, S\'os and
Vesztergombi~\cite{BCLSV:1} tells us that
with probability at least $1-e^{-n^2/(2\log_2 n)}=1-o(1)$,
we have $\dcut(\ka_\BN,\Btau)\le 10\sup\Btau/\sqrt{\log_2 n}=o(1)$.
It follows that $\dcut(\ka_\BN,\Btau)\to 0$ both in probability
and almost surely.
Using \eqref{c1}, we see that
\begin{equation}\label{c2}
  \dcut(\ka_{\Btaux}, \Btau) \to 0
\end{equation}
almost surely. Note that $\ka_{\Btaux}$ depends on the sequences $\vx$.

Let $G_n'$ and $G_n$ be the simple graphs underlying $H_n'$ and $H_n$.
From \refL{l_easypart}, \eqref{nbp} holds whp.
For the rest of the proof, {\em we condition on $\vx$ and on $H_n'$}.
We assume, as we may, that \eqref{nbp} holds
for all large enough $n$, and that \eqref{c2} holds.
It suffices to show that with {\em conditional} probability
$1-o(1)$ we have $C_1(G_n)\ge (\rho(\kf)-2\eps) n$.
Recall that, given $\vx$, the (multi-)hypergraphs $H_n'$ and $\tH_n$
are independent, so after our conditioning (on $\vx$ and $H_n'$),
the hypergraph $\tH_n$ is formed by selecting
each $r$-tuple $v_1,\ldots,v_r$ to be an edge independently,
with probability $\delta\ka_r(x_{v_1},\ldots,x_{v_r})/n^{r-1}$.

Let $\Gt_n=\Gt(n,\delta\kf,\vx)$ be the random (multi-)graph
defined from $\tH_n$ by taking one edge from each hyperedge as
in Definition~\ref{Gnkdef},
noting that $G_n'\cup \Gt_n\subseteq G_n$.
Since we have conditioned on $\vx$ (and $G_n'$),
as noted after Definition~\ref{Gnkdef}, each possible edge
$ij$ is present in $\Gt_n$ independently. In the multi-graph
version, the number of copies of $ij$ is Poisson with mean
$\tauxij/n$.
Passing to a subgraph, we shall take instead
the number of copies to be Poisson with mean $\Btauxij/n$.
Since this mean is $O(1/n)$, the probability that one
or more copies of $ij$ is present is $\tauxpij/n$, where
$\tauxpij=\Btauxij+O(1/n)$.
Since $\dcut(\ka_{\tauxp},\ka_{\Btaux})=O(1/n)=o(1)$,
we have $\dcut(\ka_{\tauxp},\Btau)\to 0$.
Since $\Btau$ is an irreducible bounded kernel,
the (simple graphs underlying) $\Gt_n$ satisfy the assumptions
of \refL{l_BBCR}, so there is a constant $c>0$
such that for any two set $V_n$, $V_n'$ of at least $\eps n/2$
vertices of $\Gt_n$, the probability
that $V_n$ and $V_n'$ are {\em not} joined by a path
in $\Gt_n$ is at most $e^{-cn}$.

Recall that we have conditioned on $G_n'$, assuming \eqref{nbp}.
Suppose also that $C_1(G_n)\le (\rho(\kf)-2\eps)n$.
Then there is a partition $(V_1,V_2)$ of the set of vertices of $G_n'$
in large components in $G_n'$ with $|V_1|,|V_2|\ge \eps n$ such that
there is no path in $G_n$ from $V_1$ to $V_2$.
Let us call such partition $(V_1,V_2)$ a {\em bad partition}.
Having conditioned on $G_n'$, noting that in any potential bad partition $V_1$
must be a union of large components of $G_n'$,
the number of possible choices
for $(V_1,V_2)$ is at most $2^{n/\omega}=e^{o(n)}$.
On the other hand, since $\Gt_n\subseteq G_n$, the probability that
any given partition is bad is at most $e^{-cn}$, so the expected number
of bad partitions is $o(1)$, and whp there is no bad partition.
Thus $C_1(G_n) > (\rho(\kf)-2\eps)n$ holds whp, as required.
\end{proof}

\begin{remark}\label{R_red}
The restriction to irreducible kernel families in \refT{th1} is of course
necessary; roughly speaking, if $\kf$ is reducible, then our graph $G(n,\kf)$
falls into two or more parts. \refL{l_easypart} still applies to show that
we have $\rho(\kf)n+\op(n)$ vertices in large components,
but it may be that two or more parts have giant components,
each of smaller order than $\rho(\kf)n$.

More precisely, let $\kf$ be a reducible, integrable kernel family.
Thus the edge kernel $\kae$ is reducible. By Lemma 5.17 of~\cite{BJR},
there is a partition $\sss=\bigcup_{i=0}^N \sss_i$, $N\le \infty$,
of our ground space $\sss$ (usually $[0,1]$) such that
each $\sss_i$ is measurable, the restriction of $\kae$ to $\sss_i$
is irreducible (in the natural sense), and, apart from a measure zero
set, $\kae$ is zero off $\bigcup_{i=1}^N \sss_i\times \sss_i$.

Suppressing the dependence on $n$, let $G_i$ be the subgraph
of $G(n,\kf)$ induced by the vertices with types in $\sss_i$.
Since the vertex types are \iid, the probability that $G(n,\kf)$ contains
any edges other than those of $\bigcup_{i\ge 1}G_i$ is 0.
Now $G_i$ has a random number $n_i$ of vertices,
with a binomial $\Bi(n,\mu(\sss_i))$ distribution, which is concentrated around its mean.
Given $n_i$, the graph $G_i$ is another instance of our model.

Let $a_i=\int_{\sss_i} \rho_\kf(x)\dd\mu(x)$, so that
$\sum_i a_i=\rho(\kf)<1$.
From the remarks above it is easy to check that \refT{th1} gives
$C_1(G_i)/n\pto a_i$ and $C_2(G_i)=\op(n)$; we omit the details.
Sorting the $a_i$ into decreasing order
$\ha_1,\ha_2,\ldots$, it follows that
$C_i(G(n,\kf))=\ha_i n+\op(n)$ for each fixed (finite) $1\le i\le N$,
in particular, for $i=1$ and $i=2$.
\end{remark}

\section{Disconnected atoms and percolation}\label{Sdisconnected}

One of the most studied features of the various inhomogeneous network models
is their `robustness' under random failures, and in particular, the critical
point for site or bond percolation on these random graphs.
For example, this property of the Barab\'asi--Albert~\cite{BAsc} model
was studied experimentally by Barab\'asi, Albert and Jeong~\cite{AJBattack},
heuristically by
Callaway, Newman, Strogatz and Watts~\cite{CNSW} (see also~\cite{BAsurv}) and
Cohen, Erez, ben-Avraham and Havlin~\cite{CohenRobust},
and rigorously in~\cite{robust,Rsmall}.
In the present context, given $0<p<1$, we would like to study
the random subgraphs $\Gp(n,\kf)$ and $\Gpp(n,\kf)$ of $G(n,\kf)$ obtained
by deleting edges or vertices respectively, keeping each edge
or vertex with probability $p$, independently of the others. In the
edge-only model of~\cite{BJR}, these graphs were essentially equivalent
to other instances of the same model:  roughly speaking, $\Gp(n,\ka) \isom G(n,p\ka)$
and $\Gpp(n,\ka)\isom G(pn,p\ka)$. (For precise statements, see~\cite[Section 4]{BJR}.)

Here, the situation is a little more complex. When we delete edges randomly from $G(n,\kf)$,
it may be that what is left of a particular atom $F$ is disconnected. This forces
us to consider {\em generalized kernel families} $(\ka_F)_{F\in\G}$ with one kernel
$\ka_F$ for each $F\in \G$, where the set $\G$ consists of one representative
of each isomorphism class of finite (not necessarily connected) graphs.

Rather than present a formal statement, let us consider a particular example.
Suppose that $\kf$ is the generalized kernel family with only one kernel
$\ka_F$, corresponding to the disjoint union $F$ of $K_3$ and $K_2$.
Let $\kf'$ be the kernel family with two kernels,
\[
 \ka_3(x,y,z) = \int_{\sss^2} \ka_F(x,y,z,u,v)\dd\mu(u)\dd\mu(v),
\]
corresponding to $K_3$ and 
\[
 \ka_2(u,v) = \int_{\sss^3} \ka_F(x,y,z,u,v)\dd\mu(x)\dd\mu(y)\dd\mu(z)
\]
for $K_2$. Then $G(n,\kf)$ and $G(n,\kf')$ are clearly very similar;
the main differences are that $G(n,\kf)$ contains exactly the same number
of added triangles and $K_2$s, whereas in $G(n,\kf')$ the numbers
are only asymptotically equal, and that in $G(n,\kf)$ a triangle
and a $K_2$ added in one step are necessarily disjoint.
Since almost all pairs of triangles and $K_2$s in $G(n,\kf')$ are disjoint anyway,
it is not hard to check that $G(n,\kf)$ and $G(n,\kf')$ are `locally equivalent',
in that the neighbourhoods of a random vertex
in the two graphs can be coupled to agree up to a fixed size whp.

More generally, given a generalized kernel family $\kf=(\ka_F)_{F\in\G}$,
let $\kf'$ be the kernel family obtained by replacing each kernel
$\ka_F$ by one kernel for each component $F'$ of $F$, obtained
by integrating over variables corresponding to vertices of $F\setminus F'$
as above. This may produce several new kernels for a given connected $F'$;
we of course simply add these together to produce a single kernel $\ka'_{F'}$.
Note that
\[
 \sum_{F'} |F'|\int_{\sss^{|F'|}} \ka'_{F'} = \sum_F |F| \int_{\sss^{|F|}} \ka_F,
\]
so if $\kf$ is integrable, then so is $\kf'$.
Although $G(n,\kf)$ and $G(n,\kf')$ are not exactly equivalent, the truncation
and local approximation arguments used to prove \refT{th1} carry over easily
to give the following result.
\begin{theorem}\label{th_disc}
Let $\kf=(\ka_F)_{F\in\G}$ be a generalized kernel family,
let $\kf'$ be the corresponding kernel family as defined above,
and let $\kf''=\bkf'$ be the hyperkernel corresponding to $\kf'$,
defined by \eqref{hk}.
If $\kf'$ is irreducible, then
\begin{equation*}%\label{disc}
 C_1(G(n,\kf))=\rho(\kf'')n + \op(n),
\end{equation*}
and $C_2(G(n,\kf))=\op(n)$.\noproof
\end{theorem}
Note that the hyperkernel $\kf''$ corresponding to $\kf'$ is obtained by replacing
each (now connected, as before) atom $F'$ by a clique; this corresponds
to replacing each {\em component} of an atom $F$ in $G(n,\kf)$ by a clique.

Turning to bond percolation on $G(n,\kf)$, i.e., to the study of the random
subgraph $\Gp(n,\kf)$ of $G(n,\kf)$,
let $\kfq$ be the kernel family obtained by replacing each kernel $\ka_F$
by $2^{e(F)}$ kernels $\ka_{F'} =p^{e(F')}(1-p)^{e(F)-e(F')} \ka_F$,
one for each spanning subgraph of $F$. (As before, we then
combine kernels corresponding to isomorphic graphs $F'$.)
Working work with the
Poisson multigraph formulation of our model,
the graphs $\Gp(n,\kf)$ and $G(n,\kfq)$ have exactly the same distribution.
This observation
and \refT{th_disc} allow us (in principle, at least) to decide
whether $\Gp(n,\kf)$ has a giant component, i.e., to find the critical point
for bond percolation on $G(n,\kf)$. 

Let us illustrate this with the very simple special case in which each
kernel $\ka_F$, $F\in\G$, is constant, say $\ka_F=c_F$. We assume that 
$\kf$ is integrable, i.e., that $\sum_F |F|c_F<\infty$.
In this case each kernel $\kap_F$ making up $\kfq$ is also constant, and the same
applies to the hyperkernel $\kf''$ corresponding to $\kfq$.
Hence, from the remarks above and \eqref{unif},
$\Gp(n,\kf)$ has a giant component if and only if the asymptotic edge density
$\xie(\kf'')$ of the hyperkernel $\kf''$ is at least $1/2$.
Since we obtain $\kf''$
by first taking random subgraphs of our original atoms $F$,
and then replacing each component by a clique, we see that
\[
 \xie(\kf'') = \sum_{F\in\F} c_F \theta_F(p),
\]
where $\theta_F(p)$ is the expected number of unordered pairs of distinct
vertices of $F$ that lie in the same component of the random subgraph $\Fp$
of $F$ obtained by keeping each edge with probability $p$, independently of the others.
Alternatively,
\[
 2\xie(\kf'') = \sum_{F\in \F} c_F |F| (\chi(\Fp)-1),
\]
where $\chi(\Fp)$ is the {\em susceptibility} of $\Fp$, i.e., the expected size
of the component of a random vertex of $\Fp$.
If we have only a finite number of non-zero $c_F$, then $\xie(\kf'')$
may be evaluated as a polynomial
in $p$, and the critical point found exactly.

Turning to site percolation, there is a similar reduction to another instance of our model,
most easily described by modifying the type space. Indeed, we add a new type $\star$
corresponding to deleted vertices, and set $\mu'(\star)=1-p$. Setting $\mu'(A)=p\mu(A)$
for $A\subset \sss$, we obtain a probability measure $\mu'$ on $\sss'=\sss\cup\{\star\}$.
Replacing each kernel $\ka_F$ by $2^{|F|}$ kernels $\ka_{F'}$ on $\sss'$ defined
appropriately (with $F'$ corresponding to the subgraph of $F$ spanned by the non-deleted vertices),
one can show that $\Gpp(n,\kf)$ is very close to
(in the Poisson version, identical to)
a suitable instance $G(n',\kf')$
of our model, where $n'$ is now random but concentrated around its mean $pn$.
In the first instance $\kf'$ may include kernels for disconnected graphs,
but as above we can find an asymptotically
equivalent kernel family involving only connected graphs.
In this way one can find the asymptotic size of any giant component in $\Gpp(n,\kf)$; we omit the
mathematically straightforward but notationally complex details.

\section{Vertex degrees}\label{sec_degrees}

Heuristically, the vertex degrees in $\gnk$ can be described as follows.
Consider a vertex $v$ and condition on its type $x_v$. 
The number of atoms that contain $v$ then is
asymptotically Poisson with a certain mean depending on $\kf$ and $x_v$.
However, each atom may add several edges to the vertex $v$, and
thus the asymptotic distribution of the vertex degree is compound
Poisson (see below for a definition).
Moreover, this compound Poisson distribution typically depends on the
type $x_v$, so the final result is that, asymptotically, the vertex
degrees have a mixed compound Poisson distribution. In this
section we shall make this precise and rigorous.

We begin with some definitions. 
If $\ll$ is a finite measure on $\bbN$, then $\cpol$, the
\emph{compound Poisson distribution with intensity $\ll$}, is defined
as the distribution of $\sumj jX_j$, where $X_j\sim\Po(\ll\set j)$ are
independent Poisson random variables.
Equivalently, $\cpol$ is the distribution of the sum $\sum_\nu\xi_\nu$
of the points of a Poisson process \set{\xi_\nu} on $\bbN$ with intensity $\ll$,
regarded as a multiset.
(The latter definition generalizes
to arbitrary measures $\ll$ on $(0,\infty)$ such that
$\int_0^\infty t\wedge 1\dd\ll(t)<\infty$, but we consider in this
paper only the integer case.)
Since $X_j$ has \pgf{} 
$\E z^{X_j}=e^{\ll\set j (z-1)}$,
$\cpol$ has \pgf
\begin{equation*}
  \gfcpol(z)
=
\E z^{\sumj j X_j}
=
\E \prodj z^{j X_j}
=
\prodj e^{\ll\set j (z^j-1)}
=
 e^{\sumj\ll\set j (z^j-1)},
\end{equation*}
whenever this is defined, which it certainly is for $|z|\le1$.

If $\La$ is a random finite measure on $\bbN$,
then $\mcpol$ denotes the
corresponding \emph{mixed compound Poisson distribution}.
From now on, for each $x\in\sss$, $\ll_x$ will be a finite measure on $\bbN$, 
depending measurably on $x$. We shall write $\La$ for the corresponding
random measure on $\bbN$, obtained by choosing $x$ from $\sss$
according to the distribution $\mu$ and then taking $\la_x$.
Thus $\mcpol$ is
defined by the point probabilities
\begin{equation*}
\mcpol\set i
=\ints\cpolx\set i\dd\mu(x)
\end{equation*}
or, equivalently, the \pgf
\begin{equation*}
  \gfmcpol(z)
=
\ints\gfcpolx(z)\dd\mu(x)
=
\ints e^{\sumj\ll_x\set j (z^j-1)} \dd\mu(x).
\end{equation*}

\begin{remark}\label{Rcpo}
Since we have assumed that $\ll$ is a finite measure,
$\E\sum_j X_j=\ll(\bbN)<\infty$; thus a.s.\ $\sum_jX_j<\infty$ and
only finitely many $X_j$ are non-zero, whence $\sum_jjX_j<\infty$ a.s.
This verifies that $\cpol$ is a proper probability distribution.
On the other hand, the mean of $\cpol$ is
\begin{equation*}
\E\cpol
=\sum_jj\E X_j
=\sum_jj\ll\set j
=\int_0^\infty t\dd\ll(t),  
\end{equation*}
which may be infinite.
As a consequence,
\begin{equation}\label{rcpo}
\E\mcpol
=\int_\sss\int_0^\infty t\dd\ll_x(t)\dd\mu(x)
\le\infty.  
\end{equation}
\end{remark}

Let $\dtv$ denote the total variation distance between two random
variables, or rather their probability distributions, defined by
\begin{equation}\label{dtv}
  \dtv(X,Y)=\sup_A|\Pr(X\in A)-\Pr(Y\in A)|,
\end{equation}
where the supremum is taken over all measurable sets $A\subseteq\bbR$.
%in the relevant sample space ($\bbR$ in our case).
We shall use the following trivial upper bound on the
total variation distance between two compound Poisson distributions.
\begin{lemma}
  \label{LB}
If $\ll$ and $\ll'$ are two finite measures on $\bbN$, then
\begin{equation*}
\dtv\bigpar{\CPo(\ll),\CPo(\ll')}
\le\|\ll-\ll'\|
=\sum_j|\ll\set j  -\ll'\set j|.
\end{equation*}
\end{lemma}
\begin{proof}
  Let $X_j\sim\Po(\ll\set j)$ be as above and let
  $X'_j\sim\Po(\ll'\set j)$ be another family of independent Poisson
  variables.
We can easily couple the families so that
$\Pr(X_j\neq  X'_j) \le |\ll\set j-\ll'\set j|$ for every $j$.
                                   
Then
\begin{align*}
\dtv\bigpar{\CPo(\ll),\CPo(\ll')}
&\le\Pr\Bigpar{\sum_j jX_j\neq\sum_jjX'_j}
\le\sum_j\Pr(X_j\neq X'_j)
\\&
=\sum_j|\ll\set j  -\ll'\set j|.	
\qedhere
\end{align*}
\end{proof}

Given an integrable kernel family $\kf$
and $x\in\sss$, $F\in\F$ and $j\in V(F)=[\kF]$,
let
\begin{equation}\label{v2}
  \lxfj
=\int_{\sss^{\kF-1}} \ka_F(x_1,\dots,x_{j-1},x,x_{j+1},\dots,x_{\kF})
\dd \mu(x_1)\dotsm\dd \mu(x_{j-1})\dd\mu(x_{j+1})\dotsm\dd\mu(x_{\kF})
\end{equation}
be the (asymptotic) expected number of added copies of $F$ containing a
given vertex of type $x$ in which the given vertex corresponds to
vertex $j$ in $F$.
Let $\dfj$ be the degree of vertex $j$ in $F$, and define the measure
\begin{equation}\label{lax}
  \ll_x
=
\sum_{F\in\F}\sumjf\lxfj\delta_\dfj,
\end{equation}
where, as usual, $\delta_d$ denotes the probability measure assigning mass 1 to $d$.
Thus $\ll_x$ is a measure on $\bbN$, with point masses
%$\lxd=\ll_x\set d$ 
given by
\begin{equation}\label{lxd}
  \lxd
=
\sum_{F\in\F}\sum_{j:\dfj=d}\lxfj,
\end{equation}
the (asymptotic) expected number of atoms containing a given
vertex of type $x$ and having degree $d$ there.
From \eqref{v2}, 
$\int_\sss\lxfj\dd\mu(x)=\int_{\sss^{\kF}}\ka_F$, and thus by \eqref{lax}
\begin{equation*}
\int_\sss\|\ll_x\|\dd\mu(x)
=\sumfj\int_\sss\lxfj \dd\mu(x)
=\sum_F\kF\int_{\sss^{\kF}}\ka_F<\infty.
\end{equation*}
Consequently, $\ll_x$ is a finite measure on $\bbN$ for a.e.\ $x$, and
the mixed compound Poisson distribution $\mcpol$ is defined.

Let the random variable $D=\dn$ be the degree of any fixed vertex
in \gnk. Equivalently, by symmetry, we can take $\dn$ to be the degree of
a uniformly random vertex. Furthermore, for $\ell\ge0$, let $\nl$ be the number of
vertices with degree $\ell$ in \gnk. Then the random
sequence $(\nl/n)_{\ell=0}^\infty$ can be regarded as a (random)
  probability distribution, viz., the conditional distribution of the
  degree of a random vertex in \gnk, given this random graph.
Note that $\Pr(\dn=\ell)=\E\nl/n$.

\begin{theorem}\label{Tdegree}
  Suppose that $\kf=(\ka_F)_{F\in\F}$ is an integrable kernel family.
Then, as \ntoo,
  \begin{romenumerate}
\item\label{tda}
$\dn\dto\mcpol$, and
\item\label{tdc}
$\displaystyle%  \begin{equation*}
\E\dn\to\E\mcpol
=\sumF 2e(F)\int_{\sss^{\kF}}\ka_F
=2\xiek
\le\infty
.
$%	  \end{equation*}
\item	\label{tdb}
Moreover, for every fixed $\ell$,
\begin{equation}\label{nl}
  \nl=\mcpol\set\ell n+\op(n)
\end{equation}
and thus
$(\nl/n)_{\ell=0}^\infty\dto\mcpol$ in the space of probability
measures on $\bbN$.
  \end{romenumerate}
\end{theorem}

Note that the limit distribution exists for every integrable
kernel family, but has finite expectation only if the kernel family is
edge integrable.

As usual, \refT{Tdegree} applies to the variants of the model $G(n,\kf)$ 
discussed in Section~\ref{sec_ir}. In the proof, we shall mostly
work with the (non-Poisson) multi-graph form, where we add at most
one copy of a certain small graph $F$ with a particular vertex set,
but keep any resulting multiple edges.
\begin{proof}
Assume first that $\kf$ is a bounded kernel family, with $\ka_F\le M$
and $\ka_F=0$ if $\kF>M$.
Fix a vertex $v\in[n]$, and let $D$ be the degree of $v$.
For $F\in\F$ with $\kF\le M$ and $j\in V(F)$,
let $\nfj$ be the
number of added copies of $F$ that contain $v$ with $v$
corresponding to vertex $j$ in $F$.
Let 
\begin{equation}\label{dq}
\dq
=\sumfj\nfj\dfj;
\end{equation}
this is the number of edges added to $v$, including possible
repetitions.
Thus $D=\dq$ unless two added edges with endpoint $v$
coincide. For any other vertex $w$,
conditioned on the types $\vx=(x_1,\ldots,x_n)$, the number of atoms
containing both $v$ and $w$ is a sum $\sum_\nu I_\nu$
of independent Bernoulli variables $I_\nu\sim\Be(p_\nu)$, for $\nu$ in some index set.
For each $r=2,\dots,M$
there are $O(n^{r-2})$ such variables, each with $p_\nu=O(n^{1-r})$.
Hence,
\begin{equation*}
  \Pr\Bigpar{\sum_\nu I_\nu\ge2\Bigm|\vx}
\le\sum_{\nu_1\neq\nu_2}p_{\nu_1}p_{\nu_2}
\le\Bigpar{\sum_\nu p_\nu}^2
=O(n\qww).
\end{equation*}
%Taking the expectation over $\vx$, we find,
Since there are $n-1$ possible choices for $w$, it follows that
\begin{equation}\label{ddq}
\dtv\bigpar{(D\mid\vx),\,(\dq\mid\vx)}\le
\Pr(D\neq \dq\mid\vx)=O(n\qw)
.
\end{equation}
Hence, in proving \ref{tda}, it makes no difference whether we work with 
$\dq$ or with $D$, i.e., with the multi-graph or simple graph version of $G(n,\kf)$.

Conditioned on $\vx$, $\nfj$ is a sum of independent Bernoulli
variables $\Be(\pfjax)$ for $\ga$ in some index set $\cA\fj$, with
$\pfjax=O(n^{1-\kF})$ given by \eqref{pdef} and $|\cA\fj|=O(n^{\kF-1})$.

Let $\hlfj(\vx)=\E(\nfj\mid\vx)=\suma\pfjax$.
By a classical Poisson approximation theorem 
(see \cite[(1.8)]{BHJ}),
\begin{equation}\label{v1}
\dtv\bigpar{(\nfj\mid\vx),\,\Po(\hlfjx)}
\le\suma\pfjax^2
=
O(n^{1-\kF})
=O(n\qw).
\end{equation}
(This follows easily from the elementary $\dtv(\Be(p),\,\Po(p))\le
p^2$; see e.g.\ \cite[page 4 and Theorem 2.M]{BHJ} for history and
further results.) 
Furthermore, given $\vx$, the random variables $\nfj$ are
independent, and thus \eqref{dq} and \eqref{v1} imply that if
$\hX\fj\sim\Po(\hlfjx)$ are independent, then
\begin{equation*}
  \dtv\Bigpar{(\dq\mid\vx),\,\sumfj\dfj\hX\fj}
=O(n\qw).
\end{equation*}
Since $\sumfj\dfj\hX\fj$ has a compound Poisson distribution
$\CPo(\hLa(\vx))$ with intensity
$\hLa(\vx)=\sumfj\hlfjx\gd_{\dfj}$, we have
\begin{equation*}
  \dtv\bigpar{(\dq\mid\vx),\,\CPo(\hLa(\vx))}
=O(n\qw).  
\end{equation*}
By \eqref{ddq} and \refL{LB}, this yields
\begin{equation*}
  \begin{split}
  \dtv\bigpar{(D\mid\vx),\,\CPo(\ll_{x_v})}
&\le O(n\qw)+\|\hLa(\vx)-\ll_{x_v}\|.
%\\&
%\le\sumfj|\hlfjx-\lfj(x_v)|+O(n\qw).	
  \end{split}
\end{equation*}
In particular, for every $\ell\in\bbN$, taking $A=\set\ell$ in \eqref{dtv},
\begin{equation}\label{v4-}
 \bm{ \Pr(D=\ell\mid\vx)-\CPo(\ll_{x_v})\set\ell }
  \le O(n\qw)+\|\hLa(\vx)-\ll_{x_v}\|.
\end{equation}
Taking the expectation of both sides, and noting that
$\E\Pr(D=\ell\mid\vx)=\Pr(D=\ell)$ and
$\E\CPo(\ll_{x_v})\set\ell=\mcpol\set\ell$,
we find that
\begin{equation}\label{v4}
 \bm{ \Pr(D=\ell)-\mcpol\set\ell } \le O(n\qw)+\E\|\hLa(\vx)-\ll_{x_v}\|.
\end{equation}
We shall show that the final term is small.

By \eqref{pdef}, with $r=\kF$,
\begin{equation*}
  \hlfjx=n^{1-r}\sum\ka_F(x_{v_1},\dots,x_{v_r}),
\end{equation*}
where the sum runs over all $(n-1)\fall{r-1}$ sequences $v_1,\dots,v_r$ of
distinct elements in $[n]$ with $v_j=v$. Consequently, by \eqref{v2},
\begin{equation}\label{v3}
\E\bigpar{\hlfjx\mid x_v}
=\bigpar{1-O(n\qw)}\lfj(x_v).
  \end{equation}
Recalling that $\kf$ is bounded, it is easy to check
(as in the similar argument in the
proof of \refL{LA}) that
  \begin{equation*}
\Var\bigpar{\hlfjx\mid x_v}
=\E\bigpar{(\hlfjx-\E(\hlfjx\mid x_v))^2\mid x_v}
=O(n\qw)
  \end{equation*}
and thus, by the \CSineq{} and \eqref{v3},
\begin{equation*}
  \E\bigpar{|\hlfjx-\lfj(x_v)|\bigm| x_v}
  =O(n\qqw).
\end{equation*}
Consequently, using again that $\kf$ is bounded,
  \begin{equation}\label{esmall}
\E\|\hLa(\vx)-\ll_{x_v}\|
\le
\E\sumfj|\hlfjx-\lfj(x_v)|
=O(n\qqw) =o(1),
  \end{equation}
so
\begin{equation*}%\label{lpc}
 \|\hLa(\vx)-\ll_{x_v}\| \pto 0.
\end{equation*}
Combining \eqref{esmall} and \eqref{v4}
we see that $\Pr(D=\ell)\to\mcpol\set\ell$ for every
$\ell$, i.e., $D\dto\mcpol$, which proves \ref{tda} for
bounded $\kf$.

Next we turn to the proof of \ref{tdb}, assuming still that $\kf$ is bounded.
Fix a number $\ell\in\bbN$, and
for $v\in[n]$ let $D_v$ be the degree of $v$ in \gnk{}, and
$I_v$ the indicator $\ett{D_v=\ell}$.

Fix two distinct vertices $v$ and $w$,
let 
$\cG$ be the set of atoms that contain both $v$ and
$w$, and let 
$\tD_v$ and
$\tD_w$ be the degrees of the vertices if we delete (or ignore)
the atoms in $\cG$.
Since $\kf$ bounded,  the expected number $\E|\cG|$
of such exceptional atoms is $O(n\qw)$, 
and thus
\begin{equation*}
  \Pr(D_v\neq\tD_v),
\,
\Pr(D_w\neq\tD_w)
\le\Pr(\cG\neq\emptyset)
\le\E|\cG|
=O(n\qw).
\end{equation*}
Moreover, these bounds hold  conditional on $\vx$.
Furthermore, given $\vx$, $\tD_v$ and $\tD_w$ are independent.
Consequently, for any $\ell\in\bbN$,
\begin{equation*}
  \begin{split}
\E(I_vI_w \mid \vx)
&=
  \Pr(D_v=D_w=\ell \mid \vx)
=
  \Pr(\tD_v=\tD_w=\ell \mid \vx)+o(1)
\\&
=
  \Pr(\tD_v=\ell\mid \vx)\Pr(\tD_w=\ell\mid \vx)+o(1)
\\&
=
  \Pr(D_v=\ell\mid \vx)\Pr(D_w=\ell\mid \vx)+o(1)
\\&
=
\E(I_v\mid\vx)\E(I_w\mid\vx)+o(1),
  \end{split}
\end{equation*}
and thus $\Cov(I_v,I_w\mid \vx)=o(1)$.
Since $\nl=\sum_v I_v$, 
it follows that $\Var(\nl\mid \vx)=o(n^2)$ and thus 
\begin{equation}\label{w1}
\nl=\E(\nl\mid\vx)+\opn.
\end{equation}
Further, if we write $h(x)=\cpolx\set l$ and sum \eqref{v4-} (where
$D=D_v$) over $v$, we obtain
\begin{equation*}
  \Bigabs{\E(\nl\mid\vx)-\sum_{v=1}^nh(x_v)}
=
 \Bigabs{\sum_{v=1}^n\bigpar{\Pr(D_v=\ell\mid\vx)-h(x_v)}}
\le O(1)+ \sum_{v=1}^n \E\|\hLa(\vx)-\ll_{x_v}\|.
\end{equation*}
By \eqref{esmall}, the right-hand side has expectation $o(n)$ and thus
\begin{equation}\label{w2}
  \E(\nl\mid\vx) = \sum_{v=1}^nh(x_v)+\opn.
\end{equation}
Now $h(x_1),\dots,h(x_n)$ are \iid\ random variables with mean
\begin{equation*}
  \E h(x_v)=\ints h(x)\dd\mu(x)
=\ints \cpolx\set l \dd\mu(x)
= \mcpol\set l.
\end{equation*}
Hence, by the law of large numbers, 
$\frac1n\sum_{v=1}^nh(x_v)\pto\mcpol\set l$, which is the same as
\begin{equation}
  \label{w3}
\sum_{v=1}^nh(x_v)=\mcpol\set l n+\opn.
\end{equation}
The result \eqref{nl} follows from \eqref{w1}, \eqref{w2}, \eqref{w3}.

Furthermore,
\eqref{nl} says that $(\nl/n)_\ell\pto\mcpol$ in the space
$\bbR^\infty$ of sequences, equipped with the product topology, which
is the same as separate convergence of the components. However, it is
well-known, and easy to see (e.g.\ by compactness) that restricted to
the set of probability distributions, this equals the standard
topology there.

We have proved \ref{tda} and \ref{tdb} for bounded $\kf$. For general $\kf$ we use
truncations: define $\kaM_F$ in analogy with \eqref{trunc},
setting $\kaM_F=\ka_F\wedge M$ for $|F|\le M$ and $\kaM_F=0$ for $|F|>M$. 
We use $\ll\M$, $\nl\M$ and so on to denote quantities defined for
\gnkm. 
For fixed $M$, applying \eqref{nl} for the bounded
kernel family $\kfM$, we have
$\nl\M/n\pto\mcpolm\set\ell$ as $n\to\infty$, and thus
by dominated convergence
\begin{equation}\label{v6}
\E\bigabs{\nl\M/n-\mcpolm\set\ell}\to0.
\end{equation}
Furthermore, for every $x\in\sss$ and $d\ge 1$,
the intensities $\lxdm$ converge to $\lxd$ 
as \Mtoo{}, by \eqref{lxd}, \eqref{v2} and monotone convergence.
Thus a simple coupling shows that 
$\mcpolm\dto\mcpol$ as \Mtoo.
We may couple $\gnk$ and $\gnkm$ in the obvious way so that $\gnk$ is
obtained from $\gnkm$ by adding further atoms, say $\nfm$
copies of each $F\in\F$. Then $\E\nfm\le n\intsf(\ka_F-\kaM_F)$, and since at
most $\sum_F\kF\nfm$ vertices are affected by the extra additions,
\begin{equation}\label{52}
  \E\Bigabs{\frac{\nl}n-\frac{\nl\M}n}
\le\frac1n\E\sum_F\kF\nfm
\le\sum_F\kF\intsf(\ka_F-\kaM_F).
\end{equation}
The right hand side is independent of $n$, and tends to 0 as
$M\to\infty$ by dominated convergence and our assumption that $\kf$ is
integrable. For any $\eps>0$, we may thus choose $M$ so large that
the right hand side of \eqref{52} is less than $\eps$, and also so that
$|\mcpolm\set\ell-\mcpol\set\ell|<\eps$; then by \eqref{v6}, for large
enough $n$,
\begin{equation*}
\E\bigabs{\nl/n-\mcpol\set\ell}<3\eps,
\end{equation*}
which proves \eqref{nl} and thus \ref{tdb}.
Further, \eqref{nl} and dominated convergence yields 
$\Pr(D_n=\ell)=\E(\nl/n)\to\mcpol\set\ell$, which proves \ref{tda}.

Finally we prove \ref{tdc}. (This could also easily be done directly in a
fairly straightforward way.)
First, \eqref{rcpo} and \eqref{lxd} yield 
\begin{equation*}
  \E\mcpol=\ints\sumfj\lxfj\dfj\dd\mu(x)
=\sumF\sumjf\dfj\intsf\ka_F,
\end{equation*}
which yields the formula for $\E\mcpol$ 
claimed in the theorem, since $\sum_j\dfj=2e(F)$.

Next, the convergence in distribution \ref{tda} yields (by a version of
Fatou's Lemma) the inequality
$\liminf_\ntoo\E D_n\ge\E\mcpol$.
Finally, recalling the definition \eqref{dq} of $\dq_n$ (denoted $\dq$ in \eqref{dq}),
we have $D_n\le\dq_n$ and thus
\begin{equation*}
  \begin{split}
\E D_n
&\le \E\dq_n
=\sumfj\dfj\E\nfj
=\sumF\sumjf\dfj \frac{(n-1)\fall{\kF-1}}{n^{\kF-1}}\intsf\ka_F
\\&
\le
\sumF 2e(F)\intsf\ka_F
=\E\mcpol,
  \end{split}
\end{equation*}
yielding the opposite inequality
$\limsup_\ntoo\E D_n\le\E\mcpol$.
\end{proof}

Part \ref{tdc} of \refT{Tdegree} is not surprising.
Also, since by symmetry $\E D_n=\frac2n\E e(\gnk)$,
it follows from \refT{Tedges} (which we shall not prove until the next section).
For bounded kernel families, it is easy to see that also higher
moments of $D_n$ converge to the corresponding moments of $\mcpol$,
for example by first showing that $\E D_n^m=O(1)$ for every fixed $m$
and combining this with \ref{tda}. This extends to certain unbounded
kernel families, but somewhat surprisingly not to all integrable
kernel families, as the following example shows.

\begin{example}\label{badP2}
Let $\sss=[0,1)$ with Lebesgue measure, and regard $\sss$ as a
circle with the usual metric $d(x,y)=\min(|x-y|,\,1-|x-y|)$. We construct our
random graph by adding triangles only; thus $\ka_F=0$ for $F\neq
K_3$, and we take
\begin{equation}\label{badP2k}
 \ka_3(x,y,z)=d(x,y)^{\eps-1}+d(x,z)^{\eps-1}+d(y,z)^{\eps-1}
\end{equation}
for some small $\eps>0$, for example $\eps=1/10$.
Clearly, $\kf$ is an integrable kernel family (and a hyperkernel).

Let $\gD=\min_{1\le i<j\le n} d(x_i,x_j)$ be the minimal spacing
between the $n$ independent uniformly distributed random points $x_i$,
$1\le i\le n$. It is well-known that this minimal spacing is of 
order $n^{-2}$; in fact, it is easy to see that for $0\le s\le 1/n$
we have
$\Pr(\gD>s)=(1-sn)^{n-1}\le e^{-sn(n-1)}$, and in
particular $\gD\le n^{\eps-2}$ whp.
Hence there exist whp two distinct indices $i$ and $j$ with
$d(x_i,x_j)\le n^{\eps-2}$, and thus, for large $n$ and every $x_k$,
$\ka_3(x_i,x_j,x_k)\ge n^{(\eps-2)(\eps-1)} \ge 2n^{2-3\eps}$. 
If $i$ and $j$ are chosen such that this holds, then from \eqref{pdef}
we have $p(i,j,k;K_3)\ge 2n^{-3\eps}$ for all $k\neq i,j$, and thus
the number of $k$ such that the triangle $ijk$ is an atom
stochastically dominates the binomial distribution
$\Bi(n-2,2n^{-3\eps})$; hence this number is whp at least $n^{1-3\eps}$. 

We have shown that whp there are at least two vertices $i$ and $j$
with degrees $\ge n^{1-3\eps}$, and thus, for large $n$, $\Pr(D_n\ge
n^{1-3\eps})\ge(1-o(1))\frac2n \ge \frac1n$. Consequently, for large $n$,
\begin{equation*}
 \E D_n^2\ge \frac1n n^{2(1-3\eps)} = n^{1-6\eps} 
\to\infty.
\end{equation*}

On the other hand, for some finite $c=\int_{\sss^3}\ka_3$, by
symmetry, $\la_{K_3,j}(x)=c$ and $\la_x=3c\gd_2$. Hence
$\mcpol=\CPo(3c\gd_2)$, which is the distribution of $2X$ with
$X\sim\Po(3c)$, which has all moments finite.
\end{example}

As we shall see in Theorem~\ref{th_reg}, this situation cannot arise in the
edge-only version of the model, i.e., the model in~\cite{BJR}; in
the terminology of the next section, all copies of $P_2$ are then `regular'.

In Section~\ref{sec_pl} we shall illustrate \refT{Tdegree} by
giving a natural family of examples with degree distributions with
power-law tails.

\section{Small subgraphs}\label{sec_subgraphs}

In this section we turn to the final general property
of $G(n,\kf)$ we shall study, the asymptotic number of copies
of a fixed graph $F$ in $G(n,\kf)$; throughout this section, $\kf$ denotes
a kernel family $(\ka_F)_{F\in\F}$, rather than a hyperkernel. We work
with the multi-graph version of the model.

Although mathematically not as interesting as the phase transition,
the number of small graphs in $G(n,\kf)$ is important as it is directly related
to the original motivation for the model. Indeed, recall that perhaps the main
defect of the model of~\cite{BJR}, i.e., the edge-only case of the present
model, is that it produces graphs with very few (usually $\Op(1)$) triangles,
i.e., graphs with clustering coefficients that are essentially zero. This contrasts
strongly with many of the real-world networks we wish to model.

The simplest way that a copy of some graph $F$ may arise in $G(n,\kf)$ is 
as an atom. The expected number of such copies is simply
\begin{equation}\label{nf}
 \frac{n\fall{|F|}}{n^{|F|-1}} \int_{\sss^{|F|}} \ka_F \le n \int_{\sss^{|F|}} \ka_F.
\end{equation}

The next simplest way that a copy of $F$ may arise is as a subgraph of some atom $F'$
of $G(n,\kf)$.
Let us call such copies of $F$ {\em direct};
we include the case $F'=F$.
Let $n(F,F')$ denote the number of subgraphs of $F'$ isomorphic to $F$,
so $n(K_3,K_4)=4$, for example. Set
\[
 \tone(F,\kf) = \sum_{F'\in\F} n(F,F')\int \ka_{F'} \le \infty,
\]
and let $\nd(F,G(n,\kf))$ denote the number of direct copies of $F$ in $G(n,\kf)$.
Then from \eqref{nf} we see that
\[
 \E \nd(F,G(n,\kf)) \le \tone(F,\kf)n,
\]
and that if $\kf$ is bounded, then
\[
 \E \nd(F,G(n,\kf)) = \tone(F,\kf)n+O(1).
\]
The reason for the somewhat peculiar notation $\tone$ is as follows: the subscript
$1$ indicates direct copies (arising from only one atom). The tilde will be
useful later to differentiate from standard notation $t(F,\ka)$ in other contexts.

It will turn out that in well behaved cases (for example for all bounded kernel families),
essentially all copies of any $2$-connected graph $F$ in $G(n,\kf)$ arise directly.
Unfortunately, this is not the case for general $F$. Perhaps the main special
cases we are interested in are stars; the number of copies of the star $K_{1,2}$
(i.e., the path $P_2$)
is needed to calculate the clustering coefficient, for example.
Note that the number of copies of the star $K_{1,k}$ ($k\ge2$) in any
graph $G$ is simply $|G|/k!$ times the $k$th factorial
moment of the degree of a random vertex; hence counting stars 
is closely related to studying the degree distribution, which we did
for \gnk{} in \refS{sec_degrees}. 

Let us say that a copy of $F$ in $G(n,\kf)$ arises {\em indirectly} if it contains
edges of at least two of the atoms making up $G(n,\kf)$.
To understand the expected number of such copies we first need to understand
the probability that a certain set of vertices form a copy of $F$
{\em given the types of the vertices}. More precisely, we consider
the expectation of the number of copies of $F$ with a given vertex set, even though this number
is highly unlikely to exceed 1.

Let $F$ be a connected graph with $r$ vertices.
Let $\emb(F,F')$ denote the number of {\em embeddings} of $F$ into $F'$,
i.e., the number of injective homomorphisms from $F$ to $F'$,
so $\emb(F,F')=n(F,F')\aut(F)$. 
Fixing a labelling of $F$, let $X_F^0(G)$
denote the number of copies
of $F$ in a multigraph $G$ with vertex $i$ of $F$ corresponding
to vertex $i$ of $G$. (Thus $X_F^0(G)$ is $0$ or $1$ if $G$ is simple.)
The contribution to $\E X_F^0(G(n,\kf)\mid x_1,\ldots,x_r)$
from copies of $F$ arising as subgraphs of
atoms isomorphic to a given $F'$ with $r'$ vertices is exactly
\[
 \sum_{\varphi:F\to F'} \frac{(n-r)\fall{r'-r}}{n^{r'-1}} \int_{\sss^{r'-r}} \ka_{F'}(y_1,\ldots,y_{r'}),
\]
where $\varphi$ runs over all $\emb(F,F')$ embeddings of $F$ into $F'$,
we take $y_j=x_i$ if $\varphi(i)=j$, and we integrate over the remaining $r'-r$ variables
$y_j$.

Set
\begin{equation}\label{sigdef}
 \sigma_F(x_1,\ldots,x_r)=\sigma_F(x_1,\ldots,x_r;\kf) =
 \sum_{F'} \sum_{\varphi:F\to F'} \int_{\sss^{r'-r}} \ka_{F'}(y_1,\ldots,y_{r'}).
\end{equation}
Then we have
\begin{equation}\label{sup}
 \E\bb{ X_F^0(G(n,\kf)) \bigm| x_1,\ldots,x_r } \le n^{-(r-1)} \sigma_F(x_1,\ldots,x_r;\kf),
\end{equation}
and if $\kf$ is bounded then the relative error is $O(n^{-1})$.

Comparing \eqref{sigdef} and \eqref{kaedef}, note that if $F=K_2$, then $\sigma_F=\kae$.
Before continuing, let us comment on the normalization. Recall that in defining $G(n,\kf)$,
we consider all $r!$ possible ways of adding a (labelled)
copy of $F$ on vertex set $\{1,2,\ldots,r\}$, say,
adding each copy with probability $\ka_F(x_1,\ldots,x_r)/n^{r-1}$.
This means that the contribution from $\ka_F$ to $X_F^0(G(n,\kf))$ is
$\aut(F)\ka_F(x_1,\ldots,x_r)/n^{r-1}$, and, correspondingly,
the contribution from $\ka_F$ to $\sigma_F$ is $\aut(F)\ka_F(x_1,\ldots,x_r)$.
In other words, while having the same form as a kernel, $\sigma_F$ is normalized differently.
This situation arises already in the edge-only case, where $\kae(x,y)=2\ka_2(x,y)$.
It turns out that here the normalization of $\sigma_F$, giving directly the probability
that a certain set of edges forming a copy of $F$ is present, is the natural
one. Note that if we had used this normalization from the beginning, then formulae such
as \eqref{nf} would have extra factors.

Let $F$ be a connected graph with vertex set $[r]$. We say that a set $F_1,\ldots,F_a$
of connected graphs forms a {\em tree decomposition} of $F$ if each $F_i$ is connected,
the union of the $F_i$ is exactly $F$, any two of the $F_i$ share at most one vertex,
and the $F_i$ intersect in a tree-like structure. The last condition may be
expressed by saying that the $F_i$ may be ordered so that each $F_j$ 
other than the first meets the union of the previous ones in exactly one vertex.
Equivalently, the intersection is tree-like if $|F|=1+\sum_i (|F_i|-1)$.
Equivalently, defining (as usual) a {\em block} of a graph $G$ to be
either a maximal 2-connected subgraph of $G$ or a bridge in $G$,
$F_1,\ldots,F_a$ forms a tree composition of $F$ if each $F_i$
is a connected union of one or more blocks of $F$, with each block contained in exactly one $F_i$.
(Cf.\ \cite[p.~74]{Bollobas:MGT}.)

Note that we allow $a=1$, in which case $F_1=F$. For $a\ge 2$, the
order of the factors is irrelevant, so, for example, $K_{1,2}$ has a
unique non-trivial tree decomposition, into two edges. Note also that if
$F$ is 2-connected, then it has only the trivial tree decomposition.

Let us say that a copy of $F$ in $G(n,\ka)$ if {\em regular} if it is the union
of graphs $F_1,\ldots,F_a$ forming a tree decomposition of $F$, where each $F_i$
arises directly as a subgraph of some atom $F_i'$, and $V(F_i')\cap V(F_j')=V(F_i)\cap V(F_j)$
for all $i\ne j$
(with this intersection containing at most one vertex).
We can write down exactly the probability that $G(n,\kf)$ contains a regular copy of $F$
with vertex set $1,\ldots,r$ in terms of certain integrals of products of conditional expectations
$\E(X_{F_i}^0(G)\mid x_1,\ldots,x_{s})$. We shall not do so.
Instead, let
\begin{equation}\label{t1fdef}
 \traw(F,\kf) = \sum_{\{F_1,\ldots,F_a\}} \int_{\sss^r} \sigma_{F_1}\sigma_{F_2}\cdots\sigma_{F_a} \dd\mu(x_1)\cdots\dd\mu(x_r),
\end{equation}
where the sum runs over all tree decompositions of $F$ and each term $\sigma_{F_i}$ is evaluated
at the subset of $x_1,\ldots,x_r$ corresponding to the vertices of $F_i\subset F$, and set
\begin{equation}\label{t2fdef}
 \ts(F,\kf) = \aut(F)^{-1} \traw(F,\kf).
\end{equation}
Note that these definitions
extend to disconnected graphs $F$, taking
the sum over all combinations of one tree decomposition for each component of $F$.

The upper bound \eqref{sup} easily implies that the expected number of regular copies of $F$
in $G(n,\kf)$ is at most $\ts(F,\kf)n$, and furthermore this bound is correct
within a factor $1+O(n^{-1})$ if $\kf$ is bounded; the factor $\aut(F)^{-1}$ appears
because there are $n\fall{r}/\aut(F)$ potential copies of $F$.
Note that the number $\emb(F,G)$ of embeddings of a graph $F$ into a graph $G$,
i.e., the number of injective homomorphisms from $F$ to $G$, is simply $\aut(F)n(F,G)$.
Hence $\traw(F,\kf)$ is the appropriate normalization for counting embeddings of $F$ 
into $G(n,\kf)$ rather than copies of $F$. In other contexts, when dealing with dense
graphs, it turns out to be most natural to consider homomorphisms from $F$ to $G$,
the number of which will be very close to $\emb(F,G)$. Thus the normalization
in \eqref{t1fdef} is standard in related contexts. (See, for example,
Lov\'asz and Szegedy~\cite{LSz1}.)

Let us illustrate the definitions above with two simple examples.

\begin{example}\label{EtK2}
  The simplest case is $F=K_2$.
In this case, there is only the trivial tree decomposition, and
\eqref{t1fdef} and \eqref{t2fdef} yield
\begin{equation}\label{etK2}
  \ts(K_2,\kf)
=\frac12\int_{\sss^2}\sigma_{K_2}(x,y)
=\frac12\int_{\sss^2}\kae(x,y)
=\xiek.
\end{equation}
\end{example}

\begin{example}\label{EtP2}
Suppose that $\kf$
contains only two non-zero kernels, $\ka_2$, corresponding to an edge, and $\ka_3$,
corresponding to a triangle; our aim is to calculate $\ts(P_2,\kf)$ in this case,
where $P_2$ is the path with 2 edges.
Using symmetry of $\ka_2$ and $\ka_3$,
\begin{equation}\label{sK2}
 \sigma_{K_2}(x,y) = 2\ka_2(x,y) + 6\int_\sss \ka_3(x,y,z)\dd\mu(z),
\end{equation}
while
\begin{equation}\label{sK3}
 \sigma_{P_2}(x,y,z) = 6\ka_3(x,y,z),
\end{equation}
reflecting the fact the $P_2$ $ijk$ appears directly in $G(n,\kf)$ if and only
if we added a triangle with vertex set $\{i,j,k\}$, and this vertex set
corresponds to 6 $3$-tuples.

Since $\aut(P_2)=2$, it follows that
\[
 \ts(P_2,\kf) = \frac{1}{2}\int \bigpar{\sigma_{P_2}(x,y,z)
 + \sigma_{K_2}(x,y)\sigma_{K_2}(y,z)} \dd\mu(x)\dd\mu(y)\dd\mu(z).
\]
\end{example}

More generally, let $F$ be any (simple)
subgraph of $G_n=G(n,\kf)$ with $k$ components.
(We abuse notation by now writing $F$ for a specific subgraph of $G_n$, rather than
an isomorphism class of graphs.)
Let $F_1',\ldots,F_a'$ list all atoms contributing edges of $F$,
and let $F_i=F_i'\cap F$, where we take the intersection in the multigraph
sense, i.e., intersect the edge sets.
For example, if $e_1$ and $e_2$
are parallel edges in $G_n$ forming a double edge from $i$
to $j$, and $e_1\in E(F)$, $e_2\in E(F_1')$,
then $F_1=F_1'\cap F$ contains no $ij$ edge, even though $F_1'$ and $F$ each do so.
By definition each $F_i$ contains at least one edge, and
$F$ is the edge-disjoint union of the $F_i$.
Since $F$ has $k$ components, 
when adding the $F_i$ one by one, at least $a-k$ times a new component is {\em not} created,
so at least $a-k$ times at least one vertex of $F_i$, and hence of $F_i'$,
is repeated. It follows that
\begin{equation}\label{reg?}
 \sum_i (|F_i'|-1)  \ge \Bigabs{\bigcup_i F_i'}-k. 
\end{equation}
Extending our earlier definition, 
we call $F$ {\em regular} if equality holds in \eqref{reg?}, 
and {\em exceptional} otherwise.
Note that if any $F_i$ is disconnected, then $F$ is exceptional.

Let $\nr(F,G_n)$ denote the number of regular copies of $F$ in $G_n=G(n,\kf)$,
and $\nx(F,G_n)$ the number of exceptional copies.

\begin{theorem}\label{th_ssbd}
Let $G_n=G(n,\kf)$, where $\kf$ is a kernel family, and let $F$ be a graph with $k$ components.
Then
\[
 \E \nr(F,G_n) \le n^k \ts(F,\kf).
\]
If $\kf$ is bounded, then
\[
 \E \nx(F,G_n) =O(n^{k-1}),
\]
\[
 \Var(n(F,G_n)) =O(n^{2k-1}),
\]
and 
\[
 n(F,G_n)=\nr(F,G_n)+\nx(F,G_n) = n^k \ts(F,\kf) (1+\Op(n^{-1/2})).
\]
\end{theorem}
\begin{proof}
We have essentially given the proof of the first statement, so let us just outline it.
To construct a regular copy of $F$ in $G_n$ we must first choose graphs $F_1,\ldots,F_a$
on $V(F)$ forming a tree decomposition of each component of $F$. Then we must choose
a graph $F_i'$ containing each $F_i$ to be the atom that will contain $F_i$.
Then we must choose $s=|\bigcup_i F_i'|$ distinct vertices
$v_1,\ldots,v_s$
from $1,\ldots,n$ to be the vertices
of the $F_i'$, where (since $F$ is regular), we have
$s=k+\sum_i (|F_i'|-1)$.

Note that there are $n\fall{s}\le n^s$ choices for the vertices $v_i$. (We are glossing over
the details of the counting, and in particular various factors $\aut(H)$ for various
graphs $H$. It should be clear comparing the definition of $\ts(F,\kf)$ with what follows
that these are in the end accounted for correctly.)

Given the vertex types,
the probability that these particular graphs $F_i'$ arise is then (up to certain factors
$\aut(F_i')$) a product of factors of the form $\ka_{F_i'}/n^{|F_i'|-1}$,
where the kernel is evaluated at an appropriate subset of $x_{v_1},\ldots,x_{v_s}$.
Note that the overall power of $n$ in the denominator is $\sum_i (|F_i'|-1)=s-k$.
Integrating out over the variables $x_j$ corresponding to $V(F_i')\setminus V(F_i)$,
and summing over all $F_i'\supset F_i$, the factor $\ka_{F_i'}$ becomes a factor
$\sigma_F$. Finally, integrating out over the remaining variables,
corresponding to vertices of $F$, and summing over decompositions,
we obtain $n^k \ts(F,\kf)$ as an upper bound.

If $\kf$ is bounded then the number $s$ of vertices appearing above
is bounded, so $n\fall{s}/n^s = 1-O(n^{-1})$, where the error term is uniform over
all choices for $F_1',\ldots,F_a'$.
It follows that in this case,
\[
 \E \nr(F,G_n) = \ts(F,\kf)n^k(1-O(n^{-1})).
\]

Arguing similarly for exceptional copies, the power $\sum_i (|F_i'|-1)$ of $n$ in the denominator
is now at least $s-k+1$, and it follows that if $\kf$ is bounded,
then $\E \nx(F,G_n) = O(n^{k-1})$ as claimed. It follows that
\begin{equation}\label{tot}
 \E n(F,G_n) = \ts(F,\kf)n^k +O(n^{k-1}).
\end{equation}

Finally, for the variance we simply note that
$\E n(F,G_n)^2$ is the expected number of ordered pairs $(F_1,F_2)$ of not necessarily disjoint
copies of $F$ in $G_n$.
If $F_1$ and $F_2$ share one or more vertices, then $F_1\cup F_2$ has at most $2k-1$
components. From \eqref{tot}, the expected number of such pairs is $O(n^{2k-1})$.
The expected number of pairs with $F_1$ and $F_2$ disjoint is simply 
$N\E n(2F,G_n)$, where $2F$ is the disjoint union of two copies of $F$
and $N=\aut(2F)/\aut(F)^2$
is a symmetry factor, the number of ways $2F$ can be divided
into 2 copies of $F$. (If $F$ is connected then simply $N=2$ and in
general, if $F$ has distinct components $F_j$ with multiplicities
$m_j$, then $N=\prod_j\binom {2m_j}{m_j}$.)
Since $\traw(2F,\kf)=\traw(F,\kf)^2$, we have
$\ts(2F,\kf)=\ts(F,\kf)^2/N$, so \eqref{tot} gives
\[
 \E n(F,G_n)^2  = \ts(F,\kf)^2n^{2k} + O(n^{2k-1}),
\]
from which the variance bound follows. The final bound follows by Chebyshev's inequality.
\end{proof}

For bounded kernel families, \refT{th_ssbd} is more or less the end of the story, although
one can of course prove more precise results. For unbounded kernel families the situation
is much more complicated. Let us first note that regular copies of $F$ do not give rise
to any problems.

\begin{theorem}\label{th_reg}
Let $\kf$ be a kernel family and $F$ a connected graph, and
let $G_n=G(n,\kf)$. Then $\nr(F,G_n)/n\pto \ts(F,\kf)\le\infty$.
In other words,
if\/ $\ts(F,\kf)=\infty$, then for any constant $C$, whp $\nr(F,G_n)\ge Cn$,
while if \/ $\ts(F,\kf)<\infty$, then $\nr(F,G_n) =\ts(F,\kf)n+\op(n)$.
\end{theorem}
\begin{proof}
We consider the truncated kernel families $\kfM$. Since $\traw(F,\kf)$
is a sum of integrals of products of sums of integrals
of the kernels $\ka_{F'}$, by monotone convergence we have
$
 \traw(F,\kfM) \to \traw(F,\kf) \le\infty
$
as $M\to\infty$, and hence $\ts(F,\kfM)\to \ts(F,\kf)$.

If $\ts(F,\kf)=\infty$, choose $M$ so that $\ts(F,\kfM)>C$, and couple
$G_n'=G(n,\kfM)$ and $G_n$ in the natural way so that $G_n'\subseteq G_n$.
Since $\kfM$ is bounded, Theorem~\ref{th_ssbd} implies that $\nr(F,G_n')\ge Cn$ whp.
Since $\nr(F,G_n)\ge \nr(F,G_n')$, the result follows.

If $\ts(F,\kf)<\infty$, then given $\eps>0$, the truncation argument above shows
that $\nr(F,G_n)\ge (\ts(F,\kf)-\eps)n$ holds whp.
By the first statement of Theorem~\ref{th_ssbd}, $\E \nr(F,G_n)\le \ts(F,\kf)n$.
Combining these two bounds gives the result.
\end{proof}

Note that we do not directly control the variance of $\nr(F,G_n)$; as we shall
see in Section~\ref{sec_pl}, there are natural examples where $\nr(F,G_n)/n$
is concentrated about its finite mean even though its variance 
tends to infinity.

The very simplest case of \refT{th_reg} concerns edges; we stated this
as a separate result in the introduction.

\begin{proof}[Proof of \refT{Tedges}]
Since all copies of $K_2$ in $G_n=\gnk$ are regular (and direct),
$e(G_n)=n(K_2,G_n)=\nr(K_2,G_n)$, and taking $F=K_2$ in \refT{th_reg}
and using \eqref{etK2}
yields $e(G_n)/n\pto \xiek$, which is the first claim of
\refT{Tedges}.
It remains to show that $\E e(G_n)=\E\nr(K_2,G_n)\to\xie(\kf)$.
The lower bound follows from the first part,
since convergence in probability implies $\liminf_\ntoo \E e(G_n)/n\ge\xiek$, 
while \refT{th_ssbd} gives $\E e(G_n)/n\le \ts(K_2,\kf)=\xiek$, completing the proof.
\end{proof}

It is also easy to prove \refT{Tedges} directly, using 
truncations as in this section but avoiding many complications
present in
the general case.

By a {\em moment} of a kernel family $\kf$ we shall mean any integral of the form
\[
 \int_{\sss^d} \ka_{F_1}\ka_{F_2}\cdots\ka_{F_r} \dd\mu(x_1)\cdots\dd\mu(x_d),
\]
where $F_1,\ldots,F_r$ are not necessarily distinct, and each term $\ka_{F_i}$
is evaluated at some $|F_i|$-tuple of distinct $x_j$.
The proof of Theorem~\ref{th_ssbd} shows that for any connected $F$, $\E\nx(F,G(n,\kf))$
is bounded by a sum of moments of $\kf$.
%(In fact, in the multigraph version, we almost have equality. 
%(Not quite. We almost have equality to one such sum, but we
%  bound by a larger sum that also includes terms with redundancies 2
%  or more, which disappear in the limit.)
%The only difference is certain factors
%$n\fall{k}/n^k$ corresponding to choices for a set of $k$ distinct vertices.)
This gives a very strong condition under which we can control $\nx(F,G(n,\kf))$.

\begin{theorem}\label{th_mf}
Let $\kf$ be a kernel family in which only finitely many kernels $\ka_F$ are
non-zero. Suppose also that all moments of $\kf$ are finite.
Then for any connected $F$, $\E\nx(F,G(n,\kf))=O(1)$,
and the conclusions of Theorem~\ref{th_reg} apply with $\nr(F,G_n)$ replaced by $n(F,G_n)$.
\end{theorem}
\begin{proof}
This is essentially trivial from the comments above and Theorem~\ref{th_reg}. We omit the details.
\end{proof}

Example~\ref{badP2} shows that some conditions are necessary to control
$\nx(F,G(n,\kf))$; we refer the reader to \eqref{badP2k} for the 
description of the kernel family in this case.
Plugging \eqref{badP2k} into \eqref{sK2}, in this case
we have $\sigma_{K_2}(x,y)=6d(x,y)^{\eps-1}+a$ for some constant $a$ (in fact,
$a=24\eps^{-1}2^{-\eps}$), and it easily follows that $\ts(P_2,\kf)<\infty$.
However, as shown in the discussion of that example, whp there is a vertex
with degree at least $n^{1-3\eps}$, and hence at least $n^{2-6\eps}$ copies of
$P_2$, which is much larger than $n$ if $\eps<1/6$. In this case the problem
is exceptional $P_2$s $ijk$ arising from atoms $ij\ell$ and $jk\ell$:
the corresponding moment
\[
 \int_{\sss^4} \ka_3(x_1,x_2,x_4)\ka_3(x_2,x_3,x_4)
\]
is infinite, due to the contribution from $d(x_2,x_4)^{2\eps-2}$.

Of course, not all moments contribute to $\E\nx(F,G_n)$; as we shall see in the next section,
it is easy to obtain results similar to Theorem~\ref{th_mf} under weaker assumptions in special cases.
Also, in general it may happen that $\nx(F,G_n)$ has infinite expectation (in the multigraph form),
but is nonetheless often small, i.e., that the large expectation comes from the small
probability of having a vertex in very many copies of $F$. Much more generally,
it turns out that when $\kf$ is integrable, whp all exceptional copies of $F$
sit on a rather small set of vertices.

\begin{theorem}\label{th_ess}
Let $\kf$ be an integrable kernel family and $F$ a connected graph, with $\ts(F,\kf)$ finite.
Let $G_n=G(n,\kf)$.

For any $\eps>0$, there is a $\delta>0$ such that whp {\em every} graph $G_n'$ formed
from $G_n$ by deleting at most $\delta n$ vertices has $n(F,G_n')\ge (\ts(F,\kf)-\eps) n$.

For any $\eps>0$ and any $\delta>0$, whp there is {\em some} graph $G_n'$
formed from $G_n$ by deleting at most $\delta n$ vertices such that
$n(F,G_n')\le (\ts(F,\kf)+\eps) n$.
\end{theorem}
Together the statements above may be taken as saying that $G_n$ contains
essentially $(\ts(F,\kf)+\op(1))n$ copies of $F$, where `essentially' means that
we may ignore $o(n)$ vertices. In other words, the `bulk' of $G_n$
contains this many copies of $F$, though a few exceptional vertices may
meet many more copies.
\begin{proof}
We start with the second statement, since it is more or less immediate.
Indeed, writing $\ii{\kf}$
for $\sum_{F\in \F} |F|\int_{\sss^{|F|}} \ka_F$,
and considering truncations $\kfM$ as usual, from monotone convergence
we have $\ii{\kfM}\upto \ii{\kf}$ as $M\to\infty$.
Let $\eps>0$, $\delta>0$ and $\eta>0$ be given.
Since $\kf$ is integrable, i.e., $\ii{\kf}<\infty$, there is some $M$
such that $\ii{\kfM}\ge \ii{\kf}-\delta\eta/2$.
Coupling $G_n^M=G(n,\kfM)$ and $G_n=G(n,\kf)$ in the usual way, let us call
a vertex {\em bad} if it meets an atom present in $G_n$ but not $G_n^M$.
The expected number of bad vertices is at most the expected sum of the sizes
of the extra atoms, which is at most
$n(\ii{\kf}-\ii{\kfM})\le \delta\eta n/2$.
Hence the probability that there are more than $\delta n$ bad vertices is at most $\eta/2$.

Deleting all bad vertices from $G_n$ leaves a graph $G_n'$ with at most $n(F,G_n^M)$ copies
of $F$. Applying Theorem~\ref{th_ssbd}, this number is at most
$(\ts(F,\kfM)+\eps)n \le (\ts(F,\kf)+\eps)n$
whp, so we see that if $n$ is large enough, then with probability at least $1-\eta$
we may delete at most $\delta n$ vertices to leave $G_n'$ with at most $(\ts(F,\kf)+\eps)n$
copies of $F$, as required.

Turning to the first statement, we may assume without loss of generality that $\kf$
is bounded. Indeed, there is some truncation $\kfM$ with $\ts(F,\kfM)\ge \ts(F,\kf)-\eps/2$,
and taking $G(n,\kfM)\subset G(n,\kf)$ as usual, it suffices to prove the same statement
for $G(n,\kfM)$ with $\eps$ replaced by $\eps/2$.
Assuming $\kf$ is bounded,
then by Theorem~\ref{th_ssbd} we have $n(F,\kf)\ge (\ts(F,\kf)-\eps/2)n$ whp, 
so it suffices to prove that if $\kf$ is bounded and $\eps>0$, then there is some
$\delta>0$ such that whp any $\delta n$ vertices of $G_n=G(n,\kf)$ meet
at most $\eps n$ copies of $F$. 

Let $v$ be a fixed vertex of $F$, and for $1\le i\le n$ let $a_i$ denote the number of
homomorphisms from $F$ to $G_n$ mapping $v$ to vertex $i$.
Let $F'$ be the graph formed from two copies of $F$ meeting only at $v$.
Then there are exactly $a_i^2$ homomorphisms from $F'$ to $G_n$ mapping $v$ to $i$,
so in total there are $\sum a_i^2$ homomorphisms from $F'$ to $G_n$.
Now the image of any homomorphism from $F'$ to $G_n$ is a connected subgraph $F''$
of $G_n$, and each such subgraph is the image of $O(1)$ homomorphisms.
Applying Theorem~\ref{th_ssbd} to each of the $O(1)$ possible isomorphism types of $F''$,
it follows that there is some constant $C$ such that, whp,
\[
 \sum_i a_i^2 = \hom(F',G_n) \le Cn.
\]
When the upper bound holds,
given any set $S\subset [n]$ with $|S|\le \delta n$, by the Cauchy--Schwartz inequality
we have
\[
 \sum_{i\in S} a_i \le \sqrt{|S|}\sqrt{\sum_i a_i^2} \le \sqrt{\delta n}\sqrt{Cn}=\sqrt{C\delta}n.
\]

Repeating the argument above for each vertex $v$ of $F$
and summing,
we see that there is some $C'<\infty$
(given by the sum of at most $|F|$ constants corresponding to $\sqrt{C}$ above)
such that whp
for any $\delta>0$, and any set $S$ of at most $\delta n$ vertices of $G_n$,
there are at most $C'\sqrt{\delta}n$ homomorphisms from $F$ to $G_n$ mapping any vertex
of $F$ into $S$. This condition implies that $S$ meets at most $C'\sqrt{\delta}n$ copies
of $F$, so choosing $\delta$ such that $C'\sqrt{\delta}<\eps$,
we see that whp any $\delta n$ vertices meet at most $\eps n$ copies of $F$.
As noted above, the first statement of the theorem follows.
\end{proof}

\section{A power-law graph with clustering}\label{sec_pl}

Our aim in this paper has been to introduce a very general family of sparse
random graph models, showing that despite the generality, the models are
still susceptible to mathematical analysis. The question of which special
cases of the model may be relevant in applications is a very broad one,
and not our focus. Nevertheless, in the light of the motivation of the model,
we shall investigate one special case. We should like to show that,
with an appropriate choice of kernel family, our model
gives rise to graphs with power-law degree distributions, with various ranges
of the degree exponent,
the clustering coefficient (see~\eqref{clusteringcoeff}),
and the mixing coefficient (see~\eqref{mixingcoeff}).
We achieve this in the simplest possible way, considering a `rank 1'
version of the model in which we add only edges and triangles. We do not claim
that this particular model is appropriate for any particular real-world example;
nevertheless, it shows the potential of our model to produce
graphs that are similar to real-world graphs, where similarity is measured
by the values of these important and much studied parameters.

Throughout this section we fix three parameters, $\alpha>1$, and $A,B\ge 0$ 
with $A+B>0$.
We consider one specific kernel family $\kf$ on $\sss=(0,1]$ with
$\mu$ Lebesgue measure. Our kernel family has only two non-zero kernels,
$\ka_2$, corresponding to edges, and $\ka_3$ to triangles, with
\[
 \ka_2(x,y) = Ax^{-1/\alpha}y^{-1/\alpha}
\]
and
\[
 \ka_3(x,y,z) = Bx^{-1/\alpha}y^{-1/\alpha}z^{-1/\alpha}.
\]
We could of course consider many other possible functions, but these seem 
the simplest and most natural for our purposes. It would be straightforward
to carry out computations such as those that follow with each
of the $\alpha$s above replaced by a different constant, for example,
although we should symmetrize the kernels in this case. However,
one of these exponents would determine the power law, and it seems most natural to
take them all equal.

For convenience, we define
\begin{equation}\label{gbk}
   \gbx_k=\int_{0}^1x^{-k/\alpha}\dd x=
   \begin{cases}
\frac{\alpha}{\alpha-k},& \alpha>k, \\
\infty, & \alpha\le k.	 
   \end{cases}
\end{equation}
In particular, $\gb={\alpha}/({\alpha-1})$.
We then have
\[
 \int_{\sss^2} \ka_2 = A\gb^2 \quad
 \hbox{ and }\quad
 \int_{\sss^3} \ka_3 = B\gb^3,
\]
so $\kf$ is integrable.
Also, for the asymptotic edge density in \refT{Tedges},
\begin{equation}\label{xiex}
  \xiek= \int_{\sss^2} \ka_2 +3 \int_{\sss^3} \ka_3 
= A\gb^2+ 3B\gb^3.
\end{equation}
In the following subsections we apply our general results to determine various
characteristics of this particular random graph $G_n=G(n,\kf)$. 

\subsection{Degree distribution}
From \eqref{v2} and symmetry of $\ka_2$ and $\ka_3$ we see that
\[
 \la_{K_2,1}(x)=\la_{K_2,2}(x) = \int_\sss \ka_2(x,y)\dd\mu(y) = A\gb x^{-1/\alpha},
\]
while for $j=1,2,3$,
\[
 \la_{K_3,j}(x)=\int_{\sss^2} \ka_3(x,y,z)\dd\mu(y)\dd\mu(z) = B\gb^2x^{-1/\alpha}.
\]
Since an edge contributes 1 to the degree of each endvertex, while a
triangle contributes 2 to the degrees of its vertices, for each $x$,
the measure $\la_x$ defined by \eqref{lax} is given by
\[
 \la_x = 2A\gb x^{-1/\alpha} \delta_1 + 3B\gb^2 x^{-1/\alpha} \delta_2.
\]
Theorem~\ref{Tdegree} then tells us that the degree distribution of $G_n=G(n,\kf)$
converges to the mixed compound Poisson distribution $\mcpol$, where
$\LL$ is the random measure corresponding to $\la_x$ with $x$ chosen uniformly from
$(0,1]$.

Note that if $B=0$, then the limiting degree distribution is mixed Poisson, while
if $A=0$, almost all degrees are even and the degrees divided by $2$
have a mixed Poisson distribution.

For the power law, note that the mean $\la(x)$ of $\la_x$ is simply
\[
 \la(x) = (2A\gb+6B\gb^2)x^{-1/\alpha} = c x^{-1/\alpha},
\]
where $0<c=2A\gb+6B\gb^2=2\xiek/\gb<\infty$
is a constant depending on $A$, $B$ and $\alpha$.
Choosing $x$ randomly from $(0,1]$, for any $k>c$ we have
\[
 \Pr(\la(x)>k) = \Pr(x<(k/c)^{-\alpha}) = (k/c)^{-\alpha},
\]
so the distribution of $\la(x)$ has a power-law tail. Using the concentration
properties of Poisson distributions with large means, arguing as in the proof of
Corollary 13.1 of~\cite{BJR}, it follows easily that
\[
 \Pr(\mcpol >k )\sim (k/c)^{-\alpha}
\]
as $k\to\infty$, so the asymptotic degree distribution does indeed have a power-law
tail with (cumulative) exponent $\alpha$.

Let $d_k=\Pr(\mcpol=k)$, so by Theorem~\ref{Tdegree}, the asymptotic
fraction of vertices with degree $k$ is simply $d_k$.
If $A>0$ then it is not hard to check that in fact
\begin{equation}\label{dk}
 d_k\sim c'k^{-\alpha-1}
\end{equation}
as $k\to\infty$, where $0<c'=\alpha c^{\alpha}<\infty$,
so the degree
distribution is power-law in this stronger sense.
If $A=0$, then $d_k=0$ if $k$ is odd, but \eqref{dk} still holds for even
$k$, for a different (doubled) constant $c'$.

\subsection{The phase transition and the giant component}\label{ss_egg}
From~\eqref{kaedef},
we have $\kae(x,y)=(2A+6B\gb)x^{-1/\alpha}y^{-1/\alpha}$, which we may rewrite
as $\kae(x,y)=\psi(x)\psi(y)$, where
\[
 \psi(x) = (2A+6B\gb)^{1/2} x^{-1/\alpha}.
\]
By Theorems~\ref{th1} and~\ref{th2}, the largest component of $G_n$ is of
size $\rho(\kf)n+\op(n)$, and there is a giant component, i.e., $\rho(\kf)>0$,
if and only if $\norm{\Tk}>1$.
In this case $\kae$ is `rank 1' in the terminology of~\cite{BJR}, and we
have
\[
 \norm{\Tk}=\tn{\psi}^2 = (2A+6B\gb)\gbb.
\]
Hence, fixing $\alpha>2$ and thus $\gb$ and $\gbb$, there is a giant
component if and only if
\begin{equation}\label{supcrit}
 2A+6B\alpha/(\alpha-1)>(\alpha-2)/\alpha.
\end{equation}

Turning to the normalized size $\rho(\kf)$ of the giant component,
\refT{unique} allows us to calculate this in terms of the solution
to a functional equation. Usually this is intractable, but for the
special $\kf$ we are considering this simplifies greatly, as in the rank 1 case
of the edge-only model; see Section 16.4 of~\cite{BJR}, or Section 6.2 of~\cite{Rsmall}.
Indeed, writing $\rho(x)$ for the survival probability of $\bpk(x)$,
from \eqref{Sk} we have
\begin{multline*}
 \Sk(\rho)(x) = \int_0^1 2Ax^{-1/\alpha}y^{-1/\alpha}\rho(y)\dd y \\
  + \int_0^1\int_0^1 3Bx^{-1/\alpha}y^{-1/\alpha}z^{-1/\alpha}(\rho(y)+\rho(z)-\rho(y)\rho(z))\dd y \dd z,
\end{multline*}
which simplifies to
\[
 \Sk(\rho)(x) =  x^{-1/\alpha}(2AC+6B\gb C-3BC^2),
\]
where
\begin{equation}\label{Ceq}
 C=\int_0^1 x^{-1/\alpha} \rho(x)\dd x.
\end{equation}
By Lemma~\ref{l_max}, we have $\rho(x)=1-\exp(-\Sk(\rho)(x))$, so
\begin{equation}\label{rsol}
 \rho(x) = 1-\exp\bb{-(2AC+6B\gb C-3BC^2) x^{-1/\alpha} }.
\end{equation}
Although we defined $C$ in terms of $\rho$, we can view $C$ as an unknown constant,
define $\rho$ by \eqref{rsol}, and substitute back into \eqref{Ceq}.
The function $\rho$ then solves \eqref{fS} if and only if $C$ solves
\begin{equation}\label{ceq}
 C = \int_0^1 x^{-1/\alpha} 
\left(1-\exp\bigpar{-\bigpar{(2A+6B\gb)C -3BC^2}x^{-1/\alpha}}\right),
\end{equation}
and every solution to \eqref{fS} arises in this way. In particular,
by Theorems~\ref{unique} and~\ref{th2}, there is a positive solution only
in the supercritical case (when \eqref{supcrit} holds), and that
solution is then unique; $C=0$ is always a solution.
Transforming the integral using the substitution $y=x^{-1/\alpha}$, 
one can rewrite the right hand
side of \eqref{ceq} in terms of an incomplete gamma function, although
it is not clear this is informative. The point is that the form of $\rho(x)$
is given by \eqref{rsol}, and the constant can in principle be found
as the solution to an equation, and can very easily be found numerically
for given values of $A$, $B$ and $\alpha$.

\subsection{Subgraph densities}\label{ss_counts}
In the following subsections we shall need expressions for $\ts(F,\kf)$
for various small graphs $F$, where $\ts(F,\kf)$, defined by \eqref{t1fdef} and \eqref{t2fdef},
may be thought of as the asymptotic density of copies of $F$ in the kernel
family $\kf$.

We start with direct copies of $F$. Since all atoms are edges or triangles,
the only graphs $F$ that can be produced directly are edges, triangles,
and $P_2$s, i.e., paths with 2 edges.

Putting the specific kernels $\ka_2$ and $\ka_3$ into the formulae \eqref{sK2}
and \eqref{sK3} from the previous section, we have
\begin{equation*}%\label{gsk2} 
 \sigma_{K_2}(x,y) = (2A+6B\gb)x^{-1/\alpha}y^{-1/\alpha},
\end{equation*}
and
\[
 \sigma_{P_2}(x,y,z) = 6Bx^{-1/\alpha}y^{-1/\alpha}z^{-1/\alpha},
\]
while
\[
  \sigma_{K_3}(x,y,z) =6\ka_{K_3}(x,y,z) = 6Bx^{-1/\alpha}y^{-1/\alpha}z^{-1/\alpha}.
\]

Edges may be formed only directly, so either from \eqref{t1fdef} and \eqref{t2fdef}
or from \eqref{etK2}, we have
\[
 \ts(K_2,\kf) = \frac{1}{2} \int \sigma_{K_2}(x,y)\dd\mu(x)\dd\mu(y) = A\gb^2+3B\gb^3,
\]
which agrees, as it should, with \eqref{xiex}.
Since a triangle is 2-connected, it has no non-trivial tree decomposition, and
\eqref{t1fdef} and \eqref{t2fdef} give
\[
 \ts(K_3,\kf) = \frac{1}{6}\int_{\sss^3} 6\ka_3(x,y,z)\dd\mu(x)\dd\mu(y)\dd\mu(z)= B\gb^3,
\]
which may also be seen by noting that the only regular copies of a triangle
are those directly corresponding to $\ka_3$.

A copy of $P_2$ may be formed by a single triangular atom (a direct copy),
but may also be formed by two edges from
different atoms.
Hence, as in \refE{EtP2},
\begin{equation*}
  \begin{split}
 \ts(P_2,\kf) &= \frac{1}{2} \int (\sigma_{P_2}(x,y,z)+\sigma_{K_2}(x,y)\sigma_{K_2}(y,z))\dd\mu(x)\dd\mu(y)\dd\mu(z) \\
 &= \frac{1}{2} \left(6B\gb^3 + \int (2A+6B\gb)^2x^{-1/\alpha}y^{-2/\alpha}z^{-1/\alpha} \dd\mu(x)\dd\mu(y)\dd\mu(z)\right) \\
 &= 3B\gb^3 + \frac{(2A+6B\gb)^2\gb^2}{2}\gbb.
  \end{split}
\end{equation*} 
In particular, if $\alpha\le 2$ then  $\ts(P_2,\kf)$ is infinite.

For $S_3=K_{1,3}$, the star with three edges, there are two types of tree-decompositions: three edges or one
edge and one copy of $P_2$, the latter occurring in 3 different ways.
(There are no direct copies.) 
Hence,
\begin{equation*}
  \begin{split}
 \traw(S_3,\kf)& = \int \sigma_{K_2}(x_1,x_2)\sigma_{K_2}(x_1,x_3)\sigma_{K_2}(x_1,x_4)
 + 3\int \sigma_{P_2}(x_2,x_1,x_3)\sigma_{K_2}(x_1,x_4)	
\\&
= (2A+6B\gb)^3\gb^3\gbbb + 18B(2A+6B\gb)\gb^3\gbb
  \end{split}
\end{equation*}
and thus
\begin{equation*}
\ts(S_3,\kf)
= \frac16{(2A+6B\gb)^3} \gb^3\gbbb + 3B(2A+6B\gb)\gb^3\gbb.
\end{equation*}

Finally, for $P_3$, there are again two types of tree-decompositions: three edges or one
edge and one copy of $P_2$, the latter now occurring in 2 different ways.
Hence,
\begin{equation*}
  \begin{split}
 \ts(P_3,\kf)& =\frac12 \int \sigma_{K_2}(x_1,x_2)\sigma_{K_2}(x_2,x_3)\sigma_{K_2}(x_3,x_4)
 + \frac22\int \sigma_{P_2}(x_1,x_2,x_3)\sigma_{K_2}(x_3,x_4)	
\\&
= \frac12(2A+6B\gb)^3\gb^2\gbb^2 + 6B(2A+6B\gb)\gb^3\gbb.
  \end{split}
\end{equation*}

As we shall see, the counts above are enough to calculate two more interesting
parameters of the graph $G_n=G(n,\kf)$.

\subsection{The clustering coefficient}

The {\em clustering coefficient} $C(G)$ of a graph $G$ was introduced
by Watts and Strogatz~\cite{WS} as a measure of the extent to which
neighbours of a random vertex in $G$ tend to be joined directly to each other.
After the degree distribution, it is one of the most studied parameters of real-world networks.
As discussed in~\cite{BRsurv}, for example,
there are several different definitions of such clustering coefficients.
One of these turns out to be
most convenient for mathematical analysis, and is also very natural;
following~\cite{BRsurv}, we call this coefficient $C_2(G)$.
(Hopefully there will be no confusion
with our earlier use of $C_2(G)$ for the number of vertices in the 2nd largest component.)
The coefficient $C_2(G)$ may be defined as a certain weighted average
of the `local clustering coefficients' at individual vertices, but is also simply given by
\begin{equation}\label{clusteringcoeff}
 C_2(G) = \frac{ 3 n(K_3,G)}{n(P_2,G)},  
\end{equation}
a ratio that is easily seen to lie between $0$ and $1$.

Now from above we have $\ts(K_3,\kf) = B\gb^3<\infty$. Hence,
by Theorem~\ref{th_reg},
\[
 \nr(K_3,G_n) = B \gb^3 n +\op(n),
\]
where, as usual, $G_n=G(n,\kf)$.
We shall return to exceptional copies of $K_3$ shortly.

If $\alpha\le 2$ then  $\ts(P_2,\kf)$ is infinite,
and $G_n$ will whp contain more than $O(n)$ copies
of $P_2$. Note that this is to be expected given the exponent of the 
asymptotic degree
distribution, since in this case the expected square degree is infinite.

From now on we suppose that $\alpha>2$, so $\ts(P_2,\kf)$ is finite.
Suppose for the moment that exceptional copies of $P_2$ and $K_3$ are negligible, i.e.,
that
\begin{equation}\label{neg}
 \nx(P_2,G_n),\,  \nx(K_3,G_n) = \op(n).
\end{equation}
By Theorem~\ref{th_reg}, we have $\nr(P_2,G_n)/n=\ts(P_2,\kf)+\op(1)$ and
$\nr(K_3,G_n)/n=\ts(K_3,\kf)+\op(1)$, so it follows that
\[
 C_2(G_n) = \frac{3\ts(K_3,\kf)}{\ts(P_2,\kf)}+\op(1) = c_2(A,B,\alpha)+\op(1)
\]
where, from the formulae in \refSS{ss_counts},
\begin{equation}\label{cAB}
 c_2(A,B,\alpha) = \frac{3B\gb^3}{3B\gb^3+2(A+3B\gb)^2\gb^2\gbb},
\end{equation}
with $\gb,\gbb$ given by \eqref{gbk}.
It follows that with the degree exponent $\alpha>2$ fixed,
this special case of our model can achieve any possible value of the clustering coefficient,
with the trivial exception of $1$ (achieved only by graphs that are vertex disjoint unions of cliques).
Indeed, $c_2(A,0,\alpha)=0$ for any $A$, while taking $A=0$ we have
\[
 c_2(0,B,\alpha) = \frac{1}{1+6B\gb\gbb},
\]
which is decreasing as a function of $B$,
and tends to $1$ as $B\to0$ and to $0$ as $B\to\infty$.

Let us note in passing that by Theorem~\ref{th_reg}, if $2<\alpha\le 4$ then
$\nr(P_2,G_n)/n$ is concentrated around its finite mean even though its
variance, which involves the expected 4th power of the degree
of a random vertex, tends to infinity.

So far we considered only regular copies of $P_2$ and $K_3$;
we now turn our attention to exceptional copies. Unfortunately, for any $\alpha$, some
moment of our kernel is infinite, so Theorem~\ref{th_mf} does not apply. However,
it is easy to describe the set of moments relevant to the calculation of $\E \nx(F,G_n)$
for the graphs $F$ we consider.

Suppose that $F$ is an exceptional triangle (or $P_2$; the argument is then almost identical)
in $G_n=G(n,\kf)$. Since $F$ has (at most) three edges, there are at most 3 atoms
$F_i$ contributing edges to $F$. Let $H$ be the union of these atoms,
%was $F_i'$
considered as a multigraph. For example, if $F$ is the triangle $abc$, then $H$ might
consist of the union of the three triangles $abd$, $bcd$, and $cad$. In some sense this
will turn out to be the `worst' case.

Let us fix the isomorphism type of $H$, defined in the obvious way. Let
$h$ be the total number of vertices in $H$, and write $r=\sum_i (|F_i|-1) -(h-1)$
for the `redundancy' of $H$. Since $F$ is exceptional, $r\ge 1$.
The expected number of exceptional $F$ arising in this way is exactly
$n\fall{h}n^{-\sum_i (|F_i|-1)}$ times a certain integral of products of $\ka_2$ and
$\ka_3$. From the form of $\ka_2$ and $\ka_3$, we may write this as
\[
 \frac{n\fall{h}}{n^{r+h-1}} \int_{\sss^h} x_1^{-n_1/\alpha}\cdots x_h^{-n_h/\alpha}\dd\mu(x_1)\cdots \dd\mu(x_h),
\]
where $n_i$ is the number of the atoms $F_j$ that contain the $i$th vertex of $H$.
The initial factor is at most $n^{1-r}\le 1$, while the integral is finite unless
$n_i\ge \alpha$ for some $i$. Since $H$ is made up of at most 3 atoms $F_j$,
we always have
$n_i\le 3$, so if $\alpha>3$ then the relevant integrals (i.e., the relevant moments)
are finite, and we have $\E\nx(K_3,G_n), \E\nx(P_2,G_n)=O(1)$, which certainly
implies \eqref{neg}.

In fact, we do not need to assume that $\alpha>3$. Suppose that $2<\alpha\le 3$.
Then in the multigraph version of the model,
$\E\nx(K_3,G_n)=\infty$. (Consider, for example, three triangles sitting on 4 vertices as above.)
On the other hand, this does not mean that $\nx(K_3,G_n)$ is often large.
Indeed, when we choose our vertex types uniformly from $(0,1]$, whp there is no vertex
whose type $x$ is at most $\delta=1/(n\log n)$, say.
Conditioning on this very likely event $\cA$, we may consider the restrictions
of $\ka_2$ and $\ka_3$ to $(\delta,1]^2$ and $(\delta,1]^3$, respectively.
Now the expected number of copies of some pattern $H$ is at most a constant times
\[
 n^{1-r} \left(\int_{\delta}^1 x^{-3/\alpha} \dd x\right)^s,
\]
where $s$ is the number of vertices $i$ of $H$ with $n_i=3$.
Since
the graph $K_3$ (or $P_2$) we are trying to form has maximum degree 2, 
every vertex of $H$ with $n_i=3$ corresponds to a redundancy, so we always
have $r\ge s$.
Up to constants and a power of $\log n$ the integral is $n^{(3-\alpha)/\alpha}\le \sqrt{n}$,
and it follows that
\[
 \E(\nx(K_3,G_n)\mid \cA) = O(\sqrt{n}) = o(n).
\]
Since $\Pr(\cA)=1-o(1)$, it follows that $\nx(K_3,G_n)=\op(n)$, even though
its expectation would not suggest this. The same holds for $\nx(P_2,G_n)$,
so we see that \eqref{neg} does indeed hold for any $\alpha>2$,
and the clustering coefficient is indeed concentrated about $c_2(A,B,\alpha)$.

\subsection{The mixing coefficient}\label{ss_mix}
Another interesting parameter of real networks is
the extent to which the degrees of the two ends of a randomly chosen edge
tend to correlate; positive correlation is known as {\em assortative mixing},
and negative correlation as {\em disassortative mixing}.
To define this precisely, let $G$ be any graph, and let $vw$ be an edge of $G$ chosen
uniformly at random. More precisely, let $(v,w)$ be chosen uniformly
at random from all $2e(G)$ ordered pairs corresponding to edges of $G$.
Let $D_v$ and $D_w$ denote the degrees of $v$ and $w$; we view these as random variables.
Since the events $\{v=v_1, w=v_2\}$ and $\{v=v_2, w=v_1\}$ have the same
probability, the random vertices $v$ and $w$ have the same distribution,
so $D_v$ and $D_w$ have the same distribution.

Let
\begin{equation}\label{mixingcoeff}
 a(G) = \frac{ \Cov(D_v,D_w) }{ \sqrt{ \Var(D_v)\Var(D_w) } } = \frac{\Cov(D_v,D_w)}{\Var(D_v)}.
\end{equation}
Here $G$ is fixed, and all expectations are with respect to the random choice
of $(v,w)$. Thus $a(G)$ is simply the correlation
coefficient between the degrees of the two ends of  a randomly chosen edge,
so $-1\le a(G)\le 1$, and $a(G)>0$ corresponds to assortative mixing and $a(G)<0$ to disassortative
mixing.
This mixing coefficient was introduced by Callaway, Hopcroft, Kleinberg, Newman and Strogatz~\cite{CHKNS},
building on work of Krapivsky and Redner~\cite{KR}, and has been studied by many people,
for example Newman~\cite{Newman}.
In~\cite{CHKNS}, $a(G)$ is denoted $\rho(G)$; we avoid this notation as it clashes
with our notation for the survival probability of a branching process.

Fortunately, we need no new theory to evaluate $a(G)$ for $G=G(n,\kf)$, since $a(G)$ can
be expressed in terms of small subgraph counts.
More precisely, for any graph $G$,
\[
 \E(D_v-1) = \frac{1}{2e(G)}\sum_i\sum_{j\sim i} ( d_i-1) = \frac{1}{2e(G)}\sum_i d_i(d_i-1)
 = \frac{n(P_2,G)}{e(G)},
\]
where $i$ runs over all vertices of $G$, then $j$ over all neighbours of $i$,
and $d_i$ is the degree of vertex $i$ in $G$.
Also,
\[
 \E\bigpar{(D_v-1)(D_w-1)} = \frac{1}{2e(G)}\sum_i \sum_{j\sim i} (d_i-1)(d_j-1) 
 = \frac{2n(P_3,G)+6n(K_3,G)}{2e(G)},
\]
so
\[
 \Cov(D_v,D_w)=\Cov(D_v-1,D_w-1) = \frac{\bigpar{n(P_3,G)+3n(K_3,G)}e(G)-n(P_2,G)^2}{e(G)^2}.
\]
Also,
\[
 2e(G) \E\bigpar{(D_v-1)(D_v-2)} = \sum_i \sum_{j\sim i} (d_i-1)(d_i-2) = \sum_i d_i(d_i-1)(d_i-2) = 6n(S_3,G),
\]
where $S_3=K_{1,3}$ is the star with $3$ edges. 
Thus
\begin{align*}
 \Var(D_v) = \Var(D_v-1) &= \E\bigpar{(D_v-1)(D_v-2)}+\E(D_v-1)-(\E(D_v-1))^2 \\
 &= \frac{3n(S_3,G)e(G)+n(P_2,G)e(G)-n(P_2,G)^2}{e(G)^2}.
\end{align*}
Hence
\begin{equation}  \label{aG}
 a(G) = 
\frac{\bigpar{n(P_3,G)+3n(K_3,G)}e(G)-n(P_2,G)^2}
{3n(S_3,G)e(G)+n(P_2,G)e(G)-n(P_2,G)^2}.
\end{equation}
In well-behaved cases, for example for bounded kernel families, it
follows from our results here (Theorems \ref{th_ssbd}--\ref{th_mf})
that if $G_n=G(n,\kf)$, then
\begin{equation}\label{aGn}
 a(G_n) = a(\kf)+\op(1),
\end{equation}
where
\begin{equation}\label{aka}
 a(\kf) = \frac{\ts(P_3,\kf)\xiek+3\ts(K_3,\kf)\xiek-\ts(P_2,\kf)^2}{3\ts(S_3,\kf)\xiek+\ts(P_2,\kf)\xiek-\ts(P_2,\kf)^2},
\end{equation}
with $\xiek=\ts(K_2,\kf)$; see \eqref{etK2}.

Returning to our present specific example,
substituting in the expressions for $\ts(\cdot,\kf)$ in \refSS{ss_counts},
the ratio \eqref{aka} turns out to be
\begin{equation}\label{akaex}
 a(\kf) = \frac{3AB\gb^5}{\bigpar{4 \xb^4 (\gb\gbbb-\gbb^2) +2(A+6B\gb)\xb^2\gbb+3AB\gb}\gb^4},
\end{equation}
where $\xb=(A+3B\gb)=\xie(\kf)/\gb^2$.
Let us make a few comments on these expressions.

Firstly,
the coefficients $\ts(K_2,\kf)$,
$\ts(P_2,\kf)$, $\ts(K_3,\kf)$ and $\ts(P_3,\kf)$ are finite for all
$\alpha>2$, while
$\ts(S_3,\kf)$ is finite if and only if $\alpha>3$.
For the numerator, 
one can argue for $P_3$ as for $P_2$ and $K_3$ above to show that the number of
exceptional copies of $P_3$ is $\opn$ and thus negligible for every
$\ga>2$, and hence 
$n(P_3,G_n)=\ts(P_3,\kf)n+\opn$ by \refT{th_reg}. Consequently, the numerator in
\eqref{aG} (with $G=G_n$) divided by $n^2$ converges in probability to
the numerator in \eqref{aka}, and this limit is finite.
For $\alpha> 3$, one
can argue in the same way to show that the number of exceptional
copies of $S_3$ is negligible, so $n(S_3,G_n)=\ts(S_3,\kf)n+\opn$ and
\eqref{aGn} does indeed hold.
For $\alpha\le 3$, when $\ts(S_3,\kf)=\infty$, 
Theorem~\ref{th_reg} implies that $n(S_3,\kf)/n\pto\infty$, 
so in this case $a(G_n)\pto 0=a(\kf)$, for the not very interesting reason
that $\Var(D_v)$ is unbounded while $\Cov(D_v,D_w)$ is not. 
In any case, we have shown that \eqref{aGn} holds in our example for every
$\ga>2$.

Secondly, we see that $0\le a(\kf)<\infty $ for every $\ga>2$, with
$a(\kf)>0$ whenever $\ga>3$ and we add both edges and triangles (i.e.,
if both $A$ and $B$ are non-zero).

Thirdly, if $A$ and $B$ are both positive and comparable
but very small, then it is
easy to see 
that $a(\kf)$ is close to $1$, for the simple reason
that the graph 
then consists of rather few (though still order $n$) edges and
triangles, which are 
almost all vertex disjoint. In this case we almost always have either
$D_v=D_w=1$, 
if we pick an edge component, or $D_v=D_w=2$ if we pick an edge of a
triangle. 
This is also easily checked algebraically from \eqref{akaex}: the
denominator is of the form $3AB\gb^5+O((A+B)^3)$, which is asymptotically
equal to the numerator if $A,B\to0$ with $A/B$ bounded above and below.
It follows that as $A$ and $B$ are varied, $a(\kf)$ can take any
value between $0$ and $1$, with $1$ excluded.

Finally, it is easy to check that the form of $a(\kf)$ as a function of $A$, $B$ and $\alpha$
is very different from that of $c_2(A,B,\alpha)$ given in \eqref{cAB}.
It follows that with the degree exponent $\alpha>3$ fixed, if we vary $A$ and $B$
we may vary the clustering coefficient and $a(\kf)$ independently, subject to certain inequalities.

It so happens that in the example considered here, $a(\kf)$ is always non-negative, but it
is easy to give examples where $a(\kf)<0$. Indeed, this arises already in the
edge-only case (of the kind we treated in \cite{BJR}), even
with the very simple type space with two elements of weights
$\mu\set{1}=p$  and $\mu\set{2}=q=1-p$, $0<p<1$, taking
$\ka_2(1,1)=0$, $\ka_2(2,2)=0$ and $\ka_2(1,2)=A>0$. In symbols,
\begin{equation*}
  \ka_2(x,y)=A\ett{x\neq y},
\end{equation*}
where $\ett{\mathcal E}$ is the indicator function of the event $\mathcal E$.

For this kernel (family)
\begin{align*}
  \kae(x,y)&=2\ka_2(x,y)=2A\ett{x\neq y},
\\
\xiek&=\int\ka_2=2Apq,
\\
\sigma_{K_2}(x,y)&=2\ka_2(x,y)=\kae(x,y)=2A\ett{x\neq y}.
\end{align*}
Expanding the integrals as sums, it follows that
\begin{align*}
\ts(K_2,\kf) &= \frac{1}{2} \int \sigma_{K_2}(x,y)\dd\mu(x)\dd\mu(y) 
= \xiek=2Apq;
\\
 \ts(P_2,\kf) &= \frac{1}{2} \int
 \sigma_{K_2}(x,y)\sigma_{K_2}(y,z)\dd\mu(x)\dd\mu(y)\dd\mu(z) 
= 2A^2(pqp+qpq)
\\&
=2A^2pq;
\\
  \ts(K_3,\kf)&=0
\\
 \ts(S_3,\kf)& =\frac16 \int 
 \sigma_{K_2}(x_1,x_2)\sigma_{K_2}(x_1,x_3)\sigma_{K_2}(x_1,x_4)
\dd\mu(x_1)\dd\mu(x_2)\dd\mu(x_3)\dd\mu(x_4) 
\\&
=\frac43 A^3(pq^3+qp^3)
=\frac43 A^3pq(p^2+q^2);
\\
 \ts(P_3,\kf)& 
=\frac12 \int
\sigma_{K_2}(x_1,x_2)\sigma_{K_2}(x_2,x_3)\sigma_{K_2}(x_3,x_4)
\dd\mu(x_1)\dd\mu(x_2)\dd\mu(x_3)\dd\mu(x_4) 
\\&
=4A^3(pqpq+qpqp)=8A^3p^2q^2.
\end{align*}

Substituting these expressions into \eqref{aka} and simplifying, we find that
\[
 a(\kf) =-\frac{A(p-q)^2}{A(p-q)^2+1}.	
\]
Hence $a(\kf)\le 0$, and we have disassortative mixing as soon as $p\neq q$, i.e.,
when $p\in(0,\frac12)\cup(\frac12,1)$. We see also that the
coefficient $a(\kf)$ can be made to take any value in $(-1,0]$ by
choosing the parameters suitably.

One can easily combine the simple example above with that considered 
in the bulk of this section to give
graphs with power-law degree distributions with various values of the clustering
coefficient and of $a(G_n)$, now with negative values of $a(G_n)$ possible.
Perhaps the simplest way of giving such graphs
is to divide the type space $(0,1]$ into two intervals
$I_1=(0,x_0]$ and $I_2=(x_0,1]$, take $\varphi(x)=x^{-1/\alpha}$ on $I_1$
and $\varphi(x)=(x-x_0)^{-1/\alpha}$ on $I_2$,
to set $\ka_2(x,y) = A_1\varphi(x)\varphi(y)$ if one of $x$ is in $I_1$ and the other in $I_2$,
and $\ka_2(x,y)=A_2\varphi(x)\varphi(y)$ otherwise,
and to define $\ka_3(x,y,z)$ to be some constant times $\varphi(x)\varphi(y)\varphi(z)$,
where the constant depends on how many of $x$, $y$ and $z$ lie in $I_1$.

\section{Limits of sparse random graphs}\label{sec_lim}

Although our main focus in this paper was the introduction of the model
$G(n,\kf)$, and the study of the existence and size of the giant component
in this graph,
we shall close by briefly discussing some connections to earlier
work that arise when considering
the local structure of $G(n,\kf)$.

Let us start by considering subgraph counts. As before, let $\G$
consist of one representative of each isomorphism class of finite
graphs, and let $\F\subset \G$ consist of the connected graphs in $\G$.
Given two graphs $F$ and $G$, let $\hom(F,G)$ be the number of
homomorphisms from $F$ to $G$, and $\emb(F,G)$ the number of embeddings,
so $\emb(F,G)=n(F,G)\aut(F)$.  Writing $G_n$ for a graph with $n$
vertices, in the dense case, where $G_n$ has $\Theta(n^2)$ edges, one
can combine the normalized subgraph or embedding counts
\[
 s(F,G_n)=n(F,G_n)/n(F,K_n)=\emb(F,G_n)/\emb(F,K_n)
\]
to define a metric that turns out to
have very nice properties. (Often one uses the equivalent homomorphism
densities $t(F,G_n)=\hom(F,G_n)/n^{|F|}$, but when we come to sparse
graphs embeddings are more natural than homomorphisms.)
A sequence $(G_n)$ converges in this {\em subgraph metric} if and only if there
are constants $s(F)$, $F\in \F$, such that $s(F,G_n)\to s(F)$ for each $F\in \F$.
Lov\'asz and Szegedy~\cite{LSz1} characterised the possible limits
$(s(F))_{F\in \F}$, both in terms of kernels and algebraically.

Borgs, Chayes, Lov\'asz, S\'os and Vesztergombi~\cite{BCLSV:1,BCLSV:2}
introduced the cut metric $\dcut$ that we used in \refS{sec_giant}.
They showed that
this metric is equivalent to the subgraph metric, as well as to various
other notions of convergence for sequences of dense graphs.
One of the nicest features of these results is that for every point
in the completion of the space of finite graphs (with respect
to any of these metrics), there is a natural random graph model
(called a $W$-random graph in~\cite{LSz1}) that produces
sequences of graphs tending to this point.
(See also Diaconis and Janson~\cite{SJ209}, where connections to certain
infinite random graphs are described.)

Turning to sparse graphs, as described in~\cite{BRsparse,BRsparse2}, the situation is much
less simple.
When $G_n$ has $\Theta(n)$ edges, as here,
the natural normalization is to consider, for each connected $F$,
\[
 {\tilde s}(F,G_n) = \emb(F,G_n)/n = \aut(F) n(F,G_n)/n.
\]
Under suitable additional assumptions on the sequences $G_n$, one can
again combine these counts to define a metric, and consider
the possible limit points. Unfortunately, not much is known about these;
see the discussion in~\cite{BRsparse2}.

Turning to our present model, \refT{th_mf} shows that if $\kf$
is a kernel family with only finitely many non-zero kernels and
all moments finite, then ${\tilde s}(F,G_n)\pto \traw(F,\kf)$ for all
connected $F$, where $G_n=G(n,\kf)$ and $\traw(F,\kf)$
is given by \eqref{t1fdef}. This suggests the following question.

\begin{question}\label{q1}
Is there a simple characterization of those vectors $(t_F)_{F\in\F}$
for which there is an integrable kernel family $\kf$ such that $t_F=\traw(F,\kf)$
for all $F\in \F$?
\end{question}

As unbounded kernel families may cause technical difficulties, it may
make sense to ask the same question with the restriction that $\kf$
should be bounded.

Note that \refQ{q1} is very different from the question answered by
Lov\'asz and Szegedy~\cite{LSz1}: our definition of $\traw(F,\kf)$
is different from the corresponding notion studied there,
since it is adapted to the setting of sparse graphs. In particular,
if $\kf$ consists only of a single kernel $\ka_2$ (as in \cite{LSz1}),
then we have $\traw(F,\kf)=0$
for any $F$ that is not a tree.

As discussed in~\cite[Question 8.1]{BRsparse2}, it is an interesting question to ask
whether, for various natural metrics on sparse graphs, one can provide
natural random graph models corresponding to points in the completion.
For those vectors $(t_F)$ where the answer to \refQ{q1} is yes,
the model $G(n,\kf)$ provides an affirmative answer (at least
if $\kf$ is bounded, say). But these points will presumably
only be a very small subset of the possible limits,
so there are many corresponding models
still to be found.

As noted in~\cite[Sections 3,7]{BRsparse2}, rather than
considering subgraph counts ${\tilde s}(F,G_n)$, for graphs with $\Theta(n)$
edges it is more natural to consider directly the probability
that the $t$-neighbourhood of a random vertex $v$ is a certain graph $F$;
the subgraph counts may be viewed as moments of these probabilities.

More precisely, let $\Gr$ be the set of isomorphism classes
of connected, locally finite rooted graphs, and for $t\ge 0$, let $\Grt$ be
the set of isomorphism classes of
finite connected rooted graphs
with {\em radius} at most $t$, i.e., in which all vertices
are within distance $t$ of the root.
A probability distribution $\pi$ on $\Gr$ naturally induces a probability
distribution $\pi_t$ on each $\Grt$, obtained by taking a $\pi$-random element
of $\Gr$ and deleting any vertices at distance more than $t$ from the root.
Given $F\in \Grt$ and a graph $G_n$ with $n$ vertices,
let $p_t(F,G_n)$ be the probability that a random vertex $v$ of $G_n$
has the property that its neighbourhoods up to distance $t$
form a graph isomorphic to $F$, with $v$ as the root.
A sequence $(G_n)$ 
with $|G_n|\to\infty$ 
has {\em local limit} $\pi$ if
\begin{equation*}
 p_t(F,G_n) \to \pi_t(F)
\end{equation*}
for every $F\in \Grt$ and all $t\ge 0$.
This notion has been introduced in several different contexts under
different names: Aldous and Steele~\cite{AS} used the term `local
weak limit', and Aldous and Lyons~\cite{AL} the name `random weak limit'.
Also, Benjamini and Schramm~\cite{BSrec} defined a corresponding
`distributional limit' of certain random graphs.
Notationally it is convenient to map a graph $G_n$ to the
point $\phi(G_n)=(p_t(F,G_n))\in X=\prod_t [0,1]^{\Grt}$, and to define
$\phi(\pi)$ similarly. Taking any metric $d$ on $X$ giving rise to the product topology,
we obtain a metric $\dloc$ on the set of graphs together with probability distributions
on $\Gr$, and $(G_n)$
has local limit $\pi$ if and only if $\dloc(G_n,\pi)\to 0$.

As noted in~\cite{BRsparse2}, under suitable assumptions (which will
hold here if $\kf$ is bounded, for example), the two notions of convergence
described above are equivalent, and one can pass from the limiting normalized
subgraph counts ${\tilde s}(F)$ to the distribution $\pi$ and {\em vice versa}.
Also, if $\ka$ is a bounded kernel, then the random graphs $G(n,\ka)$
defined in~\cite{BJR} have as local limit a certain distribution associated
to $\pi$. This latter observation extends to the present model, and
as we shall now see, no boundedness restriction is needed.

Given an integrable hyperkernel $\kf$, let $G_\kf$ be the random (potentially
infinite) rooted graph associated to the branching process $\bpk$.
This is defined in the natural way: we take the root of $\bpk$ as the root vertex,
for each child clique of the root we take a complete graph in $G_\kf$,
with these cliques sharing only the root vertex. Each child $w$ of the root
then corresponds to a non-root vertex in one of these cliques, and
we add further cliques meeting only in $w$ to correspond to the child cliques
of $w$, and so on.

More generally, given an integrable kernel family $\kf=(\ka_F)_{F\in\F}$,
we may define a random rooted graph $G_\kf$ in an analogous way; we omit the details.
We write $\pi_\kf$ for the probability distribution on $\Gr$
associated to $G_\kf$.

\begin{theorem}\label{lwl}
Let $\kf$ be an integrable kernel family and let $G_n=G(n,\kf)$.
Then $\dloc(G_n,\pi_\kf)\pto 0$.
\end{theorem}

The proof of this result, which may be seen as a much stronger form
of \refL{nkint}, will take a little preparation.

In fact, we conjecture that almost sure convergence holds
for any coupling of the $G_n$ for different $n$, and in particular
if the different $G_n$ are taken to be independent.
(The case of independent $G_n$ is the extreme case, which by
standard arguments implies a.s.\ convergence for every other coupling
too; a.s.\ convergence in this case is known as \emph{complete convergence}.)

Writing $\pi_{\kf,t}$ for the probability distribution on $\Grt$
induced by $\pi_\kf$,
by definition we have $\dloc(G_n,\pi_\kf)\pto 0$ if and only if
\begin{equation}\label{loceq}
 p_t(F,G_n)\pto \pi_{\kf,t}(F)
\end{equation}
for each $t$ and each $F\in \Grt$.
The special case where $\kf$ is a bounded hyperkernel is essentially 
immediate:
\eqref{loceq} is simply a formal statement of the local coupling established
for bounded hyperkernels in~\refS{sec_loc}. Exactly the same argument applies to a bounded
kernel family. For the extension to general kernel families we need a couple
of easy lemmas.

\begin{lemma}\label{Le1}
Let $\kf$ be an edge-integrable kernel family.
For any $\eps>0$ there is a $\delta=\delta_1(\kf,\eps)>0$
such that whp any $\delta n$ vertices of $G(n,\kf)$ meet at most $\eps n$ edges.
\end{lemma}
\begin{proof}
This is an extension of Proposition 8.11 of~\cite{BJR}; the proof carries
over {\em mutatis mutandis}, using \refT{th_ssbd} with $F=P_2$ to bound the
sum of the squares of the vertex degrees in the bounded case.
The key step is to use
edge integrability to find a bounded kernel family $\kf'$ such
that $G(n,\kf')$ may be regarded as a subgraph of $G(n,\kf)$
containing all but at most $\eps n/2+\op(n)$ of the edges.
\end{proof}

It turns out that we can weaken edge integrability to integrability.
The price we pay is that we cannot control the number
of edges incident to a small set of vertices, but only the size
of the neighbourhood. As usual, given a set $A$ of vertices in a graph
$G$, we write $N^t(A)$ for the set of vertices at graph distance at most
$t$ from $A$, so $A\subset N(A)=N^1(A)\subset N^2(A)\cdots$.

\begin{lemma}\label{Le2}
Let $\kf$ be an integrable kernel family.
For any $\eps>0$ there is a $\delta=\delta_2(\kf,\eps)>0$
such that whp every set $A$ of at most $\delta n$ vertices of $G(n,\kf)$
satisfies $|N(A)|\le \eps n$.
\end{lemma}
\begin{proof}
Replacing each atom by a clique, we may and shall assume
that $\kf$ is a hyperkernel.
Let $\kf'$ be the kernel family obtained from $\kf$ by replacing
each clique by a star.
Since $\kf$ is integrable, $\kf'$ is edge integrable.
Let $\delta_1(\eps)=\delta_1(\kf',\eps)$ be the function given
by \refL{Le1}, and set $\delta=\delta_1(\delta_1(\eps))>0$.
Then whp every set $A$ of at most $\delta n$ vertices of $G(n,\kf')$
has $|N(A)|\le \delta_1(\eps) n$ and hence $|N^2(A)|\le \eps n$.
Coupling $G(n,\kf)$ and $G(n,\kf')$ in the obvious way,
vertices adjacent in $G(n,\kf)$ are at distance at most $2$
in $G(n,\kf')$, and the result follows.
\end{proof}

Let $v(G(n,\kf))$ be the sum of the sizes (numbers of vertices)
of the atoms making up $G(n,\kf)$.
Our final lemma relates this sum  
to
$\ii{\kf} = \sum_F |F|\int_{\sss^{|F|}} \ka_F$.

\begin{lemma}\label{badv}
Let $\kf=(\ka_F)_{F\in\F}$ be an integrable kernel
family. 
Then $v(G(n,\kf))=n\ii{\kf} +\op(n)$.
\end{lemma}
\begin{proof}
Let $X_r$ be the number of atoms with $r$ vertices,
and $X=v(G(n,\kf))=\sum_{r\ge 2} r X_r$.
Let $c=\ii{\kf}$, and let
$c_r$ be the contribution to $\ii{\kf}$ from kernels
corresponding to graphs $F$ with $r$ vertices,
so $\E(r X_r)=(n)_r c_r/n^{r-1}\sim nc_r$ and $\sum c_r=c<\infty$.
Given $\eps>0$, there is an $R$ such that $\sum_{r\le R} c_r\ge c-\eps$.
Each $X_r$ has a Poisson distribution and is thus concentrated about
its mean, so whp
\[
 X \ge \sum_{r\le R} rX_r 
\ge \sum_{r\le R} c_rn- \eps n
\ge (c-2\eps)n.
\]
Writing $(x)_+$ for $\max\{x,0\}$, since $\eps$ was arbitrary we have shown that
$(c-X/n)_+\pto 0$. Since $c-X/n$ is bounded, it follows that
$\E(c-X/n)_+\to 0$. But $\E X\le cn$, so
$\E(X/n-c)_+ = \E(X/n-c) + \E(c-X/n)_+ \to 0$. Hence $(X/n-c)_+\pto 0$,
so $X/n\pto c$ as claimed.
\end{proof}

Combining the last two lemmas, we can now prove \refT{lwl}.
\begin{proof}[Proof of \refT{lwl}]
As noted after the statement of the theorem, the case where $\kf$
is bounded is straightforward.

Let $\kf$ be an integrable kernel family, and let $G_n=G(n,\kf)$.
Fix $t\ge 1$, $F\in \Grt$, and $\eps>0$. It suffices to prove that
\begin{equation}\label{cc}
  | p_t(F,G_n) - \pi_{\kf,t}(F) | \le \eps +\op(1).
\end{equation}
Then letting $\eps\to 0$ we have $p_t(F,G_n)\pto \pi_{\kf,t}(F)$,
so \eqref{loceq} holds. Since $t$ and $F$ are arbitrary,
this implies $\dloc(G_n,\pi_\kf)\pto 0$.

Applying \refL{Le2} $t$ times, there is a $\delta>0$ such that
whp any set $A$ of at most $\delta n$ vertices of $G_n$
satisfies $|N^t(A)|\le \eps n/2$.
Since $\kf$ is integrable, there is a bounded kernel family $\kfM$
which satisfies $\kfM\le \kf$ pointwise and  $\ii\kf-\ii\kfM\le \delta/2$.
As $M\to\infty$, we have $\kfM\upto \kf$ pointwise,
and it follows that $\pi_{\kfM,t}(F)\to \pi_{\kf,t}(F)$; the argument is as for \refT{TappC}(i).
Taking $M$ large enough, we may thus assume that $\bigabs{\pi_{\kfM,t}(F)-\pi_{\kf,t}(F)}\le \eps/2$.
Let $G_n'=G(n,\kfM)$. Since $\kfM$ is bounded, we have $p_t(F,G_n')\pto \pi_{\kfM,t}(F)$.
Coupling $G_n$ and $G_n'$ as usual so that $G_n'\subset G_n$,
let $B$ be the set of vertices incident with an atom present
in $G_n$ but not $G_n'$. By \refL{badv} we have $|B|\le \delta n$ whp,
so whp no more than $\eps n/2$ vertices are within distance $t$
of vertices in $B$.
But then $|p_t(F,G_n)-p_t(F,G_n')|\le \eps/2$ whp, and \eqref{cc} follows.
\end{proof}

The general question of which probability distributions on $\Gr$ arise
as local limits of sequences of finite graphs seems to be rather difficult.
There is a natural necessary condition noted in different
forms in all of~\cite{AL,AS,BSrec}; see also \cite[Section 7]{BRsparse2}.
Aldous and Lyons~\cite{AL} asked whether this condition is sufficient,
emphasizing the importance of this open question.
Let us finish with a related but perhaps much simpler question:
given $\kf$, we defined $\bpk$ as a branching process in which the particles
have types. But in the corresponding random graph  $G_\kf$ these types
are not recorded. This means that $\kf$ cannot simply be read out
of the distribution of $G_\kf$, i.e., out of $\pi_\kf$. This suggests the following question.

\begin{question}
Which probability distributions on $\Gr$ are of the form $\pi_\kf$ for some
integrable kernel family $\kf$?
\end{question}

\begin{ack}
Part of this research was done during visits of SJ
to the University of Cambridge and Trinity College in 2007,
and
to the Isaac Newton Institute in Cambridge,
 funded by a Microsoft fellowship,
 %and Churchill College 
in 2008.
\end{ack}

\newcommand\vol{\textbf}
\newcommand\jour{\emph}
\newcommand\book{\emph}
\newcommand\inbook{\emph}
\newcommand\toappear{\unskip, to appear}

\newcommand\webcite[1]{%\hfil  %???
   %\penalty0 %???
\texttt{\def~{{\tiny$\sim$}}#1}\hfill\hfill}
\newcommand\webcitesvante{\webcite{http://www.math.uu.se/~svante/papers/}}
\newcommand\arxiv[1]{\webcite{arXiv:#1.}}

\def\nobibitem#1\par{}

\end{document}